\documentclass[12pt]{article}
\usepackage{graphicx}
\usepackage{amsmath,amssymb,amsthm,esint,dsfont,tikz,float,soul}
\usepackage[english]{babel}
\usepackage[margin=1in]{geometry}
\usepackage{csquotes}
\usepackage{setspace}
\usepackage{pgfplots}
\usepackage[colorlinks, pdfpagelabels, pdfstartview = FitH, bookmarksopen=true, bookmarksnumbered = true, linkcolor = black,plainpages = false, hypertexnames = false, citecolor = black]{hyperref}
\usepackage{mathrsfs}
\usepackage{halloweenmath} 

\usetikzlibrary{decorations.pathreplacing,calligraphy,patterns}

\usepackage{color}

\newcommand{\blue}[1]{#1}

\newcommand{\answ}[1]{#1}

\newcommand{\dx}[1]{\:d#1}
\newcommand{\M}{\mathcal{M}}

\newcommand{\N}{\mathcal{N}}
\newcommand{\sym}{\mathcal{S}_0}

\newcommand{\Om}{\Omega}
\renewcommand{\k}{\kappa}
\newcommand{\cEt}{\mathcal{E}_{\xi,\eta}}
\renewcommand{\S}{\mathbb{S}}
\newcommand{\roe}{{r_\omega^\e}}
\newcommand{\bigO}{\mathcal{O}}
\newcommand{\ftf}{\sqrt[4]{24}}
\newcommand{\Qc}{\blue{Q_{\xi,\eta}^{H}}}
\newcommand{\rt}{\tilde{r}}

\newcommand{\B}{\mathcal{B}}
\newcommand{\cR}{\mathcal{R}}

\newcommand{\sgn}{\mathrm{sgn}}
\renewcommand{\P}{\mathcal{P}}
\newcommand{\exe}{E_{\xi,\eta}}
\newcommand{\ex}{E_\xi}
\renewcommand{\H}{\mathcal{H}}
\newcommand{\Qm}{Q_{\xi,\eta}}
\newcommand{\Qt}{\Qm}

\newcommand{\vd}{v^\d}
\renewcommand{\t}{\tau}
\newcommand{\vb}{\overline{v}}
\newcommand{\norm}[1]{\left|#1\right|}
\newcommand{\nnorm}[1]{\blue{|}#1\blue{|}}
\newcommand{\ph}{\widehat{\Phi}}
\renewcommand{\L}{\mathcal{L}}
\newcommand{\A}{\mathcal{A}}

\renewcommand{\d}{\delta}
\newcommand{\e}{\varepsilon}
\newcommand{\ee}{\mathbf{e}}
\newcommand{\MM}{\mathbb{M}}
\newcommand{\FF}{\mathbb{F}}
\newcommand{\nuh}{{\widehat{\nu}}}
\newcommand{\omegah}{{\widehat{\omega}}}
\newcommand{\Dd}{\mathcal{R}^\d}
\newcommand{\Qcd}{Q_{\xi,\eta}^{H,\d}}
\newcommand{\Qcdk}{Q_{\xi,\eta}^{H,\d_k}}
\newcommand{\fP}{\mathscr{P}}

\newcommand{\tr}{\mathrm{tr}}
\newcommand{\grad}{\nabla}
\newcommand{\dist}{\mathrm{dist}}
\newcommand{\C}{\mathcal{C}}
\newcommand{\R}{\mathbb{R}}

\newcommand{\restr}{\mathbin{\vrule height 1.6ex depth 0pt width 0.13ex\vrule height 0.13ex depth 0pt width 1.3ex}}

\newtheorem{theorem}{Theorem}[section]
\newtheorem{proposition}[theorem]{Proposition}
\newtheorem{lemma}[theorem]{Lemma}
\newtheorem{corollary}[theorem]{Corollary}
\theoremstyle{definition}

\title{Asymptotics for Minimizers of Landau-de Gennes with \blue{a} Magnetic Field and Tangential Anchoring}

\date{\today}

\usepackage[capitalise]{cleveref}
\usepackage{authblk}

\author[1]{Lia Bronsard\footnote{bronsard@mcmaster.ca }}
\author[2]{Dean Louizos\footnote{louiz001@umn.edu}}
\author[3]{Dominik Stantejsky\footnote{dominik.stantejsky@univ-lorraine.fr}}
\affil[1]{Department of Mathematics and Statistics, McMaster University, Hamilton, ON L8S 4L8 Canada}
\affil[2]{School of Mathematics, University of Minnesota, Minneapolis, MN 55455 USA}
\affil[3]{Universit\'e de Lorraine, Institut \'Elie Cartan de Lorraine, UMR 7502 CNRS,  54506 Vand\oe uvre-l\`es-Nancy Cedex, France}

\begin{document}

\maketitle

\begin{abstract}
In this article we prove existence of minimizers of the Landau-de Gennes energy for liquid crystals with homogeneous external magnetic field and strong uniaxial planar anchoring.
Next we consider the asymptotics of solutions to the joint minimization of the energy w.r.t.\ the function and its boundary condition.
This constitutes a generalization to arbitrary regular particle shapes of the results obtained in \cite{bls} \answ{in a particular setting}.
Moreover, we show the absence of line singularities in some asymptotic parameter regime.
Finally we characterize the optimal orientation of particles vis-\`a-vis the magnetic field direction and compute it explicitly for different particle shapes.\\ \linebreak
\textbf{Keywords:} Existence of minimizers, singular limit, tangential anchoring, boojums, line defects, Landau-de Gennes energy. \hfill\hfill
\linebreak
\end{abstract}

\section{Introduction}

In this article, we continue the study started in \cite{bls} of the effect of an external magnetic field on the type and locations of defects near a colloidal particle immersed in nematic liquid crystal in the case of tangential anchoring. 
In \cite{bls}, the case of a spherical colloid was addressed, and here we study the general case of colloids that are regular $C^{\answ{2},1}-$domains  with possibly non-trivial topology.
Colloid shapes that have been studied experimentally and numerically include
spherocylindrical
\cite{Hung,HGGAdP} and
ellipsoidal particles
\cite{TMMML},
as well as torii
\cite{Liu2013}
and cubes (and their approximations)
\cite{BGL,SPD}, and our study includes such shapes.

There are many novel aspects due to the general form of the colloid and of the anchoring conditions that must be addressed, and
even the existence of minimizers has to be established, due to the freedom allowed by the tangential anchoring condition. 
In \cite{weakanchor,colloid,colloidphys,ACS2021,fukudaetal04,fukudayokoyama06}, the study of the type and location of defects in the case of normal Dirichlet boundary conditions, when the liquid crystal molecules align perpendicularly to the surface of the particle, is well developed. 
Although  tangential anchoring is experimentally observed and has been investigated by \cite{Liu2013,Senyuk2021,Shi2005,Tas2012}, little is known mathematically in {the} case of tangential boundary conditions. 
The study of these tangential boundary conditions and the new analytical challenges that arise in this context, have been started recently in \cite{abc,ABv,bcs,bls}.

Finally, the effect of an external field has only been  mathematically  studied  in the case of normal Dirichlet boundary conditions. 
In the small particle limit and in the absence of an external magnetic field, the existence of a line singularity called a \emph{Saturn ring} defect is shown in \cite{colloid}. 
Further it is shown in \cite{abl, ACS2021} that when an external field is imposed,  a \emph{Saturn ring} defect can still occur even for \emph{large particles}. 
In this article, we are interested in studying the effect of a uniform magnetic field, imposed within the sample, on the type of defect expected in the case of \emph{tangential} anchoring for \emph{large particles}. 

We will be using the Landau-de Gennes model in which the average orientation of the nematic liquid crystal molecules is represented by a $Q$-tensor, that is an element of $\sym$, the set of traceless symmetric $3\times 3$ matrices, which accounts for the alignment of the molecules through their eigenvectors and eigenvalues.
\answ{The space $\sym$ is a Hilbert space endowed with the inner product and associated norm given by}
\begin{equation}\label{def:inner}
    \answ{\langle Q,P\rangle=\tr(QP)\quad\text{and}\quad \nnorm{Q}^2=\textstyle\sum_{i,j=1}^3Q_{ij}^2
    \, .}
\end{equation}
Another continuum model that is commonly used to describe liquid crystal is the Oseen-Frank model, in which the orientation is represented by a unit director $n\in\mathbb S^2$, and the $Q$-tensors model yields a relaxation of this uniaxial constraint. 
In fact, the Oseen-Frank model can be recovered from the Landau-de Gennes energy 
in certain parameter regimes (see \cite{Gartland})
and this convergence has recently produced new powerful mathematical analysis \cite{baumanparkphillips12,
canevari2d,
canevari3d,
singperturb,
golovatymontero14,
NguZa,
majumdarzarnescu10}.  

However, the $Q$-tensor model gives a much finer description  of the defect cores as it allows for more types of singularities such as half-degree line singularities as can be seen in the works of \cite{ABGL,ACS2021,ACS2024,BaZar,canevari3d,CO1,CO2,DiMiPi,GWZZ,majumdarzarnescu10}.
Part of our results is to show that in certain asymptotic regimes,  
those line singularities are not energetically favourable 
(see Theorem \ref{thm:line_defects_shrink_0}).

We will be using the simplified model from \cite{fukudaetal04,fukudayokoyama06} where the presence of the external field is modelled by adding a symmetry-breaking term to the energy (favouring alignment along the field), multiplied by a parameter accounting for the intensity of the field.
After adequate non-dimensionalization \cite{fukudaetal04,Gartland} we are left with two parameters $\xi,\eta>0$ which represent, in units of the particle radius, the coherence length for nematic alignment and magnetic coherence length. 
\answ{Since we are interested in the large particle limit, the nematic and magnetic coherence lengths are small compared to the particle radius, thus translating to the limit $\eta,\xi\to 0$.}
In these units, the colloidal particle is represented by a closed domain $\mathcal{D}\subset\R^3$, whose boundary surface is given by $\M$ so that the liquid crystal is contained in the domain  $\Omega=\R^3\setminus \mathcal{D}$.

Hence, the non-dimensionalized Landau-de Gennes energy functional that we study on a\answ{n open subset} $U\subset\Om$ is given by
\begin{equation}\label{def:energie}
    \exe(Q;U)=\int_U \frac{1}{2}|\grad Q|^2+\frac{1}{\xi^2}f(Q)+\frac{1}{\eta^2}g(Q)\, dx
    \, ,
\end{equation}
where
\begin{equation*}
    f(Q)=-\frac{1}{2}|Q|^2-\tr(Q^3)+\frac{3}{4}|Q|^4+\frac{2}{9}\quad\text{and}\quad g(Q)=\sqrt{\frac{2}{3}}-\frac{Q_{33}}{|Q|}
    \, {,}
\end{equation*}
with the convention that $g(0):=0$. With this definition, $g$ is lower semi-continuous on $\sym$. 
These two functions have the property that $f(Q)\geq 0$, with equality exactly when $Q\in\N$, the space of uniaxial $Q-$tensors, that is, $Q-$tensors of the form $n\otimes n-\frac13 I, n\in \mathbb{S}^2$, and $g(Q)\geq 0$ with equality if and only if $Q=sQ_\infty$ for $s\geq 0$ and $Q_\infty:= e_3\otimes e_3-\frac13 I$. 
\answ{Furthermore for $Q\in \N$ we define $n(Q)$ to denote the eigenvector of unit length associated to the leading eigenvalue of $Q$ and we remark that this choice is unique up to sign.
}

We note that the potential $\frac{1}{\xi^2}f(Q)+\frac{1}{\eta^2}g(Q)$ is minimized exactly at $Q=Q_{\infty}$ and that 
\begin{equation}\label{def:coer}
\frac{1}{\xi^2}f(Q)+\frac{1}{\eta^2}g(Q)\ge \frac{1}{\xi^2} C(\xi/\eta)\nnorm{Q-Q_{\infty}}^2
\, ,
\end{equation}
for some constant $C(\xi/\eta)>0$, see \cite{abl,ACS2021}. \answ{In addition, from \cite[Lemma 14]{canevari3d} we have the lower bound
\begin{equation}\label{bound:f}
    f(Q)\geq C_f\: \dist^2(Q,\N)
\end{equation}
for some $C_f>0$.}

We require that $\M$ be a $C^{\answ{2},1}$, \blue{connected, }compact, oriented $2-$manifold without boundary, and the unit normal of $\M$ pointing out of the particle is denoted by $\nu:\M\to\S^2$. 
We let the two principal curvatures of $\M$ be $\k_1,\k_2$ and subsequently define 
\begin{equation*}
    \k=\max_{\omega\in\M}\{|\k_1(\omega)|,|\k_2(\omega)|\}
    \, .
\end{equation*}
This gives the property that,
\begin{equation}\label{bound:curv}
    |\grad_\omega \nu(\omega)|\leq\sqrt{2}\k\quad\text{for all $\omega\in\M$}
    \, .
\end{equation}
We will frequently make use of the coordinate system $(r,\omega)$ to describe the region around $\M$, where for any $x$ sufficiently close to $\M$, we can write $x=\omega+r\nu(\omega)$ for some $\omega\in\M$ and $r>0$. Choosing $r_0$ \answ{sufficiently small so that this parametrization is globally injective and in particular such that 
\begin{equation}\label{def:r0}
    r_0\leq\frac{1}{2\k}
    \, ,
\end{equation}}
we are able to uniquely parameterize the surrounding region for $\omega\in\M$ and $r\in(0,r_0]$.

Let $U\subset\M$ be any measurable subset and let $\rho\in(0,r_0]$, then we define the cone of height $\rho$ from $U$ by
\begin{equation}\label{def:cone}
    \C_\rho(U):=\{\omega+r\nu(\omega):\omega\in U,\ r\in(0,\rho)\}
    \, .
\end{equation}
We are interested in studying minimizers of \eqref{def:energie} while imposing a strong tangential boundary condition. More precisely, we minimize $\exe$ over the admissible class of functions
\begin{equation*}
    \A:=\{Q\in Q_\infty+H^1(\Om;\sym),\ Q|_\M\in\N \:\text{ \answ{$\H^2$-}a.e.\ on $\M$},\ n(Q|_\M)\cdot\nu=0\: \text{ \answ{$\H^2$-}a.e.\ on $\M$}\}
    \, ,
\end{equation*}
\blue{where $Q|_\M$ denotes the restriction of $Q$ to $\M$ in the sense of traces.}

\answ{We note that the class $\A$ is not empty and includes not only maps with integer degree point defects, for which the associated director $n$ is liftable (and thus enjoying regularity properties like those of $Q$), but also other configurations such as tensor fields with half-integer singularities that may not be liftable  (for references on the lifting problem, see e.g.\ \cite{BeChi,MiVaS}).
Indeed, we only require $Q$ to be of the form $Q=n\otimes n - \frac13 I$ pointwise for some $n\in \S^2$ and \emph{do not assume any regularity} beyond measurability on $n$. 
In particular, $n$ is allowed to have jump discontinuities.
This is consistent with the fact that $n$ is only well-defined up to a sign.
}


Our first result is that minimizers exist within the class of admissible functions $\A$. 

\answ{\begin{proposition}\label{thm:exist}
    For any $\xi,\eta>0$ there exists $\Qm\in\A$ such that 
    \begin{equation*}
        \exe(\Qm)=\inf\{\exe(Q):Q\in\A\}
        \, .
    \end{equation*}
\end{proposition}}

The next two theorems state the limiting energetical behaviour of minimizers of \eqref{def:energie}. \answ{Theorem \ref{thm:lower} serves as a lower bound for energy bounded
sequences and for Theorem \ref{thm:upper} we construct a recovery sequence for this
lower bound, showing it is asymptotically optimal.}

\begin{theorem}[Lower Bound]\label{thm:lower}
    Let $\Qm$ minimize $\exe$ with $\Qm\in\A$. If
    \begin{equation*}
        \frac{\eta}{\xi}\to\infty\quad\text{as}\quad\xi,\eta\to0
        \, ,
    \end{equation*}
    then for any measurable set $U\subset\M$,
    \begin{equation*}
        \liminf_{\xi,\eta\to0}\: \eta\exe(\Qm;\answ{\C}_{r_0}(U))
        \ \geq \
        \ftf\int_U\Big(1-\sqrt{1-\nu_3^2}\Big)\, d\H^2
    \, .
    \end{equation*}
\end{theorem}

\begin{theorem}[Upper Bound]\label{thm:upper}
    Let $\Qm$ minimize $\exe$ with $\Qm\in\A$. If
    \begin{equation*}
        \frac{\eta}{\xi}\to\infty\quad \text{as}\quad \xi,\eta\to0
        \, ,
    \end{equation*}
    then for any measurable set $U\subset\M$,
    \begin{equation*}
        \limsup_{\xi,\eta\to0}\eta \exe(\Qm;\C_{r_0}(U))
        \ \leq \ 
        \ftf\int_U \Big(1-\sqrt{1-\nu_3^2}\Big)\, d\H^2
        \, .
    \end{equation*}
\end{theorem}

In the case of a spherical particle $\M=\mathbb{S}^2$, Theorem \ref{thm:lower} and \ref{thm:upper} give the optimal lower and upper energy bound that was proven  already in \cite{bls}.
There, a uniaxial Dirichlet condition $Q_b$ is imposed on $\M$ and the limiting energy of minimizers is found to be
\begin{equation*}
    \ftf\int_{\S^2}(1-|n_3(Q_b)|)\, d\H^2
    \, .
\end{equation*}
This expression has an explicit dependence on the imposed boundary condition $Q_b$, and it has been shown in \cite{bls} that there exists a minimizing boundary condition within the class of uniaxial and tangential $Q_b$, for which the limit reads
\begin{equation*}
    \ftf\int_{\S^2}\Big(1-\sqrt{1-\nu_3^2}\Big)\, d\H^2
    \, .
\end{equation*}
This is due to the fact that $|n_3(Q_b(\omega))|\leq\sqrt{1-(\nu_3(\omega))^2}$ for all $\omega\in\S^2$ since $\nu(\omega)$ is perpendicular to $n(Q_b(\omega))$. 

Contrary to \cite{bls}, in this work we do not impose a Dirichlet boundary condition, the boundary data $\Qm|_\M$ is possibly changing with $\xi$ and $\eta$, so we would like a uniform lower bound which does not depend on the exact values of the minimizer along the boundary. 
The lower bound in Theorem \ref{thm:lower} has this property and we prove that this is in fact the correct lower bound as we are able to construct a recovery sequence in Section \ref{sec:upper} which attains this energy in the limit $\xi,\eta\to 0$.

The two main differences in the proofs compared to \cite{bls} are the following: 
In the lower bound, due to the non-convexity of the particle, we have to refine the radial estimates to show that the transition from the boundary condition to uniform alignment (at infinity) takes place already on a finite length scale (of order $\eta$). 
This requires quantitative estimates of the distance between the minimizer $\Qm$ and $Q_{\infty}$, see Lemma \ref{lemma:length}.
For the upper bound construction, we can not construct a constant recovery sequence with optimal boundary condition as in \cite{bls} since, depending on the geometry and orientation of the particle, this would create line singularities with a non-negligible energetic cost \answ{see Figure \ref{fig:line-sing}.}
\begin{figure}[hbt!]
    \centering
    \includegraphics[width=0.5\linewidth]{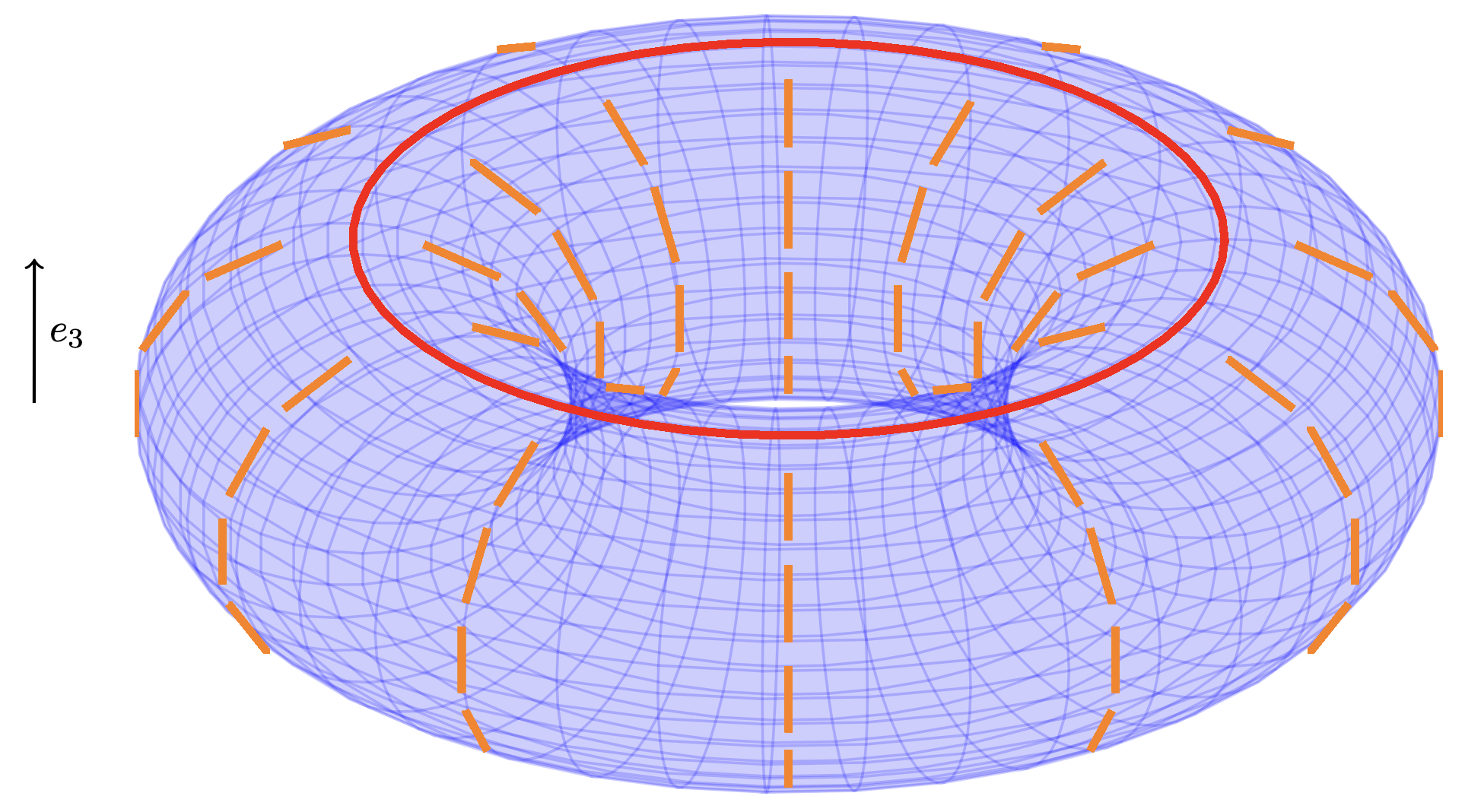}
    \caption{\answ{Torus with the optimal boundary condition (orange) and a line singularity (red).}}
    \label{fig:line-sing}
\end{figure}
Instead, line singularities are replaced by an appropriate number of point singularities and a Lipschitz extension \answ{is used to interpolate between singularities and the optimal boundary condition, see Figure \ref{fig:vectorfield}.}
\begin{figure}[hbt!]
    \centering
    \includegraphics[width=0.2\linewidth]{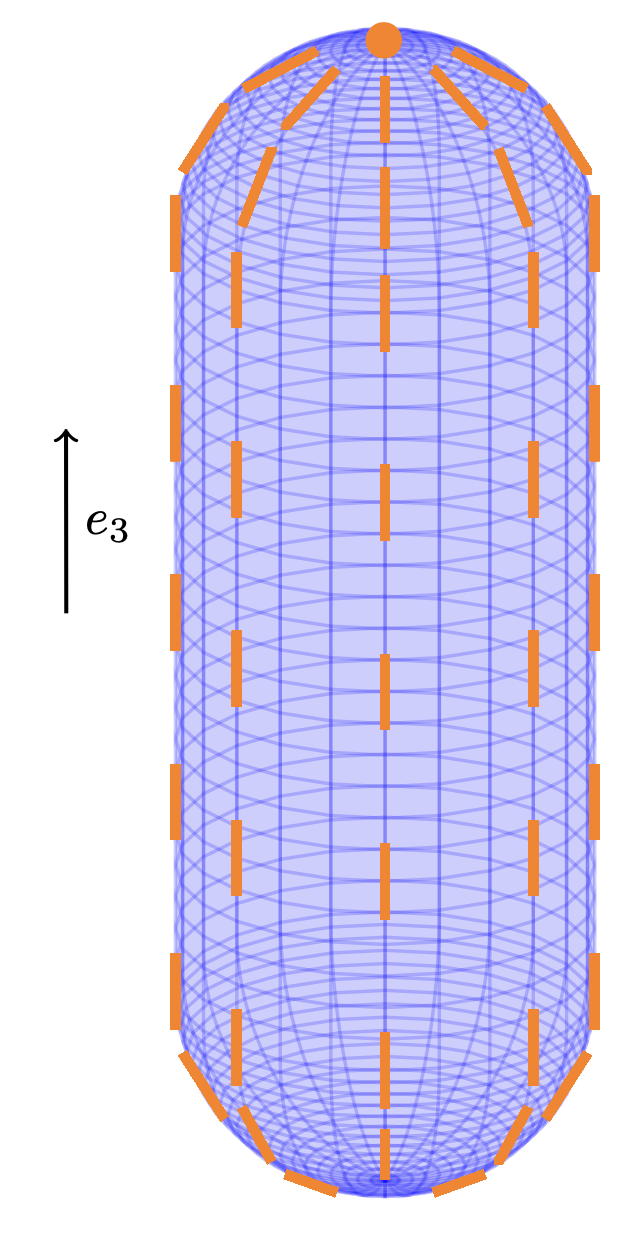}\hspace{2cm}\includegraphics[width=0.5\linewidth]{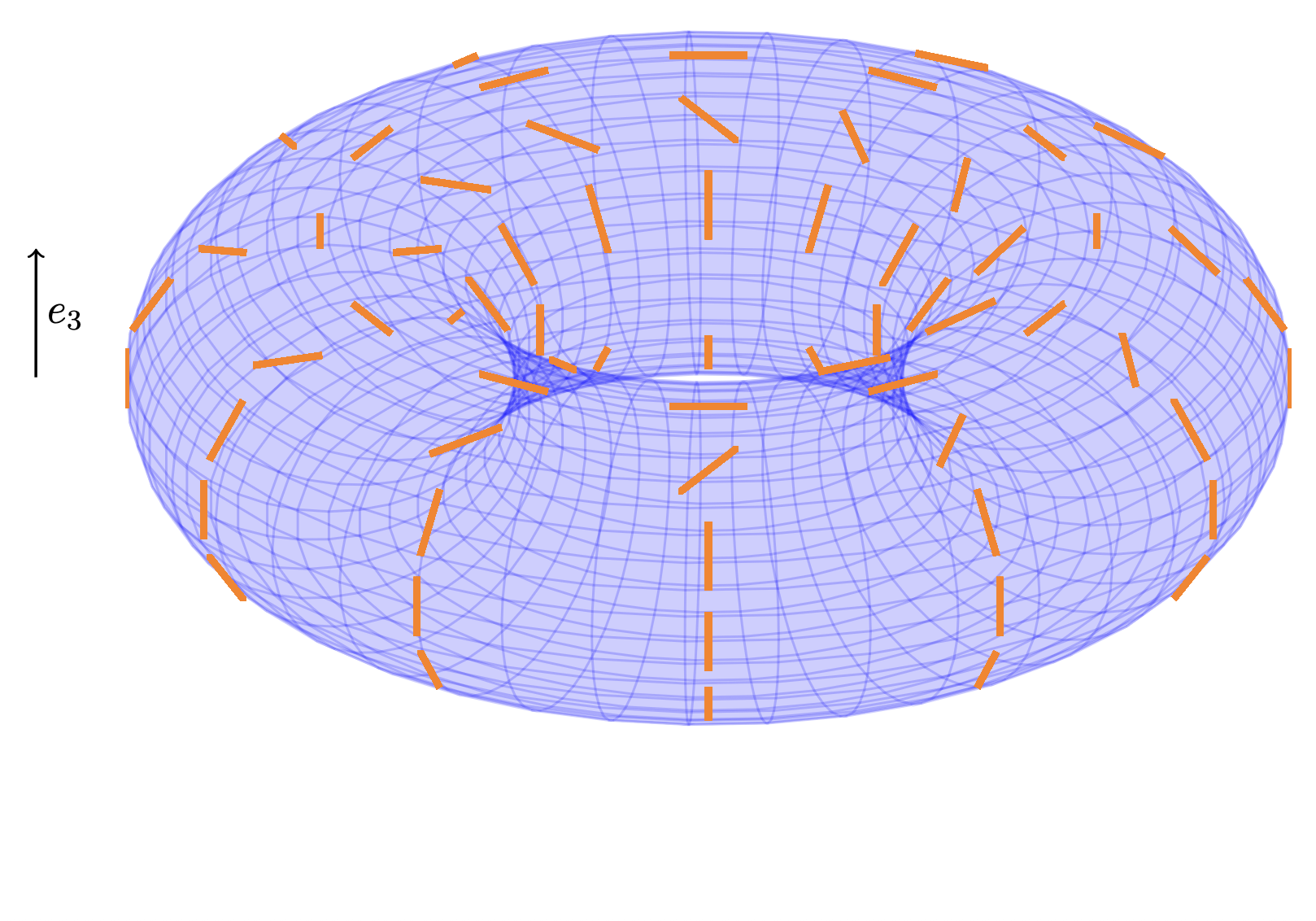}
    \caption{\answ{On the left, we depict a spherocylindrical particle with boundary condition that has two point defects (top and bottom of the spherical caps).
    On the right, a torus with close-to-optimal boundary condition that avoids the creation of a line defect as for example in Figure~\ref{fig:line-sing}, and that will be used in the construction of the upper bound. 
    No point defects are needed in this case.
    }}
    \label{fig:vectorfield}
\end{figure}
\answ{
Away from these point singularities, we use an adaptation of the optimal one-dimensional profile from \cite{abl}, which we restate in Lemma \ref{lemma:optimal} for convenience, in order to define a competitor} while respecting the tangency condition on $\M$.

The corresponding results in the case of strong homeotropic anchoring have been obtained in \cite{abl, ACS2021, ACS2024}.
While a surface energy similar to the one in Theorem \ref{thm:lower} and \ref{thm:upper} is also present in those articles, it has been observed in \cite{ACS2021, ACS2024} that in certain parameter regimes, there are additional energetic contributions due to the presence of line singularities. 
For non-convex particles, there exist minimizing configurations whose energy is \emph{not} entirely concentrated on the particle surface $\M$, see \cite{ACS2024}. 
Remarkably, this is not the case for planar anchoring as the degree of freedom in the boundary condition can be used to avoid such contributions, in particular those due to line defects. 

\answ{
This is reflected in the fact that in Theorem~\ref{thm:lower} and \ref{thm:upper}, there is no term associated to line defects and we can show in a particular regime that indeed any line defects present for $\eta,\xi>0$ must vanish in the limit, see Theorem~\ref{thm:line_defects_shrink_0}.

\blue{We remark that the assumption of connectedness of $\M$ can be relaxed. 
In particular, $\mathcal{D}$ can consist of several components. 
If each connected component of $\M$ is an oriented $C^{2,1}-$manifold of dimension $2$ without boundary and is at a positive distance from each other component, we can carry out all steps in the proof of Theorem~\ref{thm:lower} and \ref{thm:upper} separately for each component and add up the individual bounds.
This is due to the fact that all energy concentrates in the boundary layer of size $\eta$ which eventually becomes smaller than the distance separating the components.
This situation of $\mathcal{D}$ consisting of two components of positive distance corresponds to the physical situation in which the distance between two colloidal particles increases at the same rate as the size of the particles.
Therefore, we cannot expect to see an interaction between the colloidal particles in this parameter regime.
}

Contrarily to the spherical case, the energy for a general manifold $\M$ depends on its orientation relative to the applied magnetic field in $e_3-$direction. 
It is therefore natural to ask the question, what would be the energy minimizing orientation if the particle was allowed to rotate.
We answer this question for several examples in Section~\ref{sec:Rotations}, see also Figures~\ref{fig:capsule},~\ref{fig:torus}.
}

\paragraph{Organization of the paper.} 
In Section~\ref{sec:exist} we prove the existence of minimizers of the energy with tangential anchoring. 
In Section~\ref{sec:lower}, we establish the
lower bound from Theorem~\ref{thm:lower}. 
In Section~\ref{sec:upper}, we obtain the upper bound of Theorem~\ref{thm:upper}, and using an appropriate construction to deal with possible line defects, we eliminate them in Section~\ref{sec:lines}  under an additional assumption. 
Finally in Section~\ref{sec:Rotations}
we discuss the optimal orientation of $\M$ with respect to the imposed magnetic field for the limiting energy and compare to existing literature.

\paragraph{Acknowledgements:} 
The authors thank Xavier Lamy for useful discussions. 
LB is supported by an NSERC (Canada) Discovery Grants.
The main part of this work was carried out while DL and DS were still at McMaster University.

\section{Existence of Minimizers}\label{sec:exist}

The existence of minimizers of the energy $\exe$ with an imposed Dirichlet boundary condition is well established \cite{abl}.
However, we allow the boundary values to vary amongst uniaxial $Q-$tensors whose associated director is confined to the unit circle in the tangent space.  
So one needs to show that the obtained compactness for a minimizing sequence is not only enough to show lower semi-continuity for the energy, but also for these boundary conditions. 

\begin{proof}[Proof of \answ{Proposition} \ref{thm:exist}]
    Let $m := \inf\{\exe(Q):Q\in\A\}$ then by definition of the infimum there exists a sequence $Q^k\in\A$ such that $\exe(Q^k)\to m$ in $\R$. 
From this we see that there exists $C>0$ such that
\begin{equation*}
    C\geq\exe(Q^k)\geq\int_\Om\frac{1}{2}|\grad Q^k|^2\, dx=\frac{1}{2}\nnorm{\grad Q^k}^2_{L^2(\Om)}
    \, .
\end{equation*}
Using that $f(Q)+h^2 g(Q)\geq C(h)|Q-Q_\infty|^2$ for some $C(h)>0$ by (\ref{def:coer}), it holds that
\begin{equation*}
    C\geq\exe(Q^k)\geq\int_\Om\frac{1}{\xi^2}\Big(f(Q^k)+\frac{\xi^2}{\eta^2}g(Q^k)\Big)\geq \frac{1}{\xi^2} C(\xi/\eta)\nnorm{Q^k-Q_\infty}_{L^2(\Om)}^2
    \, ,
\end{equation*}
thus $Q^k-Q_\infty$ is a bounded sequence in $H^1(\Om;\sym)$. This gives the existence of a weakly convergent subsequence $Q^{k_i}$ which we relabel as $Q^k$, i.e.\ $Q^k\rightharpoonup Q^*$ in $H^1$ for some $Q^*\in Q_\infty+H^1(\Om;\sym)$. 
It remains to prove that $Q^*\in\A$ and that the energy functional $\exe$ is weakly lower semi-continuous. By the compactness of the trace operator from $H^1(\Om;\sym)\to L^2(\M;\sym)$, it follows that (up to passing to a subsequence) $Q^k|_\M\to Q^*|_\M$ in $L^2$.
Taking another subsequence if necessary, we obtain that $Q^k|_\M\to Q^*|_\M$ pointwise \answ{$\H^2$-}a.e.\ on $\M$. 
Using Fatou's lemma it follows that
\begin{equation*}
    \liminf_{k\to\infty}\int_\M f(Q^k)\, d\H^2\geq \int_\M f(Q^*)\, d\H^2\geq 0
    \, ,
\end{equation*}
but $f(Q^k)=0$ \answ{$\H^2$-}a.e.\ on $\M$, so this means
\begin{equation*}
    \int_\M f(Q^*)\, d\H^2=0
    \, ,
\end{equation*}
and hence $Q^*\in\N$ \answ{$\H^2$-}a.e.\ on $\M$.
\answ{Thus $Q^*$ admits the pointwise representation $Q^*=n(Q^*|_\M)\otimes n(Q^*|_\M)-\frac{1}{3}I$.
We want to show next that $n(Q^*|_\M)\cdot\nu=0$ $\H^2$-a.e.\ on $\M$.}
By Lebesgue's Dominated Convergence Theorem,
\begin{equation*}
    \int_\M \langle Q^k|_\M+\frac{1}{3}I,\nu\otimes\nu\rangle\, d\H^2\to\int_\M \langle Q^*|_\M+\frac{1}{3}I,\nu\otimes\nu\rangle\, d\H^2
    \, ,
\end{equation*}
\answ{where the inner product is as defined in \eqref{def:inner}.} 
Making use of the fact that \answ{$Q^k\in\N$ $\H^2$-a.e.\ on $\M$, and hence admits the pointwise representation}
\begin{equation*}
    Q^k|_\M=n(Q^k|_\M)\otimes n(Q^k|_\M)-\frac{1}{3}I,\ \text{with}\ n(Q^*|_\M)\cdot\nu=0,\ \text{\answ{$\H^2$-}a.e.\ on $\M$,}
\end{equation*}
\blue{it follows that} $Q^k|_\M+\frac{1}{3}I=n(Q^k|_\M)\otimes n(Q^k|_\M)$ \answ{$\H^2$-}a.e.\ on $\M$. 
Using that \blue{$\langle n(Q^k|_\M)\otimes n(Q^k|_\M),\nu\otimes \nu\rangle=(n(Q^k|_\M)\cdot\nu)^2=0$ $\H^2$-a.e.},
\begin{equation*}
    0=\int_\M \langle Q^k|_\M+\frac{1}{3}I,\nu\otimes\nu\rangle\, d\H^2\to\int_\M \langle Q^*|_\M+\frac{1}{3}I,\nu\otimes\nu\rangle\, d\H^2=0
    \, ,
\end{equation*}thus $n(Q^*|_\M)\cdot\nu=0$ \answ{$\H^2$-}a.e.\ on $\M$, so $Q^*\in\A$. 
Finally we show that $\exe$ is weakly lower semi-continuous. Using Fatou's lemma and the lower semi-continuity of $|\grad Q|^2$ w.r.t.\ weak $H^1$ convergence
\begin{equation*}
    \liminf_{k\to\infty}\: \exe(Q^k)\geq\int_\Om\frac{1}{2}|\grad Q^*|^2\, \answ{dx}+\answ{\int_\Omega}\liminf_{k\to\infty}\left(\frac{1}{\xi^2}f(Q^k)+\frac{1}{\eta^2}g(Q^k)\right)\, dx
    \, .
\end{equation*}
As weak $H^1$ convergence yields strong $L^2$ convergence of $Q^k$ to $Q^*$, we have {(up to passing to a subsequence)} $Q^k\to Q^*$ pointwise a.e. on $\Om$. Combining this with the continuity of $f$ and lower semi-continuity of $g$, we conclude
\begin{equation*}
    \liminf_{k\to\infty}\: \exe(Q^k)\geq\exe(Q^*)
    \, ,
\end{equation*}
and so we have proven the existence of a minimizer of $\exe$ which is in the admissible class~$\A$.
\end{proof}

\section{The Lower Bound}\label{sec:lower}

\answ{This section is devoted to the proof of the lower bound described in Theorem \ref{thm:lower}.} In what follows, we will often find it useful to project a $Q-$tensor onto the set of uniaxial $Q-$tensors {$\N$}.  
For any $Q\in\sym$ we can write it as \begin{equation}\label{eqn:decomposition_Q}
    Q=s\Big(n\otimes n-\frac{1}{3}I+t\Big(m\otimes m-\frac{1}{3}I\Big)\Big)
    \, ,
\end{equation}
where $n,m\in\S^2$ with $n\perp m$ and $s\in[0,\infty)$, $t\in[0,1]$, see \cite[Proposition 1]{majumdarzarnescu10}. \answ{Let $\lambda_1\geq\lambda_2\geq\lambda_3$ be the eigenvalues of $Q$, then} expression \answ{\eqref{eqn:decomposition_Q}} is \answ{a} unique \answ{decomposition (up to replacing $n$ (resp.\ $m$) by $-n$ (resp.\ $-m$))} provided $Q\notin\B$, where 
\begin{equation}\label{set:B}
    \B:=\{Q\in\sym: Q= 0\ \text{or}\ t= 1\}=\{Q\in\sym:\lambda_1=\lambda_2\}
    \, .
\end{equation}
Defining $n(Q)$ to be the vector $n$ from the decomposition \eqref{eqn:decomposition_Q} above, we introduce the projection onto $\N$ given by
\begin{equation}\label{def:proj}
    P(Q):=n(Q)\otimes n(Q)-\frac{1}{3}I
    \, .
\end{equation}
We note \blue{that this definition of $n(Q)$ extends the definition introduced in Section 1 to include not only $Q\in \N$ but also $Q\in\sym\setminus\B$ and we remark }that \blue{$P(Q)$} is $C^1$ on $\sym\setminus\B$ \answ{as seen in \cite[Lemma 12]{canevari3d}}.

Throughout the proof of Theorem \ref{thm:lower} we will consider the energy along a part of a ray from $\omega\in \M$, which we define by
\begin{equation}\label{ray-energy}
    \cEt(Q,\omega)=\int_0^{r_0}\Big(\frac{1}{2}\Big|\frac{\partial Q}{\partial r}\Big|^2+\frac{1}{\xi^2}f(Q)+\frac{1}{\eta^2}g(Q)\Big)(1+r\k_1(\omega))(1+r\k_2(\omega))\, dr
\end{equation}
\answ{for $Q\in H^1((0,r_0);\sym)$}. We note that,
\begin{equation*}
    |\grad \Qt|^2=\left|\frac{\partial \Qt}{\partial r}\right|^2+\frac{1}{(1+r|\k_1|)^2}\left|\frac{\partial \Qt}{\partial \omega_1}\right|^2+\frac{1}{(1+r|\k_2|)^2}\left|\frac{\partial \Qt}{\partial \omega_2}\right|^2\geq \left|\frac{\partial \Qt}{\partial r}\right|^2
    \, ,
\end{equation*}
\answ{where the notation $\frac{\partial}{\partial \omega_i}$ means $(\tau_i\cdot \nabla)$ with $\tau_1,\tau_2$ being orthonormal vectors in $T_\omega\M$ which diagonalize the second fundamental form}, so  \eqref{ray-energy} gives \blue{the following lower bound on} \eqref{def:energie}
\begin{equation*}
    \eta \exe(\Qt;\C_{r_0}(U))\geq\int_U\eta\cEt(\Qt(\cdot,\omega),\omega)\, d\H^2(\omega)
    \, ,
\end{equation*}
where $\C_{r_0}(U)$ is defined \blue{in} (\ref{def:cone}).
For each $\omega\in U$ there are two possibilities. 
Either
\begin{equation*}
    \liminf_{\xi,\eta\to0}\: \eta\cEt(\Qt(\cdot,\omega),\omega)\geq \ftf\Big(1-\sqrt{1-\nu_3^2}\Big)
    \, ,
\end{equation*}
in which case there is nothing left to prove, or
there always exist $\blue{\xi,\eta}>0$ small enough such that $\eta\cEt(\Qt(\cdot,\omega),\omega)<\ftf(1-\sqrt{1-\nu_3^2})$. 
As preparation for the estimates in this case,
we provide some useful facts and definitions.
For a\blue{n arbitrary} fixed \blue{$0<q_0<\sqrt{2/3}$}, define
\begin{align}\label{def:Lip_g}
    L_g
    \ := \
    \sup\bigg\{ \frac{|g(Q_1) - g(Q_2)|}{|Q_1-Q_2|} \: : \: Q_j\in \sym\text{ with } \sqrt{\frac{2}{3}}-q_0\leq |Q_j|\leq \sqrt{\frac{2}{3}}+q_0\, , j=1,2\} \bigg\}
    \, ,
\end{align}
i.e.\ $L_g$ is the Lipschitz constant of $g$ on the set $\{Q\in\sym:\sqrt{2/3}-q_0\leq |Q|\leq \sqrt{2/3}+q_0\}$.
In particular, small neighbourhoods of $\N$ are contained in this set.
We then choose $\alpha:=\sqrt\frac83(1+L_g)+4$ and for $\e>0$ we define the sets
\begin{equation*}
    \answ{A^\e:=\{Q\in\sym:|Q-Q_\infty|^2< \alpha\e\}}
    \, ,
\end{equation*}
\begin{equation*}
    B^\e:=\{Q\in\sym:|Q-Q_\infty|^2\geq \alpha\e,\ \dist(Q,\N)<\e\}
    \, ,
\end{equation*}
and
\begin{equation*}
    C^\e=\{Q\in\sym:|Q-Q_\infty|^2\geq\alpha\e,\ \dist(Q,\N)\geq\e\}
    \, .
\end{equation*}
The definition of these sets and the inspiration for the following lemma come from \cite[Lemma 5.3]{ACS2024}.
\begin{lemma}\label{lemma:length}
    Let $\e>0$ be sufficiently small, let $\omega\in \M$ and $\xi,\eta>0$.
    Assume that 
    \begin{equation*}
        \eta\:\cEt(\Qt(\cdot,\omega),\omega)
        \ < \
        \ftf\Big(1-\sqrt{1-\nu_3^2}\Big)
        \, .
    \end{equation*}
    Then, for $I^\e:=\{r\in[0,r_0]: |\Qt(r,\omega)-Q_\infty|^2<\alpha\e\}$ it holds
    \begin{equation*}
        |I^\e|\geq r_0-C\left(\frac{\eta}{\e}+\frac{\xi^2}{\eta\e^2}\right)
        \,
    \end{equation*}
    for some constant $C>0$\answ{, which is independent of $\xi$, $\eta$ and $\e$}.
\end{lemma}
\answ{We remark that Lemma \ref{lemma:length} will be used in the proof of Theorem \ref{thm:lower}, in which $\xi$ and $\eta$ are both very small compared to $\e$ and $\e^2$. As well, we will consider the regime where $\eta/\xi\to\infty$ so Lemma \ref{lemma:length} tells us that $|I^\e|$ is very close to $r_0$ for sufficiently small $\xi,\eta>0$.}
\begin{proof}
    We begin by decomposing $[0,r_0]$ into three disjoint sets, $I^\e, J^\e$, and $K^\e$, where
    \begin{equation}\label{def:k-eps}
        J^\e=\{r\in[0,r_0]:\Qt(r,\omega)\in B^\e\}\quad\text{and}\quad K^\e=\{r\in[0,r_0]:\Qt(r,\omega)\in C^\e\}
        \, .
    \end{equation}
    We can see that this means,
    \begin{equation*}
        |I^\e|=r_0-|J^\e|-|K^\e|
    \end{equation*}
    so it suffices to find upper bounds for the measures of $J^\e$ and $K^\e$.
    Beginning with $J^\e$, we \answ{let
    \begin{equation}\label{def:D_omega}
        D_\omega:=\ftf(1-\sqrt{1-(\nu_3(\omega))^2})
    \end{equation}
    in order to simplify notation and we immediately see from \eqref{ray-energy} that}
    \begin{equation*}
        D_\omega
        \ > \ 
        \eta\:\cEt(\Qt(\cdot,\omega),\omega)
        \ \geq \ 
        \frac{1}{\eta}\int_{J^\e}g(\Qt)(1-r_0\k)^2\ dr
        \ \geq \ 
        \frac{1}{4\eta}|J^\e|g_{\mathrm{min}}^\e
        \, ,
    \end{equation*}
    where $g_{\mathrm{min}}^\e=\min\{g(Q):Q\in \answ{\overline{B^\e}}\}$. \answ{Note that this minimum exists since $\overline{B^\e}$ is compact. The above inequality} implies that
    \begin{equation*}
        |J^\e|\leq\frac{4\eta D_\omega}{g_{\mathrm{min}}^\e}
        \, .
    \end{equation*}
    It remains to show that $g_\mathrm{min}^\e$ is bounded below independent of $\xi$ or $\eta$.
    Note that for $\e$ sufficiently small, $\answ{\overline{B^\e}}$ is contained in \blue{a small neighbourhood of $\N$, so} \begin{equation*}
        \answ{\overline{B^\e}}\subset \left\{Q\in\sym: \blue{\sqrt{2/3}-q_0\leq |Q|\leq \sqrt{2/3}+q_0}\right\}
        \, ,
    \end{equation*}
    \blue{where $q_0$ is as defined for \eqref{def:Lip_g}}. So we can use Lipschitz continuity of $g$ with the Lipschitz constant $L_g$. 
    Let $Q\in \answ{\overline{B^\e}}$ and recall from \eqref{def:proj} that $P(Q)$ is the projection of $Q$ onto $\N$.
    Then it holds that
    \begin{equation*}
        g(P(Q))-g(Q)
        \ \leq \
        L_g|P(Q)-Q|
        \ = \ 
        L_g\: \dist(Q,\N)
        \ < \ 
        L_g\e
        \, ,
    \end{equation*}
    from which
    \begin{equation*}
        g(Q)
        \ > \ 
        g(P(Q))-L_g\e
        \, .
    \end{equation*}
    Since $P(Q)$ is uniaxial and $|P(Q)-Q_\infty|^2=2(1-(n_3(Q))^2)$, we can write $g(P(Q))$ as
    \begin{align*}
        g(P(Q))
        \ &= \ 
        \sqrt\frac{3}{8}|P(Q)-Q_\infty|^2
        \ \geq \ 
        \sqrt\frac{3}{8}(|Q-Q_\infty|-|P(Q)-Q|)^2 \\
        \ &\geq \ 
        \sqrt\frac38(|Q-Q_\infty|^2-2|Q-Q_\infty|\dist(Q,\N))
        \ \geq \ 
        \sqrt\frac38(\alpha\e-2\e|Q-Q_\infty|)
        \, .
    \end{align*}
    We observe that $|Q-Q_\infty|\leq \dist(Q,\N)+\text{diam}(\N)\leq\e+2\sqrt{\frac23}\leq 2$, so
    \begin{align*}
        g(P(Q))
        \ \geq \ 
        \sqrt\frac38(\alpha-4)\e=(1+L_g)\e
        \quad \text{and}\quad 
        g(Q)
        \ \geq \
        (1+L_g)\e-L_g\e
        \ = \ 
        \e
        \, .
    \end{align*}
    Therefore $g_\mathrm{min}^\e\geq\e$, so we obtain the bound
    \begin{equation*}
        |J^\e|
        \ \leq \
        \frac{4\eta D_\omega}{\e}
        \, .
    \end{equation*}
    Next we consider $K^\e$ for which we have the following estimate:
    \begin{equation*}
        D_\omega
        \ \geq\ \eta\cEt(\Qt(\cdot,\omega),\omega)
        \ \geq\
        \frac{\eta}{\xi^2}\int_{K^\e}f(\Qt)(1-r_0\k)^2\, dr
        \ \geq\
        \frac{\eta}{4\xi^2}|K^\e|f_\mathrm{min}^\e
        \, .
    \end{equation*}
    Therefore
    \begin{equation*}
        |K^\e|\leq\frac{4\xi^2 D_\omega}{\eta f_\mathrm{min}^\e}
        \, ,
    \end{equation*}
    where $f_\mathrm{min}^\e=\min\{f(Q):Q\in C^\e\}$. \answ{We note that this minimum exists since $C^\e$ is closed and $f$ is both coercive and continuous.} Using that $f(Q)\geq C_f\dist^2(Q,\N)$ for some fixed constant $C_f>0$ \answ{from \eqref{bound:f}} and that $\dist(Q,\N)\geq\e$ on $K^\e$, we can see that
    \begin{equation}\label{bound:K_eps}
        |K^\e|\leq\frac{4\xi^2 D_\omega}{C_f\eta\e^2}\quad\text{and}\quad|I^\e|\geq r_0-4D_\omega\Big(\frac{\eta}{\e}+\frac{\xi^2}{C_f\eta\e^2}\Big)
        \, .
    \end{equation}
\end{proof}
\answ{Note that in the proof of Lemma \ref{lemma:length}, we have also shown that for any fixed $\e>0$, $|J^\e|,|K^\e|\to 0$ as $\xi,\eta\to 0$. This fact will be used with} Lemma \ref{lemma:length} to prove the lower bound.
\begin{proof}[Proof of Theorem \ref{thm:lower}]
    We begin by noticing that we can bound the energy in $\C_{r_0}(U)$ using the energy along rays,
    \begin{align*}
        \eta \exe(\Qt;\C_{r_0}(U))&\geq\eta\int_U \cEt(\Qt(\cdot,\omega),\omega)\, d\H^2(\omega)\\
        &=\eta\left(\int_{U_1} \cEt(\Qt(\cdot,\omega),\omega)\, d\H^2(\omega)+\int_{U_2} \cEt(\Qt(\cdot,\omega),\omega)\, d\H^2(\omega)\right)
        \, ,
    \end{align*}
    where
    \begin{equation*}
        U_1:=\Big\{\omega\in U:\liminf_{\xi,\eta\to0}\eta \cEt(\Qt(\cdot,\omega)\blue{,\omega})\geq \ftf\Big(1-\sqrt{1-\nu_3^2}\Big)\Big\}
        \, ,
    \end{equation*}
    and $U_2:=U\setminus U_1$. 
    Since for $\omega\in U_1$ the lower bound is trivially attained, we will devote our attention to $\omega\in U_2$. We first note that for $\H^2$-a.e. $\omega\in U_2$, $\Qt(\cdot,\omega)\in H^1((0,r_0);\sym)$, so it is continuous on $[0,r_0]$ by Sobolev embedding theorems. \blue{Recalling that}
    \begin{equation*}
        K^\e=\{r\in[0,r_0]:\Qt(r,\omega)\in C^\e\}
    \end{equation*}
    \blue{as defined in \eqref{def:k-eps}, continuity of $\Qm(\cdot,\omega)$ implies that $K^\e$}
    is closed and thus compact. \answ{We remark that it is possible that $K^\e$ is empty, but the following covering argument will still hold in this case with a few minor modifications.}
    Let $\mu>0$ be small, then by definition of the outer measure, we can find an open set $G_\mu$ such that $K^\e\subset G_\mu$ and $|G_\mu|\leq|K^\e|+\mu$.
    Also since $\Qm(0,\omega)\in\N$, we can see that $0\notin K^\e$ and $K^\e$ is closed, so there is a positive distance between $0$ and $K^\e$. Using this fact we can choose $G_\mu$ such that $\inf (G_\mu)>0$.
    It is however possible that $\sup(G_\mu)>r_0$, but we will fix this problem soon.
    \answ{Note that if $K^\e$ is empty, the notion of distance is not well-defined, but we can still choose $G_\mu$ with the above properties.}
    
    Since $G_\mu$ is an open set in $\R$, it is a countable union of open intervals.
    Applying compactness of $K^\e$, we can choose finitely many of these intervals which still cover $K^\e$ and call their union $G'_\mu$. Let $N$ be the number of disjoint open intervals which make up $G'_\mu$, then we take the endpoints of each interval and label them by $\{r_1^\e,\dots,r_{2N}^\e\}$, noting that $r_1^\e>0$ by our choice of $G_\mu$. Next, using Lemma \ref{lemma:length} we can choose $r_\omega^\e\in I^\e$ such that
    \begin{equation*}
        r_\omega^\e\leq C\left(\frac{\eta}{\e}+\frac{\xi^2}{\eta\e^2}\right)\quad\text{and}\quad |r_\omega^\e-r_k^\e|>0\text{ for all $k=1,\dots,2N$}
        \, .
    \end{equation*}
    \answ{We recall that by assumption $\eta/\xi\to \infty$, so for sufficiently small $\xi,\eta$ its guaranteed that $\roe<r_0$.} Further, \answ{we choose $r_i^\e$ with the property that $i=\max\{k=1,\dots,2N:r_k^\e<r_\omega^\e\}$, whenever this exists}. \answ{If no such $r_i^\e$ exists for the choice of $\roe$, then this means $[0,\roe]\subset I^\e\cup J^\e$, so the following argument becomes trivial.}
    Now we define $G_\mu''$ to be
    \begin{equation*}
        G_\mu'':=\{r\in G'_\mu:r<r_i^\e\}
        \, ,
    \end{equation*}
    \answ{where $G_\mu''=\emptyset$ if there is no possible $r_i^\e$ }and note that \answ{when} $G_\mu''$ \answ{is non-empty}, it still has the property that
    \begin{equation*}
        |G_\mu''|\leq|K^\e|+\mu\quad\text{and}\quad\sup(G_\mu'')<r_0
        \, .
    \end{equation*}
    \answ{An example of a} division of $[0,r_0]$ into $I^\e$, $J^\e$ and $K^\e$ as well as a possible choice for $G_\mu''$ and $\roe$ is shown in Figure~\ref{fig:subdivision} below. 
    \answ{This figure depicts a possible scenario to show how one might divide the ray, however the positions, sizes, and ordering of the regions $I^\e$, $J^\e$ and $K^\e$ can vary on different rays.}
    \begin{figure}[H]
    \begin{center}
    \begin{tikzpicture}
            \draw (-7,-0.5) node {$0$};
            \draw (-7,0) -- (7,0);
            \draw (-7,0.2) -- (-7,-0.2);
            \draw (7,0.2) -- (7,-0.2);
            \draw (-6,0.7) node {$J^\e$};
            \draw (-4.5,0.5) node {$K^\e$};
            
            \draw (-3,0.7) node {$J^\e$};
            
            \draw [pattern=north east lines] (-2,-0.24) rectangle (-1,0.24);
            \draw (-1.5,0.5) node {$K^\e$};
            \draw (0,0.7) node {$J^\e$};
            \draw (1,0) node {$\big($};
            \draw (2,0.5) node {$I^\e$};
            \draw (3.25,0.5) node {$K^\e$};
            \draw [pattern=north east lines] (3,-0.24) rectangle (3.5,0.24);
            \draw (5.25,0.5) node {$I^\e$};
            \draw (7,-0.5) node {$r_0$};

            \draw (-5.2,0) node {$\big($};
            \draw (-5.2,-0.6) node {$r_1^\e$};
            \draw[decorate, decoration={brace,mirror}] (-5.2,-1) -- (-3.8,-1);
            \draw (-3.8,0) node {$\big)$};
            \draw (-3.8,-0.6) node {$r_2^\e$};
            \draw (-4.5,-1.4) node {$G_\mu''$};
            
            \draw (-2.2,0) node {$\big($};
            \draw (-2.2,-0.6) node {$r_3^\e$};
            \draw[decorate, decoration={brace,mirror}] (-2.2,-1) -- (-0.8,-1);
            \draw (-0.8,0) node {$\big)$};
            \draw (-0.8,-0.6) node {$r_4^\e$};
            \draw (-1.5,-1.4) node {$G_\mu''$};
    
            \draw[fill] (1.5,0) circle (2pt);
            \draw (1.5,-0.5) node {$r_\omega^\e$};

            \draw [pattern=north east lines] (-5,-0.24) rectangle (-4,0.24);

            \draw[decorate, decoration={brace}] (-7,0.35) -- (-5,0.35);
            \draw[decorate, decoration={brace}] (-4,0.35) -- (-2,0.35);
            \draw[decorate, decoration={brace}] (-1,0.35) -- (1,0.35);
    \end{tikzpicture}
    \end{center} 
    \caption{Illustration of \answ{an example of} the subdivision of $[0,r_0]$ into $I^\e,J^\e,K^\e$ as used in the proof of Theorem \ref{thm:lower}. The set $G_\mu''$ is indicated for one choice of $r_\omega^\e$.}\label{fig:subdivision}
    \end{figure}
    Now for all $r\in[0,r_\omega^\e]\setminus G_\mu''$, we have that $Q_{\xi,\eta}(r,\omega)\in A^\e\cup B^\e$ and hence the projection onto $\N$ given by $P$ is well defined here.
    \begin{align*}
        \eta\cEt(\Qt(\cdot,\omega),\omega) &\geq\int_{[0,r_\omega^\e]\setminus G_\mu''}\left(\frac{\eta}{2}\Big|\frac{\partial \Qt}{\partial r}\Big|^2+\frac{1}{\eta}g(\Qt)\right)(1+r\k_1(\omega))(1+r\k_2(\omega))\, dr\\
        &\geq\int_{[0,r_\omega^\e]\setminus G_\mu''}\left(\frac{\eta}{2}\Big|\frac{\partial \Qt}{\partial r}\Big|^2+\frac{1}{\eta}g(\Qt)\right)(1-r_\omega^\e\k(\omega))^2\, dr
        \, .
    \end{align*}
    Then using Lemma 17 in \cite{canevari3d} we have the bound
    \begin{align*}
        \int_{[0,r_\omega^\e]\setminus G_\mu''}\frac{\eta}{2}\Big|\frac{\partial \Qt}{\partial r}\Big|^2\, dr\geq\int_{[0,r_\omega^\e]\setminus G_\mu''}\frac{\eta}{2}(\gamma(\Qt))^2\Big|\frac{\partial P(\Qt)}{\partial r}\Big|^2\, dr
        \, ,
    \end{align*}
    where $\gamma:\sym\to\R$ is defined by $\gamma(Q)=\lambda_1(Q)-\lambda_2(Q)$, for $\lambda_1,\lambda_2$, the two leading eigenvalues of $Q$. Note that although in \cite{canevari3d}, the lemma is stated for $|\grad Q|$ instead of the radial derivative, it is shown in the proof that the same inequality holds for a directional derivative too, so we can apply this lemma without any modifications.
    We can then use Lemma 13 from \cite{canevari3d} which gives a Lipschitz bound for $\gamma$ away from $\B=\{Q\in\sym:Q=0\ \text{or}\ \lambda_1(Q)=\lambda_2(Q)\}$ \answ{as defined in \eqref{set:B}}.
    So we have
    \begin{align*}
        \gamma(P(\Qt))-\gamma(\Qt)\leq 2|P(\Qt)-\Qt|=2\dist(\Qt,\N)\leq 2\, \answ{\max\{\e,\sqrt{\alpha\e}\}=2\sqrt{\alpha\e}}
        \, ,
    \end{align*}
    but since $\gamma(P(\Qt))=1$, this implies that, $\gamma(\Qt)\geq 1-2\answ{\sqrt{\alpha\e}}$. Therefore,
    \begin{align*}
        \int_{[0,r_\omega^\e]\setminus G_\mu''}\frac{\eta}{2}\Big|\frac{\partial \Qt}{\partial r}\Big|^2\, dr\geq\int_{[0,r_\omega^\e]\setminus G_\mu''}\frac{\eta}{2}(1-2\answ{\sqrt{\alpha\e}})^2&\Big|\frac{\partial P(\Qt)}{\partial r}\Big|^2\, dr\\
        &=\int_{[0,r_\omega^\e]\setminus G_\mu''}\eta(1-2\answ{\sqrt{\alpha\e}})^2\Big|\frac{\partial n^+}{\partial r}\Big|^2\, dr
        \, ,
    \end{align*}
    where $n^+=n^+(\Qm)$ is chosen from two possible orientations and it has the following properties:
    \begin{itemize}
        \item $n^+$ is continuous on each interval $[r_k^\e,r_{k+1}^\e]$ for $k=0,2,4,\dots,i-2$ as well as on $[r_i^\e,r_\omega^\e]$, where we define $r_0^\e=0$.
        \item $n^+(r_\omega^\e)$ is close to $e_3$, i.e.\ $n^+(r_\omega^\e)\cdot e_3>0$
        \item $n^+(r_k^\e)\cdot n^+(r_{k+1}^\e)\geq 0$ for all $k=1,3,5,\dots,i-1$.
    \end{itemize}
    This last property of $n^+$ will be very important later on in the proof. \answ{In the case that $G_\mu''=\emptyset$, $n^+$ is chosen to be continuous on $[0,\roe]$ and satisfy the second property.}
    Next we can see that
    \begin{equation*}
        \int_{[0,r_\omega^\e]\setminus G_\mu''}\frac{1}{\eta}g(\Qt)\, dr\geq \int_{[0,r_\omega^\e]\setminus G_\mu''}\frac{1}{\eta}g(P(\Qt))\, dr-\int_{[0,r_\omega^\e]\setminus G_\mu''}\frac{1}{\eta}|g(\Qt)-g(P(\Qt))|\, dr
        \, .
    \end{equation*}
    We then use that $g$ is Lipschitz as well as the Cauchy\answ{-Schwarz} inequality to get,
    \begin{align*}
        \int_{[0,r_\omega^\e]\setminus G_\mu''}\frac{1}{\eta}|g(\Qt)-g(P(\Qt))|\, dr\leq \int_{[0,r_\omega^\e]\setminus G_\mu''}&\frac{L_g}{\eta}\dist(\Qt,\N)\, dr\\
        &\leq\frac{L_g}{\eta}\left(|r_\omega^\e|\int_{[0,r_\omega^\e]\setminus G_\mu''}\dist^2(\Qt,\N)\, dr\right)^{1/2}
        \, .
    \end{align*}
    \answ{From \eqref{bound:f} we have} that $f(Q)\geq C_f\dist^2(Q,\N)$ for all $Q\in\sym$ and some constant $C_f>0$, so
    \begin{equation*}
        \int_{[0,r_\omega^\e]\setminus G_\mu''}\dist^2(\Qt,\N)\, dr\leq\frac{\xi^2}{C_f\eta}\int_{[0,r_\omega^\e]\setminus G_\mu''}\frac{\eta}{\xi^2}f(\Qt)\, dr\leq \frac{\xi^2}{C_f\eta}D_\omega\, ,
    \end{equation*}
    \blue{recalling that by definition of $U_2$, there is a subsequence of $\xi,\eta$ with $\eta\cEt(Q_{\xi,\eta}(\cdot,\omega),\omega)<D_\omega$, where $D_\omega$ is as defined in \eqref{def:D_omega}.} Since furthermore 
    \begin{equation*}
        r_\omega^\e\leq C\left(\frac{\eta}{\e}+\frac{\xi^2}{\eta\e^2}\right)
        \, ,
    \end{equation*}
    we get that
    \begin{align*}
        \int_{[0,r_\omega^\e]\setminus G_\mu''}\frac{1}{\eta}|g(\Qt)-g(P(\Qt))|\, dr&\leq 
        \frac{L_g \answ{C}}{\eta}\left(\frac{\xi^2}{C_f\eta}\right)^{1/2}\left(\frac{\eta}{\e}+\frac{\xi^2}{\eta\e^2}\right)^{1/2}D_\omega^{\answ{1/2}}\\
        &=L_gC\left(\frac{\xi^2}{\eta^2\e}+\frac{\xi^4}{\eta^4\e^2}\right)^{1/2}D_\omega^{\answ{1/2}}\to0
    \end{align*}
    as $\xi,\eta\to0$. 
    So we have that
    \begin{align*}
        \eta\cEt(\Qt(\cdot,\omega),\omega)&\geq(1-r_\omega^\e\k(\omega))^2\int_{[0,r_\omega^\e]\setminus G_\mu''} \eta(1-2\answ{\sqrt{\alpha\e}})^2\Big|\frac{\partial n^+}{\partial r}\Big|^2+\frac{1}{\eta}g(P(\Qt))\, dr-\bigO\Big(\frac{\xi}{\eta}\Big)\\
        &\geq 2(1-r_3^\e\k(\omega))^2(1-2\answ{\sqrt{\alpha\e}})\int_{[0,r_\omega^\e]\setminus G_\mu''} \Big|\frac{\partial n^+}{\partial r}\Big|\sqrt{g(P(\Qt))}\, dr-\bigO\Big(\frac{\xi}{\eta}\Big)
        \, .
    \end{align*}
    Coming from the proof of Lemma 3.4 in \cite{abl} we have the inequality that if $n=(n_1,n_2,n_3)$ and $|n|=1$, defining $N=(\sqrt{1-n_3^2},0,n_3)$ we obtain the estimate,
\begin{equation}\label{bound:n3}
    |\Dot{N}|\leq|\dot{n}|\quad\text{i.e.,}\quad\frac{|\Dot{n_3}|^2}{1-n_3^2}\leq |\Dot{n}|^2
\, .
\end{equation}
    Therefore since $|n^+|=1$, it follows from \eqref{bound:n3} that
    \begin{equation*}
        \Big|\frac{\partial n^+}{\partial r}\Big|
        \ \geq  \
        \frac{1}{\sqrt{1-(n^+_3)^2}}\Big|\frac{\partial n^+_3}{\partial r}\Big|
        \quad\text{and}\quad 
        g(P(\Qt))
        \ = \ 
        \sqrt{\frac{3}{2}}(1-(n^+_3)^2)
        \, .
    \end{equation*}
    Thus, 
    \begin{align*}
        \eta\cEt(\Qt(\cdot,\omega),\omega)&\geq2(1-r_3^\e\k(\omega))^2(1-2\answ{\sqrt{\alpha\e}})\int_{[0,r_\omega^\e]\setminus G_\mu''}\sqrt[4]{\frac{3}{2}}\Big|\frac{\partial {n}^+_3}{\partial r}\Big|\, dr-\bigO\Big(\frac{\xi}{\eta}\Big)\\
        &\geq{\ftf} (1-r_3^\e\k(\omega))^2(1-2\answ{\sqrt{\alpha\e}})\left|\int_{[0,r_\omega^\e]\setminus G_\mu''}\frac{\partial {n}^+_3}{\partial r}\, dr\right|-\bigO\Big(\frac{\xi}{\eta}\Big)
        \, .
    \end{align*}
    Applying the Fundamental Theorem of Calculus, we see that
    \begin{align*}
        \left|\int_{[0,r_\omega^\e]\setminus G_\mu''}\frac{\partial {n}^+_3}{\partial r}\, dr\right|&=\Big|n^+_3(r_\omega^\e)-n^+_3(r_i^\e)+n^+_3(r_{i-1}^\e)-\dots+n^+_3(r_1^\e)-n^+_3(0)\Big|\\
        &\geq\Big(|n_3^+(r_\omega^\e)|-|n_3^+(0)|\Big)-\sum_{k=1}^{i/2}|n^+_3(r_{2k}^\e)-n^+_3(r_{2k-1}^\e)|
        \, .
    \end{align*}
    \answ{In the case $G_\mu''=\emptyset$, the lower bound is exactly $|n_3^+(\roe)|-|n_3^+(0)|$. }We want to estimate $|n_3(r_\omega^\e)|$ using that it is close to $Q_\infty$ and close to $\N$, so we have the bound
    \begin{align*}
        |Q_{\infty}-P(\Qt(r_\omega^\e,\omega))|\leq|Q_\infty-\Qt(r_\omega^\e,\omega)|+|\Qt(r_\omega^\e,\omega)-P(\Qt(r_\omega^\e,\omega))|\leq\answ{2\sqrt{\alpha\e}}
        \, .
    \end{align*}
    Looking only at the difference between the 33-components of $Q_\infty$ and $P(\Qm)$,
    \begin{align*}
        |Q_{\infty,33}-P(\Qm(\roe,\omega))_{33}|\leq 2\sqrt{\alpha\e}
        \, .
    \end{align*}
    We can remove the absolute value, maintaining the same inequality and write the left hand side explicitly as,
    \begin{align*}
        1-({n}_3^+(\roe))^2\leq 2\sqrt{\alpha\e}
        \, .
    \end{align*}
    Rearranging this inequality yields
    \begin{equation*}
        |n_3^+(\roe)|\geq\sqrt{1-2\sqrt{\alpha\e}}
        \, .
    \end{equation*}
    Then using the fact that for any tangential $Q_b$, we have $|n_3(Q_b)|\leq \sqrt{1-\nu_3^2}$,
    \begin{equation*}
        |n_3^+(r_\omega^\e)|-|n_3^+(0)|\geq \sqrt{1-2\sqrt{\alpha\e}}-\sqrt{1-\nu_3^2}
        \, .
    \end{equation*}
    We now want to find an upper bound on each $|n_3^+(r_{2k}^\e)-n^+_3(r_{2k-1}^\e)|$. We can see that
    \begin{align*}
        |n_3^+(r_{2k}^\e)-n^+_3(r_{2k-1}^\e)|^2&\leq|n^+(r_{2k}^\e)-n^+(r_{2k-1}^\e)|^2\\
        &=2-2\left(n^+(r_{2k-1}^\e)\cdot n^+(r_{2k}^\e)\right)\\
        &\leq2-2\left(n^+(r_{2k-1}^\e)\cdot n^+(r_{2k}^\e)\right)^2\\
        &=|P(\Qt(r_{2k}^\e,\omega))-P(\Qt(r_{2k-1}^\e,\omega))|^2
        \, ,
    \end{align*}
    where we used here that $n^+$ is chosen such that the dot product is non-negative. 
    From here we use the fact that the projection $P$ is $C^1$ as seen in \cite[Lemma 12]{canevari3d}. 
    So on a compact neighbourhood of $\N$, there exists a constant $C>0$ such that
    \begin{equation*}
        |P(\Qt(r_{2k}^\e,\omega))-P(\Qt(r_{2k-1}^\e,\omega))|
        \leq C|\Qt(r_{2k}^\e,\omega)-\Qt(r_{2k-1}^\e,\omega)|
    \end{equation*}
    therefore,
    \begin{align*}\label{bound:r1r2}
        |n_3^+(r_{2k}^\e)-n^+_3(r_{2k-1}^\e)|\leq C|\Qt(r_{2k}^\e,\omega)-\Qt(r_{2k-1}^\e,\omega)|
        \, .
    \end{align*}
    Now we see that
    \begin{equation*}
        |\Qt(r_{2k}^\e,\omega)-\Qt(r_{2k-1}^\e,\omega)|=\left|\int_{r_{2k-1}^\e}^{r_{2k}^\e}\frac{\partial \Qt}{\partial r}\, dr\right|\leq\int_{r_{2k-1}^\e}^{r_{2k}^\e}\left|\frac{\partial \Qt}{\partial r}\right|\, dr
        \, .
    \end{equation*}
    Putting together all regions $[r_{2k-1}^\e,r_{2k}^\e]$, we see that,
    \begin{equation*}
        \sum_{k=1}^{i/2}|n_3^+(r_{2k}^\e)-n_3^+(r_{2k-1}^\e)|\leq C\int_{G_\mu''}\norm{\frac{\partial \Qm}{\partial r}}\, dr
        \, ,
    \end{equation*}
     and using the Cauchy-\answ{Schwarz} inequality \answ{and that $|G_\mu''|\leq |K^\e|+\mu$}
    \begin{equation*}
        \int_{G_\mu''}\left|\frac{\partial \Qm}{\partial r}\right|\, dr\leq|G_\mu''|^{1/2}\left(\int_0^{r_0}\left|\frac{\partial \Qm}{\partial r}\right|^2\, dr\right)^{1/2}\leq\Big(|K^\e|+\mu\Big)^{1/2}\Big(\frac{2D_\omega}{\eta}\Big)^{1/2}
        \, .
    \end{equation*}
    \answ{In the above chain of inequalities we are using the definition of $U_2$ and $\cEt$ to see that
    \begin{equation*}
        \int_0^{r_0}\norm{\frac{\partial \Qm}{\partial r}}^2\, dr=\frac{2}{\eta}\left(\int_0^{r_0}\frac{\eta}{2}\norm{\frac{\partial \Qm}{\partial r}}^2\, dr\right)\leq\left(\frac{2}{\eta}\right)\eta\cEt(\Qm(\cdot,\omega),\omega)<\frac{2D_\omega}{\eta}
        \, .
    \end{equation*}}
    Now \answ{we estimate $|K^\e|$ using \eqref{bound:K_eps} to} see that
    \begin{equation*}
        \sum_{k=1}^{i/2}|n_3^+(\Qm(r_{2k}^\e,\omega))-n_3^+(\Qm(r_{2k-1}^\e,\omega))|\leq C\Big(\frac{\xi^2}{\eta^2\e^2}+\frac{\mu}{\eta}\Big)^{1/2}
    \end{equation*}
    thus
    \begin{align*}
        \left|\int_{[0,r_\omega^\e]\setminus G_\mu}\frac{\partial {n}^+_3(\Qt)}{\partial r}\, dr\right|\geq \sqrt{1-2\sqrt{\alpha\e}}-\sqrt{1-\nu_3^2}-C\Big(\frac{\xi^2}{\eta^2\e^2}+\frac{\mu}{\eta}\Big)^{1/2}
        \, .
    \end{align*}
    We get the following bound on the energy along a ray
    \begin{align*}
        \eta&\cEt(\Qm(\cdot,\omega),\omega)\\
        &\geq\ftf(1-\roe\k(\omega))^2(1-2\answ{\sqrt{\alpha\e}})\left(\sqrt{1-2\sqrt{\alpha\e}}-\sqrt{1-\nu_3^2}-C\Big(\frac{\xi^2}{\eta^2\e^2}+\frac{\mu}{\eta}\Big)^{1/2}\right)-\bigO\Big(\frac{\xi}{\eta}\Big)
        \, ,
    \end{align*}
    \answ{recalling the assumption that $\xi/\eta\to 0$ as $\xi,\eta\to 0$.} From here we take $\mu\to0$, then using that $(1-r_\omega^\e\k(\omega))^2\to1$ as $\xi,\eta\to0$ since $r_\omega^\e\to0$, we can see that
    \begin{align*}
        \liminf_{\xi,\eta\to0}\: \eta \exe(\Qt;\C_{r_0}(U_2))\geq\int_{U_2}\ftf(1-2\answ{\sqrt{\alpha\e}})\Bigg(\sqrt{1-2\sqrt{\alpha\e}}-\sqrt{1-\nu_3^2}\Bigg)\, d\H^2
        \, .
    \end{align*}
    Now taking $\e\to0$ we have
    \begin{align*}
        \liminf_{\e\to0}\left(\liminf_{\xi,\eta\to0}\: \eta\exe(\Qt;\C_{r_0}(U_2))\right)&\geq\int_{U_2}\ftf\Big(1-\sqrt{1-\nu_3^2}\Big)\, d\H^2
        \, .
    \end{align*}
    Finally we can combine the integrals on $U_1$ and $U_2$ to recover,
    \begin{equation*}
        \liminf_{\xi,\eta\to0}\: \eta\exe(\Qt;\C_{r_0}(U))\geq\int_{U}\ftf \Big(1-\sqrt{1-\nu_3^2}\Big)\, d\H^2
        \, .
    \end{equation*}
\end{proof}
This lower bound characterizes the energy in cones of height $r_0$, but in order to rule out the possibility of line defects in Section \ref{sec:lines}, we will need a lower bound estimate which works for open sets rather than cones.
\begin{corollary}\label{cor:open}
    Let $\Qm$ minimize $\exe$ with $\Qm\in\A$. If
    \begin{equation*}
        \frac{\eta}{\xi}\to\infty\quad\text{as}\quad\xi,\eta\to0
        \, ,
    \end{equation*}
    then for any open set $U\subset\mathbb{R}^3$,
    \begin{equation*}
        \liminf_{\xi,\eta\to0}\: \eta\exe(\Qm;U\cap\Om)\geq\int_{U\cap\M}\ftf\Big(1-\sqrt{1-\nu_3^2}\Big)\, d\H^2
    \, .
    \end{equation*}
\end{corollary}
{Note that in the case where $U\cap\M$ is empty, this lower bound is trivial, but still holds.}
We remark that although Theorem \ref{thm:lower} deals with cones of height $r_0$, the proof of this Theorem {does} not rely on the specific value of $r_0$. In fact, we would obtain the same lower bound for $\C_\rho(U)$ for $\rho\in(0,r_0]$. This fact will be used in proving Corollary \ref{cor:open}.
\begin{proof}[Proof of Corollary \ref{cor:open}]
    Every point $\omega\in U\cap\M$ is an interior point of $U$, so there exists a ball of radius $r_\omega\in(0,r_0)$ given by $B(\omega,r_\omega)$ such that $B(\omega,r_\omega)\subset U$. Define $\rho:U\cap\M\to(0,r_0)$ by
    \begin{equation*}
        \rho(\omega)=\sup\big\{r'\in(0,r_0]:\big(\omega+r\nu(\omega)\big)\in U\text{ for all $r\in(0,r')$}\big\}
        \, ,
    \end{equation*}
    then $\rho(\omega)\geq r_\omega>0$ for all $\omega\in U\cap\M$.
    Using that $\rho$ is measurable, we divide $U\cap \M$ into countably many disjoint measurable sets $U_k$ defined by
    \begin{equation*}
        U_k=\Big\{\omega\in U\cap\M:\frac{r_0}{k+1}<\rho(\omega)\leq\frac{r_0}{k}\Big\}
        \, ,
    \end{equation*}
    see also Figure \ref{fig:Uk}.
    Note that by construction $U\cap\M=\bigcup_{k=1}^\infty U_k$. 
    Then for any $K\in\mathbb{N}$,
    \begin{equation*}
        \eta\exe(\Qm;U\cap\Om)\geq\sum_{k=1}^K\eta\exe(\Qm;\C_{r_0/k+1}(U_k))
    \end{equation*}
    and therefore applying Theorem \ref{thm:lower} where $r_0$ is replaced by $\frac{r_0}{k+1}$ we obtain
    \begin{align*}
        \liminf_{\xi,\eta\to0}\: \eta\exe(\Qm;U\cap\Om)\geq\sum_{k=1}^K\int_{U_k}\ftf\Big(1-\sqrt{1-\nu_3^2}\Big)\, d\H^2
        \, .
    \end{align*}
    The left side does not depend on $K$, so we can take $K\to\infty$ to see that
    \begin{align*}
        \liminf_{\xi,\eta\to0}\: \eta\exe(\Qm;U\cap\Om)&\geq\sum_{k=1}^\infty\int_{U_k}\ftf\Big(1-\sqrt{1-\nu_3^2}\Big)\, d\H^2\\
        &=\int_{U\cap\M}\ftf\Big(1-\sqrt{1-\nu_3^2}\Big)\, d\H^2
        \, .
    \end{align*}
\end{proof}

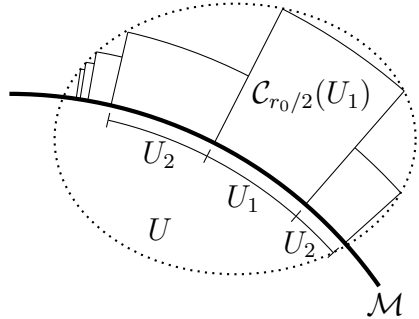
\begin{figure}[H]
    \centering
    \begin{tikzpicture}[scale=2]
            \draw [ultra thick] (90:3) arc (90:35:3);
            \draw [thick, dotted] (1.5,2.7) ellipse (1.2 and 0.9);
            \draw (63:4) arc (63:49:4);
            \draw (77:3.5) arc (77:63:3.5);
            \draw (49:3.5) arc (49:42:3.5);
            \draw (80:3.33) arc (80:77:3.33);
            \draw (81:3.25) arc (81:80:3.25);
            \draw (81.5:3.2) arc (81.5:81:3.2);

            \draw (63:2.9) arc (63:49:2.9);
            \draw [|-|] (77:2.9) arc (77:63:2.9);
            \draw [|-|] (49:2.9) arc (49:42:2.9);

            \draw (56:2.78) node {$U_1$};
            \draw (46:2.78) node {$U_2$};
            \draw (69:2.78) node {$U_2$};
            \draw (56:3.6) node {$\C_{r_0/2}(U_1)$};

            \draw (63:3) -- (63:4);
            \draw (49:3) -- (49:4);
            \draw (77:3) -- (77:3.5);
            \draw (42:3) -- (42:3.5);
            \draw (80:3) -- (80:3.33);
            \draw (81:3) -- (81:3.25);
            \draw (81.5:3) -- (81.5:3.2);

            \draw (2.5,1.6) node {$\M$};
            \draw (1,2.1) node {$U$};
        \end{tikzpicture}
    \caption{Cross section of cones from $U_k$ for $k=1,\dots,5$.}
    \label{fig:Uk}
\end{figure}

\section{The Upper Bound}\label{sec:upper}
This section is dedicated to proving Theorem \ref{thm:upper}. 
To this goal, we first construct a recovery sequence of admissible functions which attains the limiting energy from the lower bound. 
We define the ``optimal'' boundary value $Q_b^*=v^*\otimes v^*-\frac{1}{3}I$, where
\begin{equation}\label{def:optimal}
    v^*=\frac{e_3-\nu_3\nu}{|e_3-\nu_3\nu|},\quad \nu\neq\pm e_3
    \, .
\end{equation}
This is optimal in the sense that if $\nu(\omega)\neq\pm e_3$, then
\begin{equation*}
    \ftf(1-|n_3(Q_b(\omega))|)\geq \ftf (1-|n_3(Q_b^*(\omega))|)
\end{equation*}
for all $Q_b$ which are uniaxial and satisfy
\begin{equation*}
    n(Q_b(\omega))\cdot\nu(\omega)=0
    \, .
\end{equation*}
We can also see that this boundary data is consistent with the lower bound from Theorem~\ref{thm:lower} since $|n_3(Q_b^*)|=\sqrt{1-\nu_3^2}$. 
Despite the fact that $Q_b^*$ is optimal {at each point $\omega\in\M$} in the limit $\xi,\eta\to 0$, we cannot use this fixed boundary data to construct a recovery sequence, as it would require the construction of line defects that are energetically too costly. 
The energy of the resulting upper bound could be larger than the lower bound. 
Instead, we can take advantage of the degree of freedom in our boundary condition to replace the potential line defects by point defects which have a negligible energetic cost. 
With this construction we are able to achieve the same energy as in the lower bound.

\answ{We construct the recovery sequence in two steps: first we define a family of vector fields on $\M$ which will serve as our boundary data, then we extend the boundary data to the exterior region $\Omega$ in such a way to achieve close-to-optimal cost. 
The first step is the goal of Lemma \ref{bdry-data}, while the second step is addressed in Proposition \ref{recov}.

Before stating the first lemma, we introduce some notation. Let
\begin{align*}
    \M^+:=\{\omega\in\M:\nu_3(\omega)>0\}\, ,\\
    \M^0:=\{\omega\in\M:\nu_3(\omega)=0\}
    \, ,\\
    \M^-:=\{\omega\in\M:\nu_3(\omega)<0\}
    \, ,
\end{align*}
then we will construct a vector field $v^\d$ (which is close to $v^*$ but avoids line defects) on $\M^+\cup\M^0$ and the construction on the lower half $\M^-$ will follow from an analogous argument.}

\answ{
\begin{lemma}\label{bdry-data}
    Let $\M$ be a $C^{2,1}$ connected, compact, oriented 2-manifold without boundary. 
    Then for $\d>0$ sufficiently small, there exists a family of unit tangent vector fields $v^\d$ defined on $\M^+\cup\M^0$ such that $v^\d=e_3$ on $\M^0$ and for any sequence $\d_k\to 0$
    \begin{equation*}
        v_3^{\d_k}\to\sqrt{1-\nu_3^2}\quad \text{$\H^2$-a.e. on $\M^+\cup\M^0$ as $k\to \infty$.}
    \end{equation*}
    Moreover, $v^\d$ can be chosen to be Lipschitz away from finitely many isolated point discontinuities $\{p_j^\d\}_{j=1}^{N_\d}\subset \M^+$ and such that for any $\omega\in \M^+$ in a sufficiently small neighbourhood around each $p_j^\d$ it holds that
    \begin{equation*}
        |\grad_\omega v^\d(\omega)|\leq\frac{C_\d}{\dist_\M(\omega,p_j^\d)}
    \end{equation*}
    for some constant $C_\d>0$. 
\end{lemma}}
\begin{proof}
    \answ{We begin by defining the set $\M_d$ to be 
    \begin{equation*}
        \M_d:=\{\omega\in\M^+:\nu(\omega)=e_3\}
        \, .
    \end{equation*}
    This is precisely the set where we cannot uniquely define an ``optimal'' boundary value.} Using that $\M_d$ is compact, we can cover it with finitely many open balls in \blue{$\M^+$} of radius $\d>0$ for $\d$ \blue{sufficiently} small.
    The union of these balls forms $I^\d-$many \blue{connected components} which we will refer to as $\cR_i^\d$ for $i=1,\dots,I^\d$ \blue{and for simplicity of notation, we define
    \begin{equation}\label{def:Dd}
        \Dd:=\bigcup_{i=1}^{I^\d}\cR_i^\d
        \, ,
    \end{equation}
    so $\Dd$ is an open set which covers $\M_d$ such that $0<\dist_\M(\partial \Dd,\M_d)\leq\d$.}
    
    \answ{The idea behind this construction is that outside of $\Dd$, we can define $v^\d$ by a projection of $e_3$ onto the tangent plane at each point (as in \eqref{def:optimal}), while inside $\Dd$, we can artificially place point defects in each $\cR_i^\d$. 
    Finally, we extend $v^\d$ via a Lipschitz extension from small balls around each defect to the region outside of $\Dd$.}
    
    On the set, \blue{$\M^+\cup \M^0\setminus\Dd$}, we can define $\vd=v^*$, where $v^*$ is the optimal boundary condition defined in \eqref{def:optimal}, so it remains to define $\vd$ on each $\cR_i^\d$.
    
    Take any arbitrary $\cR_i^\d$, for some $i=1,\dots, I^\d$, then since $v^*$ is continuous \blue{on $\M^+\setminus\M_d$}, in particular $v^*$ is continuous on $\partial \cR_i^\d$ \answ{and also of norm $1$} and thus $\deg(v^*;\partial \cR_i^\d)=d_i$ for some $d_i\in \mathbb{Z}$, \answ{where the degree is the sum of the degrees on each component of $\partial\cR_i^\d$ taking into account the orientation induced by $\cR_i^\d$, see \cite[p. 121ff]{hirsch}}. 
    We want to balance out the topological charge around $\partial\cR_i^\d$ by artificially inserting sufficiently many point defects in the interior of $\cR_i^\d$.
    So if $d_i=0$, then no defects are needed. 
    For $d_i\neq 0$, choose $|d_i|-$many points $p^i_j$ in $\cR_i^\d$, for $j=1\dots,|d_i|$, such that $\dist_\M(p^i_j,\partial \cR_i^\d)>0$.
    Then there exists $q_i\in\mathbb{N}$ such that 
    \begin{equation*}
        \dist_\M\left(B\left(p^i_j,\frac{\d}{q_i}\right),\partial \cR_i^\d\right)>0\quad\text{and}\quad\dist_\M\left(B\left(p^i_j,\frac{\d}{q_i}\right), B\left(p^i_k,\frac{\d}{q_i}\right)\right)>0
        \, ,
    \end{equation*}
    for all $j,k=1,\dots,|d_i|$.
    For simplicity, we let $\d_i':=\d/q_i$ \blue{and
    \begin{equation*}
        \P_i:=\bigcup_{j=1}^{|d_i|}B(p_j^i,\d_i')
        \, .
    \end{equation*}}
    On each ball $\overline{B(p^i_j,\d_i')}$ we define a point defect of degree $\sgn(d_i)$ as follows. \answ{For all $\tau=(\tau_1,\tau_2)\in T_{p_j^i}\M$, let} $\widetilde m:T_{p_j^i}\M\to\R^3$ be
    \begin{equation*}
        \widetilde{m}(\tau)=(\tau_1,\sgn(d_i)\tau_2,0)
        \, ,
    \end{equation*}
    so \blue{$\widetilde{m}/|\widetilde{m}|$} could be either a $+1$ defect or a $-1$ defect.
    We then \answ{define a new map $m$ on $\overline{B(p_j^i,\d_i')}$ by the pullback $(\d_i^{-1}\exp_{p^i_j}^{-1})^*\widetilde{m}$} so that $m:\overline{B(p^i_j,\d_i')}\to\R^3$ \answ{is given by}
    \begin{equation*}
        m(\omega)=\widetilde m\left(\frac{1}{\d_i'}\exp_{p^i_j}^{-1}(\omega)\right)
        \, .
    \end{equation*}
    \answ{We remark that $m$ is well-defined since the exponential map is a bi-Lipschitz homeomorphism \cite{biLip}. As well, we note that }$m(\omega)$ \blue{is not necessarily in $T_\omega\M$ for $\omega\neq p_j^i$}, \answ{but we are constructing a unit tangent vector field}, so to fix this, we project $m$ onto the tangent space and re-scale to unit length by defining $\hat{u}:\overline{B(p^i_j,\d_i')}\answ{\setminus\{p_j^i\}}\to\S^2$ to be
    \begin{equation}\label{hat-u}
        \hat{u}(\omega)=\frac{m(\omega)-(m(\omega)\cdot\nu(\omega))\nu(\omega)}{|m(\omega)-(m(\omega)\cdot\nu(\omega))\nu(\omega)|}
        \, .
    \end{equation}
    \answ{This projection is done on a small region around \answ{the defect point }$p_j^i$, so it is well-defined since $m(\omega)\neq \nu(\omega)$ for any $\omega\in \overline{B(p_j^i,\d_i')}\setminus\{p_j^i\}$. Therefore we let $v^\d=\hat{u}$ on $\overline{B(p_j^i,\d_i')}\setminus\{p_j^i\}$ for each $i=1,\dots, I^\d$ and $j=1,\dots, |d_i|$.}
    
    \answ{So far we have defined $v^\d$ on $\M^+\cup\M^0\setminus\cR^\d$ and we have defined $v^\d$ on $\P_i$ for $i=1,\dots, I^\d$. It remains to define $v^\d$ on $\cR_i^\d\setminus\P_i$ for $i=1\dots, I^\d$ and this will be done by a Lipschitz extension as follows.}
    
    \answ{On $\partial B(p^i_j,\d_i')$ we have} $\vd=\hat{u}$, so
    \begin{equation*}
        \deg(\vd; \partial B(p^i_j,\d_i'))
        =
        \sgn(d_i)
        \, .
    \end{equation*}
    We also \blue{have} $\vd=v^*$ on $\partial \cR_i^\d$ and therefore, by construction,
    \begin{align*}
        \deg\left(\vd;\partial\left(\cR_i^\d\setminus\P_i\right)\right)=\blue{d_i-|d_i|\sgn(d_i)}=0
        \, ,
    \end{align*}
    see also Figure \ref{fig:enter-label}.
    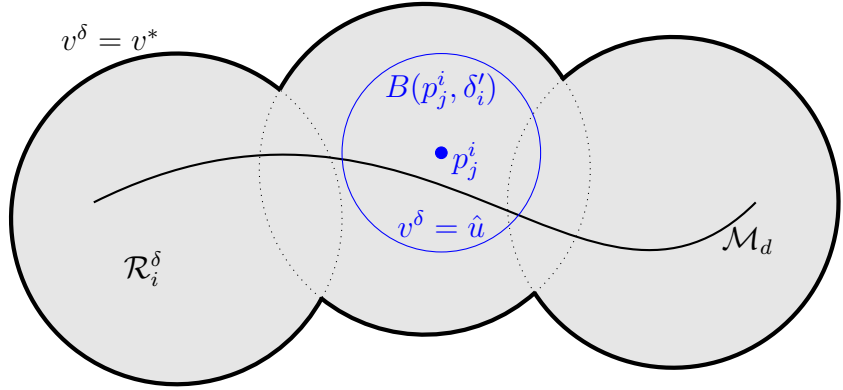
\begin{figure}
        \centering
        \begin{tikzpicture}[scale=2.2]
        \draw[fill=gray!20!white, draw=gray!20!white] (-1.5,-0.1) circle (1);
        \draw[fill=gray!20!white, draw=gray!20!white] (0,0.2) circle (1);
        \draw[fill=gray!20!white, draw=gray!20!white] (1.5,0) circle (1);
        
        \draw [thick](-2,0) .. controls (0,1) and (1,-1) .. (2,0);
        
        \draw [blue] (0.1,0.3) circle (0.6);
        \draw [blue, fill] (0.1,0.3) circle (1pt);
        
        \draw[dotted] (0,0.2) circle (1);
        \draw[dotted] (-1.5,-0.1) circle (1);
        \draw[dotted] (1.5,0) circle (1);

        \draw [ultra thick] ({0+cos(33.5)},{0.2+sin(33.5)}) arc (33.5:151:1);
        \draw [ultra thick] ({0+cos(-48)},{0.2+sin(-48)}) arc (-48:-129:1);
        \draw [ultra thick] ({-1.5+cos(51)},{-0.1+sin(51)}) arc (51:331.5:1);
        \draw [ultra thick] ({1.5+cos(-147)},{sin(-147)}) arc (-147:132.3:1);
        \draw (1.95,-0.23) node {$\M_d$};
        \draw (-1.7,-0.4) node {$\mathcal{R}_i^\delta$};
        \draw[blue] (0.1,-0.13) node {$v^\delta=\hat{u}$};
        \draw[blue] (0.25,0.25) node {$p_j^i$};
        \draw (-1.9,1) node {$v^\d=v^*$};
        \draw[blue] (0.1,0.68) node {$B(p_j^i,\delta'_i)$};
    \end{tikzpicture}
        \caption{Schematic drawing of the region $\cR_i^\d \subset\M$ around a section of $\M_d$ where the degree of $v^*$ around $\partial \cR_i^\d$ is $\pm1$ and one point singularity of matching degree has been placed at $p_j^i$.}
        \label{fig:enter-label}
    \end{figure}
    As the patches $\cR_i^\d\setminus\blue{\P_i}$ are connected we can use Theorem 1.8 from \cite[p. 126]{hirsch} to construct
    a continuous\blue{, but not necessarily Lipschitz} extension $v_0:\overline{\cR_i^\d\setminus\blue{\P_i}}\to\S^2$ which agrees with $\vd$ on the boundary \answ{and remains tangent at each point in $\cR_i^\d\setminus\P_i$. We remark that this tangency condition of the extension makes it necessary that the degree around the boundary be zero in order for $v_0$ to exist}.
    Further for {small} $\e>0$, one can adapt Theorem 10.16  from \cite{lee} to the case of a \answ{$C^{2,1}-$}manifold to show that there exists a Lipschitz function $v_\e:\overline{\cR_i^{\d}\setminus\blue{\P_i}}\to \R^3$ such that
    \begin{equation*}
        |v_0(\omega)-v_\e(\omega)|<\e,\quad\text{for all}\quad \omega\in \cR_i^{\d}\setminus\blue{\P_i}
    \end{equation*}
    and $v_\e$ agrees with $v^*$ on $\partial(\cR_i^\d)$ as well as $\hat{u}$ on \blue{$\partial\P_i$}. Since $v_\e\in\R^3$, and need not be in $T_\omega\M$, we again must  project onto the tangent space and re-normalize the vector field, so we define
    \begin{equation*}
        \vd(\omega)=\frac{v_\e(\omega)-(v_\e(\omega)\cdot\nu(\omega))\nu(\omega)}{|v_\e(\omega)-(v_\e(\omega)\cdot\nu(\omega))\nu(\omega)|},\quad\text{for}\quad \blue{\omega\in \cR_i^\d\setminus\P_i}
        \, .
    \end{equation*}
    \answ{This projection is well-defined and Lipschitz since $v_\e$ is uniformly close to $v_0$ and $v_0$ is orthogonal to $\nu$, so the denominator is bounded away from zero.}
    
    Now that $\vd$ has been chosen on all of $\M^+$ \answ{it remains to check that $v^\d$ has the desired properties. Consider the region $\M^+\cup\M^0\setminus\Dd$, then here we have
    \begin{equation*}
        v^\d=v^*=\frac{e_3-\nu_3\nu}{|e_3-\nu_3\nu|}=\frac{e_3-\nu_3\nu}{\sqrt{1-\nu_3^2}}
    \end{equation*}
    so taking a partial derivative in local coordinates $(\omega_1,\omega_2)$ for $k=1,2$ we have
    \begin{equation*}
        \frac{\partial v^\d}{\partial \omega_k}=\frac{-\left(\frac{\partial \nu_3}{\partial \omega_k}\nu+\nu_3\frac{\partial \nu}{\partial \omega_k}\right)\sqrt{1-\nu_3^2}-(e_3-\nu_3\nu)(1-\nu_3^2)^{-1/2}\nu_3\frac{\partial\nu_3}{\partial \omega_k}}{1-\nu_3^2}
        \, .
    \end{equation*}
    Therefore
    \begin{equation*}
        \norm{\frac{\partial v^\d}{\partial \omega_k}}=\frac{1}{\sqrt{1-\nu_3^2}}\left(\norm{\frac{\partial \nu_3}{\partial \omega_k}}+|\nu_3|\norm{\frac{\partial \nu}{\partial \omega_k}}\right)+\frac{|\nu_3|}{1-\nu_3^2}\norm{\frac{\partial \nu_3}{
        \partial \omega_k}}
        \, ,
    \end{equation*}
    and using that $\dist(M^+\cup M^0\setminus\Dd,\M_d)>0$, we have that there exists a constant $C_\d>0$ such that
    \begin{equation*}
        \frac{1}{\sqrt{1-\nu_3^2}}\leq C_\d\quad\text{and}\quad\frac{\nu_3}{1-\nu_3^2}\leq C_\d
        \, .
    \end{equation*}
    Thus we obtain the following bounds:
    \begin{equation*}
        \norm{\frac{\partial v^\d}{\partial \omega_k}}\leq 3C_\d\k\quad\text{and}\quad|\grad_\omega v^\d|\leq 3\sqrt{2}C_\d\k
        \, ,
    \end{equation*}
    so $v^\d$ is Lipschitz on $\M^+\cup \M^0\setminus\Dd$ and we also note that by construction, $v^\d$ is Lipschitz on $\cR_i^\d\setminus\P_i$ for each $i=1,\dots, I^\d$. Now on each $B(p_j^i,\d_i')$ we have
    \begin{equation*}
        v^\d(\omega)=\hat{u}(\omega)=\frac{m(\omega)-(m(\omega)\cdot\nu(\omega))\nu(\omega)}{|m(\omega)-(m(\omega)\cdot\nu(\omega))\nu(\omega)|}
        \, .
    \end{equation*}
    so taking a partial derivative in $\omega_i$} we can see that
    \begin{align*}
        \frac{\partial \hat{u}}{\partial \omega_k}=\frac{\Big(\frac{\partial }{\partial \omega_k}(m-(m\cdot\nu)\nu)\Big)}{|m-(m\cdot\nu)\nu|}-\frac{(m-(m\cdot\nu)\nu)\Big((m-(m\cdot \nu)\nu)\cdot\big(\frac{\partial}{\partial \omega_k}(m-(m\cdot \nu)\nu)\big)\Big)}{|m-(m\cdot\nu)\nu|^\answ{3}}
        \, .
    \end{align*}
    Thus,
    \begin{align*}
        \left|\frac{\partial \hat{u}}{\partial \omega_k}\right|
        \leq
        \frac{2|\frac{\partial}{\partial \omega_k}(m-(m\cdot\nu)\nu)|}{|m-(m\cdot\nu)\nu|}
        \leq
        \frac{C(|\frac{\partial m}{\partial \omega_k}|+|\frac{\partial \nu}{\partial \omega_k}|)}{|m|-|m\cdot\nu|}
        =
        \frac{C(|\frac{\partial m}{\partial \omega_k}|+\k)}{|m|(1-|\cos\theta|)}
        \, ,
    \end{align*}
    where $\theta$ is the angle between $m$ and $\nu$. Since $\nu$ is close to $e_3$ and $m\cdot e_3=0$, there exists a constant $C_\d>0$ such that $1-\cos\theta\geq C_\d$ for all $\omega\in B(p_j^i,\d'_i)$. So,
    \begin{equation*}
        \norm{\frac{\partial \hat{u}}{\partial \omega_k}}\leq \frac{C_\d}{|m|}\left(\norm{\frac{\partial m}{\partial \omega_k}}+\k\right)\leq \frac{C_\d}{|m|}(|\grad_\omega m|+\k)
        \, .
    \end{equation*}
    Next we have by chain rule that
    \begin{equation*}
        |\grad_\omega m|^2\leq|\grad_\t \widetilde{m}|^2|\grad_\omega \t|^2\leq \frac{C}{(\d'_i)^2}|\grad_\omega \exp_{p_j^i}^{-1}(\omega)|^2\leq\frac{C}{(\d'_i)^2}
        \, ,
    \end{equation*}
    \answ{where we use that the exponential map on a $C^{2,1}$ manifold is bi-Lipschitz \cite{biLip} }and therefore,
    \begin{equation*}
        \norm{\frac{\partial \hat{u}}{\partial \omega_k}}\leq \frac{C_\d}{|m|}\left(\frac{C}{\d'_i}+\k\right)\leq\frac{C_\d}{\d'_i|m|}\quad\text{and}\quad|\grad_\omega \hat{u}|^2\leq\frac{C_\d}{(\d'_i)^2|m|^2}
    \end{equation*}
    for some constant $C_\d>0$. Using the relation $\d'_i|m|=\dist_\M(\omega,p_j^i)$,
    \begin{equation}\label{bound:om1:tan}
        |\grad_\omega\hat{u}|^2\leq\frac{C_\d}{\dist_\M^2(\omega,p_j^i)}
        \, .
    \end{equation}
    \answ{It follows from this inequality that $v^\d$ is Lipschitz on the region $B(p_j^i,\d_i')\setminus B(p_j^i,r)$ for any $0<r<\d_i'$, so $v^\d$ is in fact Lipschitz away from the points $\{p_j^i\}$ as required.
    
    All that remains is to show that for any sequence $\d_k\to 0$, we have that $v^{\d_k}_3\to\sqrt{1-\nu_3^2}$ $\H^2$-a.e.\ as $k\to \infty$. We define
    \begin{equation*}
        \fP:=\bigcup_{k=1}^\infty\bigcup_{i=1}^{I^{\d_k}}\bigcup_{j=1}^{|d_i|}\{p_j^i\}
        \, ,
    \end{equation*}
    then $\fP$ is a countable union of sets of $\H^2$-measure zero, so $\H^2(\fP)=0$. We will show that we have pointwise convergence on $\M^+\cup\M^0\setminus\fP$. By definition, $v_3^*=\sqrt{1-\nu_3^2}$, so if $\omega\in\M^+\cup\M^0\setminus (\cR^{\d_k}\cup\fP)$, we already have that $v^{\d_k}_3=\sqrt{1-\nu_3^2}$. From \eqref{def:Dd} we have 
    \begin{equation*}
        0<\dist_{\M}(\partial \cR^{\d_k},\M_d)\leq \d_k
        \, ,
    \end{equation*}
    so if $\omega\in\M^+\cup\M^0\setminus(\M_d\cup\fP)$, then for sufficiently large $k$, we also have $\omega\not\in \cR^{\d_k}$, which means we are in the previous case where $v_3^{\d_k}=\sqrt{1-\nu_3^2}$. The final case is if $\omega\in\M_d\setminus\fP$, where 
    \begin{equation*}
        v^{\d_k}_3=0\quad\text{and}\quad \sqrt{1-\nu_3^2}=0\quad\text{on $\M_d\setminus\fP$}
        \, ,
    \end{equation*}
    since $v^{\d_k}$ is in the tangent space. Thus $v_3^{\d_k}=\sqrt{1-\nu_3^2}$ on $\M_d\setminus\fP$ and
    \begin{equation*}
        v^{\d_k}_3\to\sqrt{1-\nu_3^2}\quad \text{on $\M^+\cup\M^0\setminus\fP$.}
    \end{equation*}}
\end{proof}
\answ{A key element in the construction of the recovery sequence in Proposition \ref{recov} is the optimal profile on a ray which is discussed in \cite{abl}. 
The following lemma is a restatement of relevant parts of Lemma 3.4 from \cite{abl} in the notation of this paper as well as some properties which are useful for our estimates.
\begin{lemma}\label{lemma:optimal}
    Let $n^*:(0,\infty)\times (0,\pi)\to\S^2$ be defined by $n^*=(\sqrt{1-(n_3^*)^2},0,n_3^*)$ where
    \begin{equation*}
        n_3^*(t,\varphi)=\frac{A(\varphi)-e^{-\ftf t}}{A(\varphi)+e^{-\ftf t}}\quad\text{and}\quad A(\varphi)=\frac{1+\cos\varphi}{1-\cos\varphi}
        \, ,
    \end{equation*}
    then $n^*(0,\varphi)=(\sin\varphi,0,\cos\varphi)$ and $\lim_{t\to\infty}n^*(t,\varphi)=e_3$. Moreover,
    \begin{equation*}
        \int_0^\infty\norm{\frac{\partial n^*}{\partial t}}^2+g(n^*)\, dt=\ftf(1-\cos\varphi)\quad\text{and}\quad\norm{\frac{\partial n^*}{\partial t}}^2,\norm{\frac{\partial n^*}{\partial \varphi}}^2,\norm{n_1^*}\leq Ce^{-\ftf t}
        \, .
    \end{equation*}
\end{lemma}}
\answ{With these lemmas in hand, we are now ready to construct a sequence which nearly attains the optimal upper bound using a fixed boundary condition. 
This is achieved in the following Proposition~\ref{recov} in which we construct a competitor on the exterior region around $\M^+\cup\M^0$ with the intention of using it for the boundary condition $v^\d$ for fixed $\d>0$ as constructed in Lemma \ref{bdry-data}. 
}

\answ{
\begin{proposition}\label{recov}
    Let $v$ be a unit tangent vector field on $\M^+\cup \M^0$ which is Lipschitz away from finitely many isolated point discontinuities $\{p_j\}_{j=1}^N\subset \M^+$ and is such that for any $\omega\in \M^+$ in a sufficiently small neighbourhood around every $p_j$
    \begin{equation}\label{assumpt-bound}
        |\grad_\omega v(\omega)|\leq\frac{C}{\dist_\M(\omega,p_j)}
    \end{equation}
    for some constant $C>0$. Moreover, suppose $v=e_3$ on $\M^0$ and that
    \begin{equation*}
        \frac{\eta}{\xi}\to\infty
        \quad\text{as}\quad\xi,\eta\to0
        \, ,
    \end{equation*}
    then, for any $H\gg 1$ large enough, there exists a sequence of maps $\Qc\in \A$ with $\Qc=v\otimes v-\frac{1}{3}I$ on $\M^+\cup \M^0$ and $\Qc=Q_\infty$ on $\C_{2H\eta}(\M^0)\cup\partial \C_{2H\eta}(\M^+)\setminus\M^+$ such that
    \begin{equation*}
        \limsup_{\xi,\eta\to0}\:\eta\exe(Q_{\xi,\eta}^H;\C_{2H\eta}(\M^+\cup\M^0))
        \leq
        \int_{\M^+}\ftf(1-v_3)\, d\H^2 + \Psi(H)
        \, ,
    \end{equation*}
    where $\Psi(h)\to 0$ as $h\to\infty$.
\end{proposition}}

\answ{We remark that in the statement of the proposition, the upper bound has $v_3$ in the integrand as opposed to $|v_3|$ which is inconsistent with the statement of Theorem~\ref{thm:upper} and is not optimal at locations where $v_3<0$. 
However, removing the absolute value allows for a simpler construction and in the proof of Theorem \ref{thm:upper}, when we consider the boundary data $v^\d$, we can use the $\H^2-$pointwise almost everywhere convergence of $v_3^\d\to\sqrt{1-\nu_3^2}$ to obtain the same result as with absolute values. 
Another important remark is that the integral is taken over $\M^+$ while the construction is over $\C_{2H\eta}(\M^+\cup \M^0)$. 
This follows from the condition that $\Qc=Q_\infty$ on $C_{2H\eta}(\M^0)$, so the energy in this region is zero.}

\begin{proof}
     \answ{First, let $\Qc=Q_\infty$ on $\C_{2H\eta}(\M^0)$, then we will divide} the exterior region $\C_{2H\eta}(\M^+)$ \blue{by defining five subregions}. 
     We introduce
    \begin{equation*}
        \P=\answ{\bigcup_{j=1}^{N}\{p_j\}}
    \end{equation*}
    and we denote by $\P_\eta$ the $\eta-$neighbourhood of $\P$. Note that since the $p_j$ are isolated, for sufficiently small $\eta$, $\P_\eta$ is a disjoint collection of open balls $B(p_j,\eta)\subset \M^+$.

    We split $\C_{2H\eta}(\M^+)$ into several subregions $\Omega_i$ for $i=1,...,5$, see also Figure \ref{fig:construction},
    \blue{and} we carry out the constructions and energy estimates of $\Qc$ on each of those subregion individually. 
    We first define a unit vector field $n$ on each subregion satisfying the boundary condition (i.e.\ $n(0,\omega)=v(\omega)$ on \blue{$\M^+\cup \M^0$}) and then pose $\Qc=n\otimes n-\frac{1}{3}I$. 
    The first subregion we consider is the set where we expect the energy to concentrate, i.e.\ a boundary layer away from potential defects. 
    \begin{figure}
        \centering
        \begin{tikzpicture}[scale=1.8]
            \draw [ultra thick] (-90:5) arc (-90:-45:5);
            \draw [ultra thick] (-90:5) arc (-90:-135:5);
            \draw (-130:3.5) arc (-130:-50:3.5);
            
            \draw [fill] (-90:5) circle (2pt);

            \draw (-82:5) -- (-82:3.5);
            \draw (-74:5) -- (-74:3.5);
            \draw (-98:5) -- (-98:3.5);
            \draw (-106:5) -- (-106:3.5);

            \draw (-98:4.4) arc (-98:-82:4.4);
            \draw (-106:4.25) arc (-106:-132.5:4.25);
            \draw (-74:4.25) arc (-74:-47.5:4.25);

            \draw (-120:4.625) node {$\Om_1$};
            \draw (-60:4.625) node {$\Om_1$};
            \draw (-120:3.875) node {$\Om_2$};
            \draw (-60:3.875) node {$\Om_2$};
            \draw (-90:4.7) node {$\Om_3$};
            \draw (-90:4) node {$\Om_4$};
            \draw (-78:4) node {$\Om_5$};
            \draw (-102:4) node {$\Om_5$};
            
            \draw (-86:5.3) node {$\eta$};
            \draw (-78:5.3) node {$\eta$};
            \draw (-102.5:5.3) node {$\d'$};
            \draw (-107.5:5.65) node {$\d$};
            \draw (-79.75:4.7) node {$\eta$};
            \draw (-70.5:4.625) node {$H\eta$};
            \draw (-70.5:3.875) node {$H\eta$};
            
            \draw [<->] (-90:5.15) arc (-90:-82:5.15);
            \draw [<->] (-82:5.15) arc (-82:-74:5.15);
            \draw [<->] (-90:5.15) arc (-90:-115:5.15);
            \draw [<->] (-90:5.5) arc (-90:-125:5.5);

            \draw [<->] (-73:5) -- (-73:4.25);
            \draw [<->] (-73:4.25) -- (-73:3.5);
            \draw [<->] (-81:5) -- (-81:4.4);

            \draw (-49:5.2) node {$\M$};
        \end{tikzpicture}
        \caption{An example of a cross-section of the subregions of $C_{2H\eta}(\M^+)$ around a point defect \answ{(bold dot)} of degree +1 where $\M$ has negative mean curvature nearby the defect.}
        \label{fig:construction}
    \end{figure}
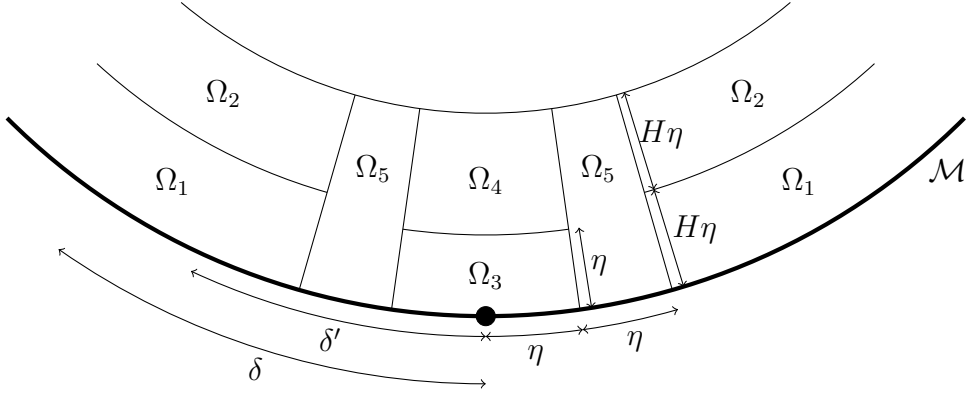
    
    In this layer we would like to take the director $n$ to be the optimal profile from \cite{abl} \answ{(see Lemma~\ref{lemma:optimal})}, but since {that} transition takes place over an infinite length, the optimal profile needs to be modified similar to \cite{ACS2024} to fit into a finite length while preserving the energy bound.
    Therefore, we introduce two layers:
    \begin{equation*}
        \Omega_1
        \ := \
        \C_{H\eta}(\M^+\setminus \P_{2\eta})
        \, ,
    \end{equation*}
    in which we take the optimal profile, and a corresponding outer region
    \begin{equation*}
        \Omega_2
        \ := \ 
        \C_{2H\eta}(\M^+\setminus\P_{2\eta})\setminus\Omega_1
        \, ,
    \end{equation*}
    in which we interpolate to make the transition to $Q_\infty$. 
    This second transition layer will have asymptotically negligible energetic cost. 
    Next, we have the subregion where point defects occur, given by
    \begin{equation*}
        \Om_3:=\C_\eta\left(\P_\eta\right)
        \, .
    \end{equation*}
    In each of the connected components of $\Om_3$ we define $n$ to be a map representing one half of a point defect of degree $+1$ or $-1$ centered at each \answ{$p_j\in\P$}. 
    Next we consider the region
    \begin{equation*}
        \Om_4:=\C_{2H\eta}(\P_\eta)\setminus\Om_3
        \, ,
    \end{equation*}
    where we interpolate radially to $Q_\infty$, similar to $\Om_2$. Finally we consider,
    \begin{equation*}
        \Om_5:=\C_{2H\eta}\left(\P_{2\eta}\right)\setminus \Om_3
        \, ,
    \end{equation*}
    where we define $n$ via a Lipschitz extension from the boundary of the region. 
    Note that in the spherical case presented in \cite{bls}, the analogues of $\Om_4$ and $\Om_5$ are combined into one region, however due to the more complicated geometry and defects which arise in this general setting, the splitting of these regions greatly simplifies the energy estimates. 
    We will see that the energy in all regions, excluding $\Om_1$, is negligible in the limit.

    \bigskip

    \noindent\underline{\textit{Energy in $\Omega_1$}}:
    Recall that $\Om_1=\C_{H\eta}(\M^+\setminus\P_{2\eta})$.
    On this set, we define $n$ by
    \begin{equation*}
        n(r,\omega)=\left(\frac{v_1(\omega)\sqrt{1-(n_3^*(\frac{r}{\eta},\varphi(\omega)))^2}}{\sqrt{1-(v_3(\omega))^2}},\frac{v_2(\omega)\sqrt{1-(n_3^*(\frac{r}{\eta},\varphi(\omega)))^2}}{\sqrt{1-(v_3(\omega))^2}},n_3^*\Big(\frac{r}{\eta},\varphi(\omega)\Big)\right)
        \, ,
    \end{equation*}
    where $n_3^*$ comes from the optimal profile along a ray defined in \answ{Lemma~\ref{lemma:optimal} (coming from} \cite{abl}) and $\varphi(\omega)=\cos^{-1}(v_3(\omega))$.
    We note that $n_3^*(0,\varphi(\omega))=v_3(\omega)$, so $n(0,\omega)=v(\omega)$. The energy in this region is given by
    \begin{equation*}
        \eta \exe(\Qc;\Om_1)=\int_{\M^+\setminus\P_{2\eta}}\int_0^{H\eta}\left(\eta|\grad n|^2+\frac{1}{\eta}g(n)\right)\prod_{i=1}^2(1+r\k_i)\, drd \H^2
        \, .
    \end{equation*}
    The gradient $|\grad n|^2$ can be written as
    \begin{equation*}
        |\grad n|^2=\norm{\frac{\partial n}{\partial r}}^2+\frac{1}{(1+r|\k_1|)^2}\norm{\frac{\partial n}{\partial \omega_1}}^2+\frac{1}{(1+r|\k_2|)^2}\norm{\frac{\partial n}{\partial \omega_2}}^2\leq \norm{\frac{\partial n}{\partial r}}^2+|\grad_\omega n|^2
        \, ,
    \end{equation*}
    so it suffices to estimate the radial and tangential derivatives of $n$ separately. First,
    \begin{equation}\label{bound:om1:rad}
        \norm{\frac{\partial n}{\partial r}}^2=\norm{\frac{\partial}{\partial r}n^*\Big(\frac{r}{\eta},\omega\Big)}^2=\frac{1}{\eta^2}\norm{\frac{\partial n^*}{\partial \rt}}^2
    \end{equation}
    for $r=\eta \rt$.
    Next we consider the tangential derivative
    \begin{align*}
        \norm{\frac{\partial n}{\partial \omega_i}}^2&=\left(\norm{\frac{\partial v_1}{\partial \omega_i}}^2+\norm{\frac{\partial v_2}{\partial \omega_i}}^2\right)\left(\frac{1-(n_3^*)^2}{1-(v_3)^2}\right)+\norm{\frac{\partial n}{\partial \varphi}}^2\norm{\frac{\partial \varphi}{\partial \omega_i}}^2-\frac{(v_3)^2(1-(n_3^*)^2}{(1-(v_3)^2)^2}\norm{\frac{\partial v_3}{\partial \omega_i}}^2\\
        &\leq\norm{\frac{\partial v}{\partial \omega_i}}^2+\frac{C}{1-(v_3)^2}\norm{\frac{\partial v_3}{\partial \omega_i}}^2\leq C\norm{\frac{\partial v}{\partial \omega_i}}^2
        \, ,
    \end{align*}
    where the last inequality is obtained from \eqref{bound:n3}. Therefore, $|\grad_\omega n|^2\leq C|\grad_\omega v|^2$, so it suffices to estimate the tangential derivatives of $v$. To do this, we consider two different subregions of $\M^+\setminus \P_{2\eta}$. \answ{Recall that by assumption \eqref{assumpt-bound} around each $p_j$ there exists a neighbourhood on which
    \begin{equation}\label{bound-defect}
        |\grad_\omega v(\omega)|\leq\frac{C}{\dist_\M(\omega,p_j)}
    \end{equation}
    and since there are finitely many $p_j$, there exists $\e>0$ sufficiently small such that \eqref{bound-defect} holds on $B(p_j,\e)$ and such that $B(p_j,\e)\subset \M^+$.
    By assumption, if $\omega\in \M^+\setminus \P_\e$, then $v$ is Lipschitz, so $v\in W^{1,\infty}(\M^+\setminus \P_\e)$, therefore there exists a constant $C>0$ such that $|\grad_\omega v(\omega)|^2\leq C$. If $\omega\in \P_\e\setminus\P_{2\eta}$, then \eqref{bound-defect} holds.}
    We can now estimate the energy in this region.
    \begin{align*}
        \eta \exe(\Qc;\Om_1)&\leq\int_{\M^+\setminus\P_{2\eta}}\int_0^{H\eta}\left(\eta\norm{\frac{\partial n}{\partial r}}^2+\eta|\grad_\omega n|^2+\frac{1}{\eta}g(n)\right)\prod_{i=1}^2(1+r\k_i)\, drd \H ^2(\omega)\\
        &=\int_{\M^+\setminus\P_{2\eta}}\int_0^{H}\left(\norm{\frac{\partial n^*}{\partial \rt}}^2+\eta^2|\grad_\omega n|^2+g(n^*)\right)\prod_{i=1}^2(1+\eta\rt\k_i)\, d\rt d \H ^2(\omega)
        \, .
    \end{align*}
    We then separate the integral into two parts,
    \begin{align*}
        \eta \exe(\Qc;\Om_1)&\leq \int_{\M^+\setminus\P_{2\eta}}\int_0^{H}\left(\norm{\frac{\partial n^*}{\partial \rt}}^2+g(n^*)\right)\prod_{i=1}^2(1+\eta\rt\k_i)\, d\rt d \H ^2(\omega)\\
        &\quad\quad+C\int_{\M^+\setminus\P_{2\eta}}\int_0^{H} \eta^2|\grad_\omega n|^2\, d\rt d\H^2(\omega)
        \, ,
    \end{align*}
    where the first part will tend to the limiting energy as $\eta\to0$ and the second part will be negligible in the limit. By Lebesgue's Dominated Convergence Theorem, the definition of $n^*$, and since $1+\eta\rt\k_i\to 1$ pointwise as $\eta\to0$, it holds
    \begin{align*}
        \limsup_{\xi,\eta\to0}\: \int_{\M^+\setminus\P_{2\eta}}\int_0^{H}&\left(\norm{\frac{\partial n^*}{\partial \rt}}^2+g(n^*)\right)\prod_{i=1}^2(1+\eta\rt\k_i)\, d\rt d \H ^2(\omega)\\
        &= \int_{\M^+}\int_0^H\left(\norm{\frac{\partial n^*}{\partial \rt}}^2+g(n^*)\right)\, d\rt d\H^2(\omega)\leq\int_{\M^+} \ftf(1-\blue{v_3})\, d\H^2(\omega)
        \, .
    \end{align*}
    Now for the second part, using \eqref{bound-defect},
    \begin{align*}
        \int_{\M^+\setminus\P_{2\eta}}\int_0^{H} \eta^2|\grad_\omega n|^2\, d\rt d\H^2(\omega)\leq& \int_{\answ{\M^+\setminus \P_\e}}\int_0^H C\eta^2\ drd\H^2(\omega)\\
        &+\answ{\sum_{j=1}^{N}}\int_{\answ{B(p_j,\e)\setminus B(p_j,2\eta)}}\int_0^H\frac{C\eta^2}{\dist_\M^2(\omega,p_j)}\, drd\H^2(\omega)
        \, .
    \end{align*}
    We can further estimate this energy by
    \begin{align*}
        \int_{\M^+\setminus\P_{2\eta}}\int_0^{H} \eta^2|\grad_\omega n|^2\, d\rt d\H^2(\omega)&\leq CH\eta^2|\M^+|+\sum_{j=1}^{N}CH\eta^2\ln\Big(\frac{\answ{\e}}{2\eta}\Big)\\
        &\leq CH\eta^2 |\M^+|+\sum_{j=1}^{N}CH\eta^2\Big|\ln\frac{2\eta}{\answ{\e}}\Big|
        \, .
    \end{align*}
    Therefore,
    \begin{equation*}
        \limsup_{\xi,\eta\to0}\: \int_{\M^+\setminus\P_{2\eta}}\int_0^{H} \eta^2|\grad_\omega n|^2\, d\rt d\H^2(\omega)=0
        \, ,
    \end{equation*}
    and
    \begin{align*}
        \limsup_{\xi,\eta\to0}\: \eta\exe(\Qc;\Om_1)\leq \int_{\M^+} \ftf(1-\blue{v_3})\, d\H^2(\omega)
        \, .
    \end{align*}

    \bigskip
    
    \noindent\underline{\textit{Energy in $\Om_2$}}: Next we consider the outer layer $\Om_2=\C_{2H\eta}(\M^+\setminus\P_{2\eta})\setminus\Om_1$, where we define $n$ by interpolating its angle with $e_3$ to 0 as $r$ goes from $H\eta$ to $2H\eta$. 
    Let $\Phi$ be the interpolation angle
    \begin{equation*}
        \Phi(r,\omega)=\Big(2-\frac{r}{H\eta}\Big)\Phi_0(\omega)
        \, 
    \end{equation*}
    for $\Phi_0(\omega)=\cos^{-1}(n_3^*(H,\varphi(\omega)))$.
    Then the director $n$ is defined by
    \begin{equation*}
        n(r,\omega)=\left(\frac{v_1(\omega)\sin\Phi(r,\omega)}{\sqrt{1-(v_3(\omega))^2}},\frac{v_2(\omega)\sin\Phi(r,\omega)}{\sqrt{1-(v_3(\omega))^2}},\cos\Phi(r,\omega)\right)
        \, .
    \end{equation*} 
     By choosing $\Phi$ and $\Phi_0$ in this way, $n(H\eta,\omega)$ agrees with the definition of $n$ on $\Om_1$ and $n(2H\eta,\omega)=e_3$, so it is a continuous extension from $\Om_1$ to the exterior of $\C_{2H\eta}(\M)$. The energy in this region is given by,
    \begin{equation*}
        \eta E_\xi(\Qc;\Om_2)=\int_{\M^+\setminus\P_{2\eta}}\int_{H\eta}^{2H\eta} \left(\eta|\grad n|^2+\frac{1}{\eta}g(n)\right)\prod_{i=1}^2(1+r\k_i)\, drd\H^2(\omega)
        \, .
    \end{equation*}
    We again obtain estimates on $|\grad n|^2$ by considering the radial and tangential derivatives separately.
    For the radial derivative we simply estimate
    \begin{align}\label{bound:om2:rad}
        \left|\frac{\partial n}{\partial r}\right|^2=\left|\frac{\partial \Phi}{\partial r}\right|^2=\frac{1}{H^2\eta^2}|\Phi_0|^2\leq\frac{C}{H^2\eta^2}
        \, .
    \end{align}
    Computing the tangential derivatives we see that
    \begin{equation*}
    \norm{\frac{\partial n}{\partial \omega_i}}^2=\frac{\sin^2\Phi}{1-(v_3)^2}\left(\norm{\frac{\partial v_1}{\partial \omega_i}}^2+\norm{\frac{\partial v_2}{\partial \omega_i}}^2\right)+\norm{\frac{\partial \Phi}{\partial \omega_i}}^2-\frac{(v_3)^2\sin^2\Phi}{(1-(v_3)^2)^2}\norm{\frac{\partial v_3}{\partial \omega_i}}^2
    \, .
    \end{equation*}
    Therefore
    \begin{align*}
        \norm{\frac{\partial n}{\partial \omega_i}}^2\leq \norm{\frac{\partial v}{\partial \omega_i}}^2+\norm{\frac{\partial \Phi_0}{\partial \omega_i}}^2&=\norm{\frac{\partial v}{\partial \omega_i}}^2+\frac{1}{1-(n_3(H\eta,\omega))^2}\norm{\frac{\partial}{\partial \omega_i}n_3(H\eta,\omega)}^2\\
        &\leq \norm{\frac{\partial v}{\partial \omega_i}}^2+\norm{\frac{\partial}{\partial \omega_i}n(H\eta,\omega)}^2
        \, ,
    \end{align*}
    where the last inequality comes from \eqref{bound:n3}.
    Now using the tangential derivative bound from $\Om_1$, we obtain the estimate on $\Om_2$
    \begin{equation*}
        |\grad_\omega n|^2\leq |\grad_\omega v|^2+|\grad_\omega n(H\eta,\omega)|^2\leq C|\grad_\omega v|^2
        \, .
    \end{equation*}
    Finally, using that $\Phi$ is monotone decreasing, $g(n(r,\omega))\leq g(n(H\eta,\omega))$ on $\Om_2$ and from the estimate on $|n_1^*|^2$ in Lemma \ref{lemma:optimal}, we can see that
    \begin{equation*}
        g(n(r,\omega))=\sqrt{\frac{3}{2}}|n_1(r,\omega)|^2\leq \sqrt{\frac{3}{2}}|n_1^*(H,\varphi(\omega))|^2\leq Ce^{-\ftf H}
        \, .
    \end{equation*}
    So the energy in this region can be estimated by
    \begin{align*}
        \eta \exe(\Qc;\Om_2)&\leq \int_{\M^+\setminus\P_{2\eta}}\int_{H\eta}^{2H\eta} C\left(\frac{1}{H^2\eta}+\eta|\grad_\omega v|^2+\frac{1}{\eta}e^{-\ftf H}\right)\, drd\H^2(\omega)\\
        &\leq\int_{\M^+\setminus\P_{2\eta}}C\left(\frac{1}{H}+H\eta^2|\grad_\omega v|^2+He^{-\ftf H}\right)\ d\H^2(\omega)
        \, .
    \end{align*}
    As in $\Om_1$, the tangential derivative term will disappear as $\eta\to0$, so 
    \begin{equation*}
        \limsup_{\xi,\eta\to0}\: \eta\exe(\Qc;\Om_2)\leq C|\M^+|\left(\frac{1}{H}+He^{-\ftf H}\right)=:\Psi(H)
        \, ,
    \end{equation*}
    and if we subsequently take $H\to\infty$, then $\Psi(H)\to 0$. 

    \bigskip
   
    \noindent\underline{\textit{Energy in $\Om_3$}}: Recall that $\Om_3$ is the region around point defects of \blue{$v$},
    \begin{equation*}
        \Om_3:=\C_\eta\left(\P_\eta\right)
        \, .
    \end{equation*}
    We estimate the energy around an arbitrary point defect, $p=p_j$ for some $j\in\{1,\dots, N\}$ and we drop the index $j$ on all related quantities to improve readability.
    We consider the region $\C_\eta(B(p,\eta))$.
    First define $u:B(p,\eta)\to\R^3$ by
    \begin{equation}\label{def:u}
        u(\omega)=\dist_\M(\omega,p)v(\omega)
        \, ,
    \end{equation}
    Then we define the vector field $n:(0,\eta)\times B(p,\eta)\to \S^2$ by
    \begin{equation*}
        n(r,\omega)=\frac{u(\omega)+r\nu(\omega)}{|u(\omega)+r\nu(\omega)|}
        \, ,
    \end{equation*}
    \answ{and we remark that $n$ is well-defined away from the point defect at $(0,p)$ since $u\perp \nu$.} \blue{The} energy \blue{is given by}
    \begin{equation*}
        \eta \exe(\Qc;\C_\eta(B(p,\eta)))=\int_{B(p,\eta)}\int_0^\eta\left(\eta|\grad n|^2+\frac{1}{\eta}g(n)\right)\prod_{i=1}^2(1+r\k_i)\, drd\H^2(\omega)
        \, .
    \end{equation*}
    We can compute the radial derivative of $n$ to be
    \begin{equation*}
        \frac{\partial n}{\partial r}=\frac{\nu|u+ r\nu|^2-(u+r\nu)[(u+ r\nu)\cdot \nu)]}{|u+ r\nu|^3}
        \, .
    \end{equation*}
    Taking the norm yields
    \begin{equation*}
        \left|\frac{\partial n}{\partial r}\right|\leq\frac{|u+ r\nu|^2+|u+ r\nu||(u+ r\nu)\cdot \nu|}{|u+ r\nu|^3}\leq \frac{2}{|u+r\nu|}
        \, .
    \end{equation*}
    Using that $u\perp \nu$, we write $|u+ r\nu|=\sqrt{|u|^2+|r\nu|^2}=\sqrt{|u|^2+r^2}$, so
    \begin{equation}\label{bound:om3:rad}
        \left|\frac{\partial n}{\partial r}\right|^2
        \leq
        \frac{C}{|u|^2+r^2}
        =
        \frac{1}{\eta^2}\left(\frac{C}{\blue{\eta^{-2}\dist^2_\M(\omega,p)}+h^2}\right)
        \, ,
    \end{equation}
    using $r=\eta h$. Next we consider the tangential derivatives of $n$. 
    For $i=1,2$, it holds
    \begin{equation*}
        \frac{\partial n}{\partial \omega_i}=\frac{(\frac{\partial u}{\partial \omega_i}+ r\frac{\partial \nu}{\partial \omega_i})|u+ r\nu|^2-(u+ r\nu)[(u+ r\nu)\cdot(\frac{\partial u}{\partial \omega_i}+ r\frac{\partial \nu}{\partial \omega_i})]}{|u+ r\nu|^3}
        \, 
    \end{equation*}
    so we can estimate the norm by
    \begin{equation}\label{om3:point:tangent}
        \left|\frac{\partial n}{\partial \omega_i}\right|
        \leq
        \frac{C|\frac{\partial u}{\partial \omega_i}+ r\frac{\partial \nu}{\partial \omega_i}|}{|u+ r\nu|}
        \leq
        \frac{C}{\eta}\left(\frac{|\frac{\partial u}{\partial \omega_i}|+|\frac{\partial \nu}{\partial \omega_i}|}{\sqrt{\blue{\eta^{-2}\dist_\M^2(\omega,p)}+h^2}}\right)
        \, .
    \end{equation}
    Using $|\grad_\omega \nu|\leq \sqrt{2}\k$ from \eqref{bound:curv}, it suffices to \answ{obtain a} bound \answ{on} $|\frac{\partial u}{\partial \omega_\blue{i}}|$.
    \begin{equation*}
        \frac{\partial u}{\partial \omega_\blue{i}}=\left(\frac{\partial}{\partial \omega_\blue{i}}\dist_\M(\omega,p)\right)\blue{v}(\omega)+\dist_\M(\omega,p)\left(\frac{\partial \blue{v}}{\partial \omega_\blue{i}}\right)
        \, .
    \end{equation*}
    Therefore, using \eqref{assumpt-bound}
    \begin{align*}
        \left|\frac{\partial u}{\partial \omega_i}\right|\leq \left|\frac{\partial}{\partial \omega_i}\dist_\M(\omega,p)\right|+\answ{\dist_\M(\omega,p)}\left|\frac{\partial \blue{v}}{\partial \omega_\blue{i}}\right|\leq C+\answ{\dist_\M(\omega,p)\left(\frac{C}{\dist_M(\omega,p)}\right)\leq C}
        \, .
    \end{align*}
    Putting this together with \eqref{om3:point:tangent}, we can see that
    \begin{equation}\label{bound:om3:tan}
        |\grad_\omega n|^2\leq \frac{1}{\eta^2}\left(\frac{C}{\blue{\eta^{-2}\dist^2_\M(\omega,p)}+h^2}\right)
        \, .
    \end{equation}
    Finally, we note that $g(n)\leq C$ for some $C>0$, thus the energy can be estimated by
    \begin{equation*}
        \eta \exe(\Qc; \C_\eta(B(p,\eta)))\leq\blue{\frac{C}{\eta}\int_{\answ{B(p,\eta)}}\int_0^1\left(\frac{C}{\blue{\eta^{-2}\dist^2_\M(\omega,p)}+h^2}+1\right)\ dh\, d\H^2(\omega)
        \, .}
    \end{equation*}
    \answ{Recall that the exponential map is bi-Lipschitz \cite{biLip}, so its Jacobian $|D\exp_p|$ is bounded on $B(0,\eta)\subset T_p\M$, hence $|D\exp_p(\eta\tau)|\leq C\eta^2$ for all $\tau\in B(0,1)\subset T_p\M$.} \blue{Therefore by the change of variables $\omega=\exp_p(\eta\tau)$, we have}
    \begin{equation*}
        \eta \exe(\Qc; \C_\eta(B(p,\eta)))\leq \blue{C\eta\int_{B(0,1)}\int_0^1\left(\frac{1}{|\tau|^2+h^2}+1\right)\ dh\, d\tau}
        \, .
    \end{equation*}
    Thus we integrate to obtain the estimate
    \begin{equation*}
        \eta \exe(\Qc; \C_\eta(B(p,\blue{\eta})))\leq \blue{C\eta}
        \, ,
    \end{equation*}
    and
    \begin{equation*}
        \limsup_{\xi,\eta\to0}\: \eta \exe(\Qc; \C_\eta(B(p,\eta)))=0
        \, .
    \end{equation*}
    This holds for each point $p=p_j$, therefore,
    \begin{equation*}
        \limsup_{\xi,\eta\to0}\eta \exe(\Qc;\Om_3)=0
        \, .
    \end{equation*}
    
    \bigskip
    
    \noindent\underline{\textit{Energy in $\Om_4$}}: 
    In this region, we do an interpolation of the angle between $n$ and $e_3$, similar to $\Om_2$, to continuously extend $n$ from $\Om_3$ to the exterior region. 
    First, choose any $p_j$ as we did in $\Om_3$, then define
\begin{equation*}
    \Om_4^{j}=\C_{2H\eta}(B(p_j,\eta))\setminus\C_\eta(B(p_j,\eta))
\end{equation*}
so that $\Om_4=\bigcup_{j=1}^{N}\Om_4^{j}$. Again we will let $p=p_j$ and drop the index $j$ wherever needed to improve readability. Now \answ{recalling $u$ from \eqref{def:u}, we }define $\vb:B(p,\eta)\to\S^2$ by
\begin{equation*}
    \vb(\omega)=\frac{u(\omega)+\eta\nu(\omega)}{|u(\omega)+\eta\nu(\omega)|}
    \, ,
\end{equation*}
and the interpolation $\Phi:(\eta,2H\eta)\times B(p,\eta)\to[0,\pi]$ by,
\begin{equation*}
    \Phi(r,\omega)=\left(1-\frac{r-\eta}{2H\eta-\eta}\right)\Phi_0(\omega)
    \, ,
\end{equation*}
where $\Phi_0(\omega)=\cos^{-1}(\vb_3(\omega))$. Finally we let
\begin{equation*}
    n(r,\omega)=\left(\frac{\vb_1(\omega)\sin\Phi(r,\omega)}{\sqrt{1-(\vb_3(\omega))^2}},\frac{\vb_2(\omega)\sin\Phi(r,\omega)}{\sqrt{1-(\vb_3(\omega))^2}},\cos\Phi(r,\omega)\right)
    \, ,
\end{equation*}
so that $n(\eta,\omega)=\vb(\omega)$ and $n(2H\eta,\omega)=e_3$. The energy in $\Om_4^{j}$ can be estimated by
\begin{equation*}
    \eta \exe(\Qc;\Om_4^{j})\leq\int_{B(p,\eta)}\int_\eta^{2H\eta}\left(\eta\norm{\frac{\partial n}{\partial r}}^2+\eta|\grad_\omega n|^2+\frac{1}{\eta}g(n)\right)\prod_{\blue{i}=1}^2(1+r\k_\blue{i})\, dr d\H^2(\omega)
    \, .
\end{equation*}
It is easy to see that,
\begin{equation}\label{bound:om4:rad}
    \norm{\frac{\partial n}{\partial r}}^2=\norm{\frac{\partial \Phi}{\partial r}}^2=\frac{1}{(2H\eta-\eta)^2}|\Phi_0|^2\leq\frac{C}{(2H-1)^2\eta^2}\leq\frac{C}{\eta^2}
    \, ,
\end{equation}
since $\Phi_0$ is bounded and $H>0$ is large. By a straightforward calculation, we also see that for $i=1,2$,
\begin{equation*}
    \norm{\frac{\partial n}{\partial \omega_i}}^2
    =
    \frac{\sin^2\Phi}{1-\vb_3^2}\left(
    \norm{\frac{\partial \vb_1}{\partial \omega_i}}^2+\norm{\frac{\partial \vb_\answ{2}}{\partial \omega_i}}^2
    \right)
    +
    \norm{\frac{\partial \Phi}{\partial \omega_i}}^2
    -
    \frac{\vb_3^2\sin^2\Phi}{(1-\vb_3^2)^2}\norm{\frac{\partial \vb_3}{\partial \omega_i}}^2
    \, ,
\end{equation*}
but using that $\sin\Phi\leq\sqrt{1-\vb_3^2}$, we can further estimate the gradient by
\begin{equation}\label{om4:tangent}
    |\grad_\omega n|^2\leq |\grad_\omega \vb|^2+|\grad_\omega \Phi|^2
    \, .
\end{equation}
Next we consider the tangential derivatives of $\Phi$ and we see that they are bounded by the derivatives of $\Phi_0$ as follows,
\begin{equation*}
    \norm{\frac{\partial \Phi}{\partial \omega_i}}^2=\left(1-\frac{r-\eta}{2H\eta-\eta}\right)^2\norm{\frac{\partial \Phi_0}{\partial \omega_i}}^2\leq\norm{\frac{\partial \Phi_0}{\partial \omega_i}}^2
    \, ,
\end{equation*}
for $i=1,2$. Using the definition of $\Phi_0$,
\begin{equation*}
    \norm{\frac{\partial \Phi_0}{\partial \omega_i}}^2=\frac{1}{1-\vb_3^2}\norm{\frac{\partial \vb_3}{\partial \omega_i}}^2\leq\norm{\frac{\partial \vb}{\partial \omega_i}}^2
    \, ,
\end{equation*}
where the last inequality comes from \eqref{bound:n3}.
It then follows that $|\grad_\omega\Phi|^2\leq|\grad_\omega \vb|^2$, so that combining with \eqref{om4:tangent} and \eqref{bound:om3:tan} we obtain the estimate
\begin{equation}\label{bound:om4:tan}
    |\grad_\omega n|^2\leq 2|\grad_\omega \vb|^2\leq\frac{1}{\eta^2}\left(\frac{C}{\blue{\eta^{-2}\dist^2_\M(\omega,p)}+1}\right)\leq\frac{C}{\eta^2}
    \, .
\end{equation}
With these estimates in hand, we bound the energy in this region by
\begin{equation*}
    \eta \exe(\Qc;\Om_4^{j})\leq\int_{B(p,\eta)}\int_\eta^{2H\eta}\frac{C}{\eta}\ drd\H^2(\omega) =C(2H-1)\blue{\H^2(}B(p,\eta)\blue{)}
    \, .
\end{equation*}
Therefore, taking $\xi,\eta\to0$ the energy in $\Om_4^{j}$ vanishes, so if we consider the energy in all of $\Om_4$ we see that
\begin{equation*}
    \limsup_{\xi,\eta\to0}\: \eta \exe(\Qc;\Om_4)=0
    \, .
\end{equation*}

\bigskip

\noindent\underline{\textit{Energy in $\Om_5$}}: The last region we need to consider is
\begin{equation*}
    \Om_5=\C_{2H\eta}(\P_{2\eta})\setminus\left(\Om_3\cup\Om_4\right)
    \, .
\end{equation*}
As before, it is enough to carry out the construction separately for each of the components of this region. Hence it suffices to consider
\begin{equation*}
    \Om_5^{j}=\C_{2H\eta}(A^{j}_\eta)
    \, ,
\end{equation*}
where $A^{j}_\eta$ is the annular region given by $A^{j}_\eta=B(p_j,2\eta)\setminus B(p_j,\eta)$ for some $j\in\{1,\dots,N\}$.
On $\Om_5^{j}$, we want to do a Lipschitz extension of $\Qc$ based on a Lipschitz extension of the angle $\ph$ between $n$ and $e_3$. To define $\ph$ on $\partial \Om_5^{j}$, we note that there exists a continuous extension of $n$ from each of the neighbouring regions $\Om_j$ for $j=1,\dots,4$ on $\partial \Om_5^{j}$. We also take $n=v$ on the surface of the manifold $\answ{\M^+}$. Then, define $\ph:\partial \Om_5^j\to[0,\pi]$ by
\begin{equation*}
    \ph(r,\omega)=\cos^{-1}(n_3(r,\omega))
    \, .
\end{equation*}
Note that by taking the gradient of $\ph$ and then applying the inequality from \eqref{bound:n3},
\begin{equation*}
    |\grad \ph|^2=\frac{1}{1-n_3^2}|\grad n_3|^2\leq C|\grad n|^2\leq C\norm{\frac{\partial n}{\partial r}}^2+C|\grad_\omega n|^2
    \, ,
\end{equation*}
for some $C>0$, so the Lipschitz constant of $\ph$ will be proportional to the maximum of the sum of radial and tangential derivatives of $n$ on $\partial \Om_5^j$. On $\partial \Om_1$, \blue{recalling that $\P=\bigcup_{j=1}^N\{p_j\}$} we can see that
\begin{equation*}
    \norm{\frac{\partial n}{\partial r}}^2\leq\frac{1}{\eta^2}\norm{\frac{\partial n^*}{\partial \rt}}^2\leq \frac{Ce^{-\ftf \rt}}{\eta^2}\leq \frac{C}{\eta^2}\quad\text{and}\quad |\grad_\omega n|^2\leq \frac{C}{\dist_\M^2(\omega,\blue{\P})}\leq\frac{C}{\eta^2}
    \, ,
\end{equation*}
using \eqref{bound:om1:rad} and \eqref{bound:om1:tan}. Then on $\partial \Om_2$, we were able to estimate the derivatives by
\begin{equation*}
    \norm{\frac{\partial n}{\partial r}}^2\leq \frac{C}{H^2\eta^2}\leq\frac{C}{\eta^2}\quad\text{and}\quad |\grad_\omega n|^2\leq \frac{C}{\eta^2}
    \, ,
\end{equation*}
since the tangential derivative estimate comes from the estimate in $\Om_1$.
Considering $\partial \Om_3$ and using both \eqref{bound:om3:rad} and \eqref{bound:om3:tan},
\begin{equation*}
    \norm{\frac{\partial n}{\partial r}}^2\leq\frac{C}{\eta^2}\left(\frac{1}{1+h^2}\right)\leq \frac{C}{\eta^2}\quad\text{and}\quad |\grad_\omega n|^2\leq |\grad_\omega n|^2\leq \blue{\frac{1}{\eta^2}\left(\frac{C}{1+h^2}\right)}\leq\frac{C}{\eta^2}
    \, ,
\end{equation*}
since \blue{$\dist_\M(\omega,p_j)=\eta$} on the boundary between $\Om_3$ and $\Om_5$. On $\partial \Om_4$, we also saw from \eqref{bound:om4:rad} and \eqref{bound:om4:tan} that
\begin{equation*}
    \norm{\frac{\partial n}{\partial r}}^2\leq\frac{C}{\eta^2}\quad\text{and}\quad |\grad_\omega n|^2\leq\blue{\frac{1}{\eta^2}\left(\frac{C}{\blue{\eta^{-2}\dist^2_\M(\omega,p_j)}+1}\right)}\leq\frac{C}{\eta^2}
    \, ,
\end{equation*}
so it remains to bound the gradient on the intersection of the boundary of $\Om_5^{j}$ with \answ{$\M^+$} and the boundary with the exterior region where $\Qc=Q_\infty$.
On this region of \answ{$\M^+$} it holds that $n=v$, so $|\grad_\omega n|=|\grad_\omega v|\leq C\eta^{-1}$ here. 
Finally in the exterior region, $\Qc=Q_\infty$, so along this boundary, $\ph$ is constant, thus the Lipschitz constant here is zero.
We can see that this means $\ph$ is Lipschitz on $\partial \Om_5^{j}$, with a Lipschitz constant 
\begin{equation*}
    \L_{\eta}\leq \frac{C}{\eta}
\end{equation*}
for some constant $C>0$. The same strategy is used in this region as in the spherical case described in \cite{bls}. Following a similar argument, we apply Theorem 2.10.43 from \cite{FedererGMT}, in order to extend $\ph$ to all of $\Om_5^{j}$ and in doing so, $\ph$ remains Lipschitz on this region with the same Lipschitz constant $\L_{\eta}$. Then one can define $n$ on $\Om_5^{j}$ by
\begin{equation*}
    n(r,\omega)=\left(\frac{v_1(\omega)\sin\ph(r,\omega)}{\sqrt{1-(v_3(\omega))^2}},\frac{v_2(\omega)\sin\ph(r,\omega)}{\sqrt{1-(v_3(\omega))^2}},\cos\ph(r,\omega)\right)
    \, .
\end{equation*}
We note that $n$ in this region is of a very similar form to $n$ on $\Om_2$ and $\Om_4$, so we can use the same computations to see that
\begin{equation*}
    \norm{\frac{\partial n}{\partial r}}^2=\norm{\frac{\partial \ph}{\partial r}}^2\leq|\grad \ph|^2\leq\L_{\eta}^2\leq\frac{C}{\eta^2}
    \, .
\end{equation*}
For the tangential derivatives, the computations from $\Om_2$ and $\Om_4$ allow us to deduce the estimate
\begin{equation*}
    |\grad_\omega n|^2\leq C|\grad_\omega v|^2+|\grad_\omega \ph|^2\leq  \frac{C}{\eta^2}
    \, .
\end{equation*}
Since $g$ is bounded, we can estimate the total energy in this region by
\begin{equation*}
    \eta \exe(\Qc;\Om_5^{j})\leq\int_{\Om_5^{j}}\left(\eta|\grad n|^2+\frac{1}{\eta}g(n)\right)\, dx\leq\int_{\Om_5^{j}}\frac{C}{\eta}\, dx\leq CH\blue{\H^2(}A^{j}_\eta\blue{)}
    \, .
\end{equation*}
Therefore,
\begin{equation*}
    \limsup_{\xi,\eta\to0}\: \eta \exe(\Qc;\Om_5^{j})=0
    \, ,
\end{equation*}
so the energy in all of $\Om_5$ is negligible, i.e.
\begin{equation*}
    \limsup_{\xi,\eta\to0}\: \eta \exe(\Qc;\Om_5)=0
    \, .
\end{equation*}

\bigskip

\noindent\underline{\textit{Conclusion}}: Now we are able to put all of the regions together. \blue{Recalling that the energy on $\C_{2H\eta}(\M^0)$ is zero, we write}
\begin{equation*}
    \eta \exe(\Qc;\C_{2H\eta}(\M^+\cup\M^0))=\sum_{k=1}^5\eta \exe(\Qc;\Om_k)
    \, ,
\end{equation*}
and it holds that
\begin{align*}
    \limsup_{\xi,\eta\to0}\: \eta \exe(\Qc;\blue{\C_{2H\eta}(\M^+\cup\M^0)})
    \leq\int_{\M^+}\ftf(1-\blue{v_3})\, d\H^2(\omega)+\Psi(H)
    \, .
\end{align*}
\end{proof}

\answ{Using both Lemma \ref{bdry-data} and Proposition \ref{recov} we are now able to prove  Theorem \ref{thm:upper}.}
\begin{proof}[Proof of Theorem \ref{thm:upper}]
    \answ{Let $v^\d_+$ be the family of vector fields described in Lemma \ref{bdry-data}, then from Proposition \ref{recov} we obtain a family of maps $\Qcd$ which satisfy
    \begin{equation*}
        \limsup_{\xi,\eta\to0}\: \eta \exe(\Qcd;\C_{2H\eta}(\M^+\cup\M^0))
    \leq\int_{\M^+}\ftf(1-v^\d_{+,3})\, d\H^2(\omega)+\Psi(H)
    \, .
    \end{equation*}
    In order to do the equivalent construction on $\M^-$, we define the reflection $\sigma:\R^3\to\R^3$ by
    \begin{equation*}
        \sigma(x_1,x_2,x_3)=(x_1,x_2,-x_3)
        \, .
    \end{equation*}
    Now using Lemma~\ref{bdry-data} we can construct a family of vector fields $\tilde{v}^\d_-$ on $\sigma(\M^-\cup\M^0)$ and we remark that choosing $v^\d_-=\sigma\circ\tilde{v}^\d_-\circ\sigma$, we have that $v^\d_-$ is a family of unit vector fields on $\M^-\cup\M^0$. Next we apply Proposition \ref{recov} (up to changing the manifold $\M$ to $\sigma(\M)$) to obtain a family of maps $\tilde{Q}^{H,\d}_{\xi,\eta}$ on $\C_{2H\eta}(\sigma(\M^-\cup\M^0))$ such that
    \begin{equation*}
        \limsup_{\xi,\eta\to0}\: \eta \exe(\tilde{Q}_{\xi,\eta}^{H,\d};\C_{2H\eta}(\sigma(\M^-\cup\M^0)))
    \leq\int_{\sigma(\M^-)}\ftf(1-\tilde{v}^\d_{-,3})\, d\H^2(\omega)+\Psi(H)
    \, .
    \end{equation*}
    As seen in the proof of Proposition \ref{recov}, $\tilde{Q}_{\xi,\eta}^{H,\d}=n\otimes n-\frac{1}{3}I$ is defined on $\C_{2H\eta}(\sigma(\M^-\cup\M^0)$, so to reflect this onto $\C_{2H\eta}(\M^-\cup\M^0)$, we define $Q_{\xi,\eta}^{H,\d}$ by
    \begin{equation*}
        Q_{\xi,\eta}^{H,\d}=\big(\sigma\circ n\circ\sigma\big)\otimes \big(\sigma\circ n\circ\sigma\big)-\frac{1}3I
        \, .
    \end{equation*}
    Note that this new definition of $Q_{\xi,\eta}^{H,\d}$ agrees with the previous one on $\C_{2H\eta}(\M^0)$ since both are identically $Q_\infty$ here. By the change of variable $\omega=\sigma(\tilde{\omega})$, we have
    \begin{equation*}
        \limsup_{\xi,\eta\to0}\: \eta \exe(Q_{\xi,\eta}^{H,\d};\C_{2H\eta}(\M^-\cup\M^0))
    \leq\int_{\M^-}\ftf(1+v^\d_{-,3})\, d\H^2(\tilde\omega)+\Psi(H)
    \, .
    \end{equation*}
    Putting together both halves of the manifold, we obtain the upper bound
    \begin{equation*}
        \limsup_{\xi,\eta\to0}\: \eta \exe(\Qcd)
    \leq\int_{\M^+}\ftf(1-v^\d_{+,3})\, d\H^2(\omega)+\int_{\M^-}\ftf(1+v^\d_{-,3})\, d\H^2(\omega)+2\Psi(H)
    \, .
    \end{equation*}
    Now let $\d_k\to 0$ as $k\to\infty$ be any sequence, then by Lemma \ref{bdry-data}, $v_{+,3}^{\d_k}\to\sqrt{1-\nu_3^2}$ and $v_{-,3}^{\d_k}\to-\sqrt{1-\nu_3^2}$, so
    \begin{equation*}
       \limsup_{k\to\infty}\left( \limsup_{\xi,\eta\to0}\: \eta \exe(Q_{\xi,\eta}^{H,\d_k})\right)
    \leq\int_{\M}\ftf\Big(1-\sqrt{1-\nu_3^2}\Big)\, d\H^2(\omega)+2\Psi(H)
    \, .
    \end{equation*}
    Note that the integral is over $\M$ as opposed to being over $\M^+\cup\M^-$, but this change has no impact since the integrand is zero on $\M^0$. Finally, taking $H\to\infty$ and recalling that $\Qm$ is a minimizer of $\exe$, we can see that}
    \begin{equation}\label{upper-full}
        \limsup_{\xi,\eta\to0}\: \eta\exe(\Qm)\leq\int_\M\ftf\Big(1-\sqrt{1-\nu_3^2}\Big)\, d\H^2(\omega)
    \end{equation}
    \answ{since $\eta\exe(\Qm)\leq\eta\exe(\Qcdk)$}.
    It is clear that
    \begin{equation*}
        0\leq\liminf_{\xi,\eta\to0}\:\eta\exe(\Qm;\Omega\setminus \C_{r_0}(\M))\leq\limsup_{\xi,\eta\to0}\:\eta\exe(\Qm;\Omega\setminus \C_{r_0}(\M))
        \, ,
    \end{equation*}
    \blue{where $r_0$ was defined in \eqref{def:r0}}, but applying Theorem \ref{thm:lower} and \eqref{upper-full} we get
    \begin{align*}
        \limsup_{\xi,\eta\to0}\:\eta\exe(\Qm;\Om\setminus\C_{r_0}(\M))=\limsup_{\xi,\eta\to0}\eta\exe(\Qm)-\liminf_{\xi,\eta\to0}\:\eta\exe(\Qm;\C_{r_0}(\M))\leq0
    \end{align*}
    therefore
    \begin{equation}\label{ext:energy}
        \lim_{\xi,\eta\to0}\:\eta\exe(\Qm;\Om\setminus\C_{r_0}(\M))=0
        \, .
    \end{equation}
    This in essence says that all of the energy is concentrated in a small layer around $\M$. 
    Now let $U\subset\M$, then
    \begin{equation*}
        \eta\exe(\Qm;\C_{r_0}(U))
        =
        \eta\exe(\Qm)
        -\eta\exe(\Qm;\C_{r_0}(U^c))
        -\eta\exe(\Qm;\Om\setminus\C_{r_0}(\M))
        \, .
    \end{equation*}
    So neglecting the energy on $\C_{r_0}(U^c)$ (which is negative) and using \eqref{upper-full} and \eqref{ext:energy}, we get that
    \begin{equation*}
        \limsup_{\xi,\eta\to0}\:\eta\exe(\Qm;\C_{r_0}(U))\leq\int_U\ftf\Big(1-\sqrt{1-\nu_3^2}\Big)\, d\H^2(\omega)
        \, .
    \end{equation*}
    Note that invoking Theorem~\ref{thm:lower} for $U^c$ we can even show that the limsup in the above is actually a limit.oll
\end{proof}

\section{Analysis of Possible Defects}\label{sec:lines}
The energy estimates from the upper and lower bound suggest that the energy concentrates on the surface of $\M$ as $\xi,\eta\to0$ and we do not see any additional terms which would result from line defects. 
In the following we further analyze the possibility of line singularities and how they impact the energy. 
For the discussion of line defects we distinguish two asymptotic regimes as $\xi,\eta\to 0$:
\begin{equation*}
    \eta|\ln\xi|\to\beta\in(0,\infty]\quad\text{and}\quad \eta|\ln\xi|\to0
    \, .
\end{equation*}
In the case of strong homeotropic anchoring, the former regime is studied in \cite{ACS2021, ACS2024} and the existence (for small $\beta$) / non-existence (for large $\beta$) of line defects has been investigated. 
The latter regime is considered in \cite{abl}, hinting at the existence of line defects. 

For our planar degenerate boundary conditions, we believe that in the first case, when $\eta|\ln\xi|\to\beta\in(0,\infty]$, minimizers cannot exhibit line defects. 
However, this does not follow from our analysis and proving this statement is beyond the scope of this investigation. 
In lieu, we show that if minimizers have a line defect, then as $\xi,\eta\to 0$, the defect will shrink to a point. 
This result is seen by Theorem \ref{thm:line_defects_shrink_0}.

For the second case when $\eta|\ln\xi|\to 0$, the energetic cost of a line defect is too small so that our estimates cannot detect them. 
A finer analysis would be required to study line defects in this regime. 

To make precise what we mean by ``line defect'', we use the notion of $d-$dimensional flat chains with values in a finite coefficient group, introduced by W.\ Fleming \cite{Fleming} in 1966.
In our specific application, the coefficient group  \answ{is given by the fundamental group of $\N$ which is isomorphic to} $\mathbb{Z}_2$, thus allowing us to see one-dimensional flat chains ($d=1$) as a measure theoretic generalization of classical (smooth) lines in $\mathbb{R}^3$ and avoid problems involving multiplicities. 
The appropriate generalization of ``size'' (in the case $d=1$ the length of the line, the area of a surface for $d=2$), is given by the mass functional $\MM$.
Another natural norm for flat chains is the so called flat norm: 
For a $d-$dimensional flat chain $S$ we define
\begin{align*}
    \FF(S) 
    \ := \
    \inf\{\MM(T) + \MM(R) \: : \: S=\partial T + R \}
    \, ,
\end{align*}
where the infimum is taken over all $(d+1)-$dimensional flat chains $T$ and $d-$dimensional flat chains $R$.

Roughly speaking, one can consider the set $\Qm^{-1}(\B)$, where $\B$ (as in \eqref{set:B}) is the set of negative uniaxial $Q-$tensors for which the projection $P$ onto $\N$ is not well-defined. 
This set $\Qm^{-1}(\B)$ is in general a very irregular set. 
Evoking the Thom transversality theorem, one can show that for almost every $Q\in\sym$, the set $(\Qm-Q)^{-1}(\B)$ is much more regular \answ{and can be given the structure of a smooth, finite complex.
Arguing locally, one can consider two-dimensional disks $B^2_r(x)$ (for some small enough $r>0$) that are intersecting the complex transversely at a single point, the center $x$ of the disk.
The set $(\Qm-Q)(\partial B^2_r(x))\subset\sym\setminus\B$ forms a loop and since it doesn't intersect $\B$, we can apply the projection $P$ onto $\N$.
Therefore, $(P\circ(\Qm-Q))(\partial B^2_r(x))$ defines a closed loop in $\N$ which allows us to associate to it an element in the fundamental group $\pi_1(\N)$, which we take as coefficient for our flat chain.} 
This construction is made precise in \cite{CO1} where an operator $\mathbf{S}$ is introduced, to pass from $\Qm$ to a flat chain $\mathbf{S}_Q(\Qm)$.
This operator is used to study various problems in \cite{CO2}.
The most important property for us is that $\mathbf{S}$ preserves the topological information of $\Qm$, allowing us to identify $\mathbf{S}_Q(\Qm)$ with the defect set of $\Qm$.


\begin{theorem}\label{thm:line_defects_shrink_0}
    Let $\Qm$ minimize $\exe$ with $\Qm\in\A$. If
    \begin{equation*}
        \frac{\eta}{\xi}\to\infty\quad\text{and in addition}\quad\eta|\ln\xi|\to\beta\in(0,\infty]\quad\text{as}\quad\xi,\eta\to0
        \, ,
    \end{equation*}
    then, for any $0<\delta^*<\dist(\N,\B)$, it holds
    \begin{equation*}
        \int_{B_{\delta^*}}\FF(\mathbf{S}_Q(\Qm)) \dx Q
        \ \to \ 
        0
        \quad\text{as}\quad
        \xi,\eta\to0
        \, ,
    \end{equation*}
    where $B_{\delta^*}=B_{\delta^*}(0)\subset\sym$.
\end{theorem}

\begin{proof}
    We begin by defining the energy functional
    \begin{equation*}
        \ex(Q):=\int_\Om\frac{1}{2}|\grad Q|^2+\frac{1}{\xi^2}f(Q)\, dx
        \, ,
    \end{equation*}
    so $\exe(Q)\geq\ex(Q)$ for all $Q\in\A$. Now,
    \begin{equation*}
        \limsup_{\xi,\eta\to0}\frac{\ex(\Qm)}{|\ln\xi|}\leq\limsup_{\xi,\eta\to0}\frac{\eta\exe(\Qm)}{\eta|\ln\xi|}\leq\frac{1}{\beta}\int_\M\ftf\Big(1-\sqrt{1-\nu_3^2}\Big)\, d\H^2<+\infty
        \, .
    \end{equation*}
    Applying Theorem C from \cite{CO2}, there exists a finite mass flat chain $S$ with values in $\mathbb{Z}_2$ such that $\int_{B_{\delta^*}}\FF(\mathbf{S}_Q(\Qm) - S) \dx Q \to 0$ as $\eta,\xi\to 0$ and
    \begin{equation}\label{thmc}
    \liminf_{\xi,\eta\to0}\frac{\eta\ex(\Qm;U_1^\e\cap\Om)}{\eta|\ln\xi|}\geq \mathbb{M}(S\restr U_1^\e)
    \, ,
    \end{equation}
    where $S\restr U_1^\e$ denotes the restriction of $S$ to $U_1^\e$, an open set which we have yet to define.
    It is therefore enough to show that $S=0$.
    We note that since $\exe\geq\ex$, it holds
    \begin{equation*}
        \liminf_{\xi,\eta\to0}\eta\exe(\Qm;U_1^\e\cap\Om)\geq\beta\mathbb{M}(S\restr U_1^\e)
        \, .
    \end{equation*}
    We remark it follows from \cite{White} that $S$ is rectifiable as a finite mass flat chain with values in the discrete group $\mathbb{Z}_2$. 
    By inner regularity of the measure $\MM(S)$ we can choose a subset $S^\e$ of the rectifiability set of $S$ which is compact and $\MM(S)-\MM(S^\e)<\e$. 
    We then choose an open set $U_1^\e\subset\R^3$ such that $S^\e\subset U_1^\e$ and $\H^2(U_1^\e\cap\M)<\e$. 
    It follows from this choice of $U_1^\e$ that
    \begin{equation*}
        \MM(S)-\MM(S\restr U_1^\e)\leq\MM(S)-\MM(S^\e)<\e
        \, .
    \end{equation*}
    Furthermore, choosing $U_2^\e=(\overline{U_1^\e})^c$, we can see that $\H^2(\M\setminus U_2^\e)<\e$ and
    \begin{equation*}
        \eta\exe(\Qm)= \eta\exe(\Qm;U_1^\e\cap\Om)+\eta\exe(\Qm;U_2^\e\cap\Om)
        \, .
    \end{equation*}
    Using Corollary \ref{cor:open} and \eqref{thmc} we can see that,
    \begin{equation*}
        \liminf_{\xi,\eta\to0}\: \eta\exe(\Qm)\geq \beta\mathbb{M}(S\restr U_1^\e)+\int_{\M\cap U_2^\e}\ftf\Big(1-\sqrt{1-\nu_3^2}\Big)\, d\H^2
        \, .
    \end{equation*}
    Now since $\mathbb{M}(S)-\mathbb{M}(S\restr U_1^\e)<\e$ and $\H^2(\M\setminus U_2^\e)<\e$, we take $\e\to 0$ to see that
    \begin{equation*}
        \liminf_{\xi,\eta\to0}\eta\exe(\Qm)\geq\int_{\M}\ftf\Big(1-\sqrt{1-\nu_3^2}\Big)\, d\H^2+\beta\mathbb{M}(S)\, .
    \end{equation*}
    Recall that from Theorem \ref{thm:upper} we have
    \begin{equation*}
        \limsup_{\xi,\eta\to0}\:\eta\exe(\Qm)\leq \int_{\M}\ftf\Big(1-\sqrt{1-\nu_3^2}\Big)\, d\H^2
        \, ,
    \end{equation*}
    so this implies that $\mathbb{M}(S)=0$ and hence $S=0$.
\end{proof}

\section{Rotations of $\M$}\label{sec:Rotations}
In the previous sections, we were concerned with a given manifold $\M$ and a given magnetic field direction {$e_3$}.
From a physical point of view, colloidal particles are free to rotate inside the liquid crystal and experimentally observed configurations are hence found in an equilibrium state with respect to the magnetic field induced alignment at infinity.
In this section we are thus exploring the energy minimizing orientations of $\M$ with respect to the magnetic field. 

From the energy bounds obtained in Theorems \ref{thm:lower} and \ref{thm:upper} we see that the limiting energy of minimizers is
\begin{equation*}
    \lim_{\xi,\eta\to0}\: \eta\exe(\Qm;\C_{r_0}(U))=\int_U\ftf\Big(1-\sqrt{1-\nu_3^2}\Big)\, d\H^2(\omega)=:E_0(U)
    \, .
\end{equation*}
Rather than rotate the manifold itself, we instead rotate the magnetic field, to change the preferred alignment at infinity. Minimizing amongst rotations of $\M$ is equivalent to minimizing over rotations of $n_\infty$, which we can see by a change of variables. Let $R\in SO(3)$ be a rotation matrix, then define the rotated manifold $R(\M)$ by
\begin{equation*}
    R(\M)=\{R\omega:\omega\in\M\}\ .
\end{equation*}
Now letting $\omegah=R\omega$, we have that $\nuh(\omegah)=R\nu(\omega)$, where $\nuh$ denotes the normal on $R(\M)$. Applying this change of variables to $E_0(R(\M))$, we obtain
\begin{align*}
    E_0(R(\M))&=\int_{R(\M)}\ftf\Big(1-\sqrt{1-(\nuh(\omegah)\cdot e_3)^2}\Big)\, d\H^2(\omegah)\\
    &=\int_\M\ftf\Big(1-\sqrt{1-(R\nu(\omega)\cdot e_3)^2}\Big)\, d\H^2(\omega)\\
    &=\int_\M\ftf\Big(1-\sqrt{1-(\nu(\omega)\cdot R^{-1}e_3)^2}\Big)\, d\H^2(\omega)
    \, .
\end{align*}
Therefore taking $n_\infty=R^{-1}e_3$, the new limiting energy is characterized by:
\answ{\begin{equation}\label{def:limit}
    E_0(\M;n_\infty):=\int_\M\ftf\Big(1-\sqrt{1-(\nu\cdot n_\infty)^2}\Big)\, d\H^2
\end{equation}}
We remark that we are abusing notation in the sense that the limiting energy $E_0(\M;\blue{n_\infty})$ was found for a specific class of manifolds $\M$, but we will compute $E_0(\M;\blue{n_\infty})$ for manifolds $\M$ which are not necessarily closed or $C^{\answ{2},1}$. \answ{However, it will be shown in Proposition \ref{prop:approx_C11} that these computations are still physically meaningful despite the lower regularity.}

We want to study the limit energy $E_0$ as a function of $n_\infty$, however exact minimizers are difficult to obtain in general. 
Rather we present a method to determine critical points of the direction for the magnetic field.
\answ{\begin{proposition}\label{prop:rotations}
    The direction $n_\infty\in \S^2$ is a critical point of the energy function $E_0(\M;\cdot)$ if and only if
    \begin{equation}\label{prop:rotations:eq}
        \int_\M \frac{(\nu\cdot n_\infty)\nu-(\nu\cdot n_\infty)^2n_\infty}{\sqrt{1-(\nu\cdot n_\infty)^2}}\, d\H^2(\omega)=0
        \, .
    \end{equation}
\end{proposition}}
\begin{proof}
    To simplify notation, we let $n=n_\infty$, then we begin with a standard approach from the calculus of variations, we let $n$ vary by setting
    \begin{equation*}
        n_t=\frac{n+tm}{|n+tm|},\quad m\in\S^2
        \, .
    \end{equation*}
    We then evaluate
    \begin{equation*}
        \Big(\frac{d}{dt}E_0(\M;n_t)\Big)\Big|_{t=0}=\int_\M \frac{(\nu\cdot n)\nu-(\nu\cdot n)^2n}{\sqrt{1-(\nu\cdot n)^2}}\cdot m\, d\H^2
        \, .
    \end{equation*}
    If $n$ is a critical point, then this quantity is zero, so we obtain the condition that
    \begin{equation*}
        \int_\M \frac{(\nu\cdot n)\nu-(\nu\cdot n)^2n}{\sqrt{1-(\nu\cdot n)^2}}\cdot m\, d\H^2(\omega)=0
        \, .
    \end{equation*}
    Since $m\in\S^2$ was arbitrary, this is equivalent to \answ{\eqref{prop:rotations:eq}.}
    \answ{For the reverse implication, assume \eqref{prop:rotations:eq} holds.
    Then for any $m\in\S^2$
    \begin{equation*}
        \int_\M \frac{(\nu\cdot n)\nu-(\nu\cdot n)^2n}{\sqrt{1-(\nu\cdot n)^2}}\cdot m\, d\H^2(\omega)=\sum_{i=1}^3 m_i\int_\M \frac{(\nu\cdot n)\nu_i-(\nu\cdot n)^2n_i}{\sqrt{1-(\nu\cdot n)^2}}\, d\H^2(\omega)=0
        \, .
    \end{equation*}
    Therefore
    \begin{equation*}
        \Big(\frac{d}{dt}E_0(\M;n_t)\Big)\Big|_{t=0}=0
    \end{equation*}
    so by definition, $n$ is a critical point of $E_0(\M;\cdot)$.}
\end{proof}

We are able to get more information when the manifold $\M$ has an axis of rotational symmetry. In this case, we reduce the degree of freedom in $n_\infty$ that we are minimizing over. If $\M$ has an axis of rotational symmetry, then $\M$ can be translated and rotated such that the axis of symmetry is now the $e_3-$axis. This leads us to the next proposition.
\begin{proposition}\label{prop:sym}
    Let $\M$ be a \answ{$C^{2,1}$}, connected, compact 2-manifold \blue{without boundary and} with rotational symmetry around the $e_3-$axis, then we can describe the limiting energy by:
    \begin{equation}\label{eqn:rot}
        E_0(\M;n_\infty)=\int_0^1\int_0^{2\pi}\ftf\Big(1-\sqrt{1-(\nu(s,\theta)\cdot n_\infty)^2}\Big)\gamma_1(s)\sqrt{\gamma_1'(s)^2+\gamma_3'(s)^2}\ d\theta ds
        \, ,
    \end{equation}
    where $\gamma:[0,1]\to\{(x_1,0,x_3):x_1\geq0\}$ parameterizes the curve generated by the intersection $\M\cap\{(x_1,0,x_3):x_1\geq0\}$ and 
    \begin{equation*}
        \nu(s,\theta)=\frac{1}{|\gamma'(s)|}\begin{pmatrix}
            \cos\theta&\sin\theta\\-\sin\theta&\cos\theta
        \end{pmatrix}\begin{pmatrix}
            \gamma_3'(s)\\-\gamma_1'(s)
        \end{pmatrix}
        \, .
    \end{equation*}
    Moreover, for any rotation $R_\phi$ about the $e_3-$axis,
    \begin{equation*}
    E_0(\M;n_\infty)=E_0(\M;R_\phi n_\infty)
    \, .
    \end{equation*}
\end{proposition}
\begin{proof}
    Since $\M$ is connected and closed, the intersection $\M\cap\{(x_1,0,x_3):x_1\geq0\}$ can be parameterized by a curve $\gamma$ and this curve has the property that either $\gamma$ is closed or its endpoints lie on the $e_3-$axis.

    The distance from $\gamma(s)$ to the $e_3$ axis is given by $\gamma_1(s)$ since $\gamma_2(s)=0$ and the arc length measure is $\sqrt{\gamma_1'(s)^2+\gamma_3'(s)^2}\ ds$. Next we can see that $\nu(s,\theta)$ as defined is obtained by applying a rotation to the normal for $\gamma$. Thus one can easily obtain \eqref{eqn:rot} from \answ{\eqref{def:limit}, using the} change of variables, $\omega\mapsto(s,\theta)$.

    Finally, to show that $E_0(\M;n_\infty)=E_0(\M;R_\phi n_\infty)$ we use the fact that
    \begin{equation*}
        \nu(s,\theta)\cdot R_\phi n_\infty=R_\phi^{-1}\nu(s,\theta)\cdot n_\infty=\nu(s,\theta-\phi)\cdot n_\infty
        \, .
    \end{equation*}
    From here we take advantage of the periodicity of $R_{\theta-\phi}$ and a change of variables $\widehat{\theta}=\theta-\phi$,
    \begin{align*}
        E_0(\M;R_\phi n_\infty)&=\int_0^1\int_0^{2\pi}\ftf\Big(1-\sqrt{1-(\nu(s,\theta-\phi)\cdot n_\infty)^2}\Big)\gamma_1(s)\sqrt{\gamma_1'(s)^2+\gamma'_2(s)^2}\, d\theta ds\\
        &=\int_0^1\int_0^{2\pi}\ftf\Big(1-\sqrt{1-(\nu(s,\widehat\theta)\cdot n_\infty)^2}\Big)\gamma_1(s)\sqrt{\gamma_1'(s)^2+\gamma'_2(s)^2}\, d\widehat\theta ds\\
        &=E_0(\M;n_\infty)
        \, .
    \end{align*}
\end{proof}

We next apply Proposition \ref{prop:sym} to determine the optimal orientation for some specific shapes $\M$.

\bigskip

\noindent\underline{\textit{The Spherocylindrical Particle}}: The first object that we wish to study is the \emph{spherocylindrical} particle which is a cylinder with two half spheres attached on the two openings of the cylinder and is denoted $\M_S$.
\begin{figure}[H]
    \centering
    \includegraphics[scale=0.33]{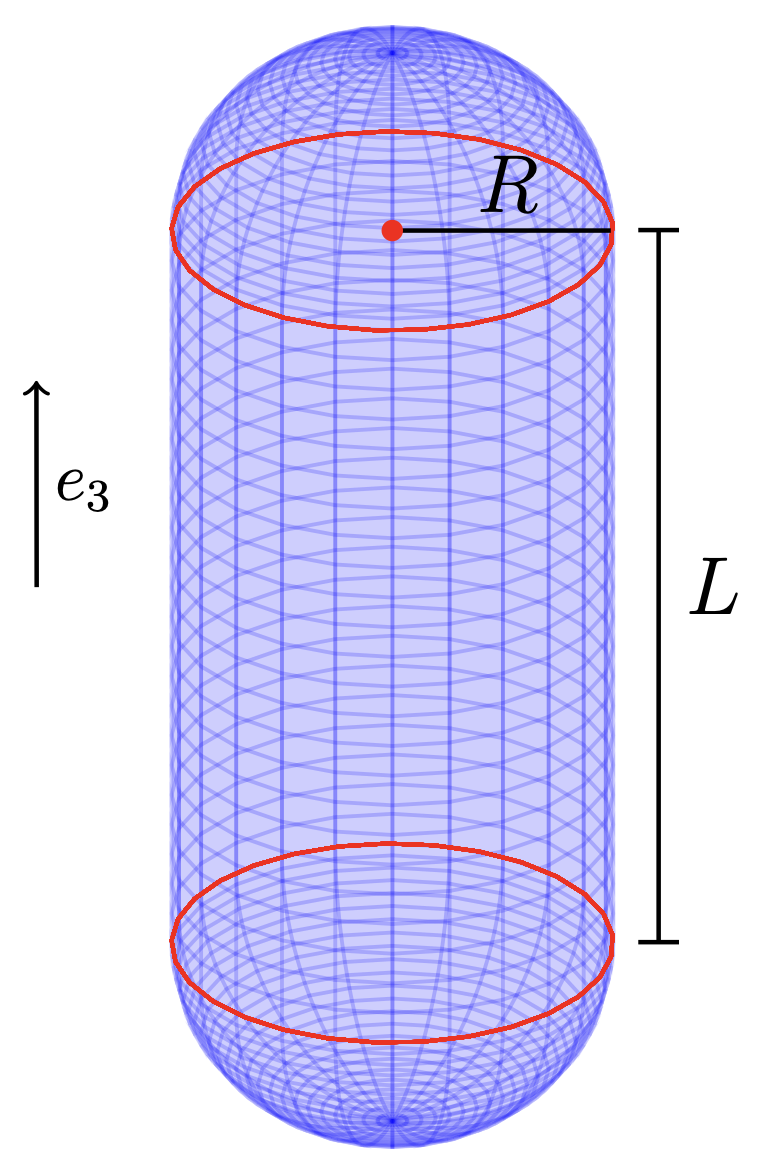}
    \caption{This figure shows the optimal orientation of a spherocylindrical particle with radius $R$ and cylinder length $L$.}
    \label{fig:capsule}
\end{figure}
The particle has radius $R$ and height $L+2R$ as observed in Figure \ref{fig:capsule} and one can also see that by constructing it in this way, the manifold is $C^\answ{\infty}$ \answ{away from the circles where the cylinder meets the half spheres}. We decompose $\M_S$ into three sets, the upper and lower half spheres of radius $R$, denoted by $\S^2_{R,+}$, $\S^2_{R,-}$ and the cylinder of radius $R$ and height $L$ denoted by $C_{R,L}$. 
So the corresponding energy is
\begin{equation*}
    E_0(\M_S;n_\infty)=\int_{\S^2_R}\ftf\Big(1-\sqrt{1-(\nu\cdot n_\infty)^2}\Big)\, d\H^2+\int_{C_{R,L}}\ftf\Big(1-\sqrt{1-(\nu\cdot n_\infty)^2}\Big)\, d\H^2
    \, .
\end{equation*}
Using the rotational symmetry of the sphere, $E_0(\S^2_R;n)=E_0(\S^2_R;m)$ for any $n,m\in\S^2$, so minimizing $E_0(\M_S;n_\infty)$ is equivalent to minimizing $E_0(C_{R,L};n_\infty)$. 
It is clear that if $n_\infty=\pm e_3$, then $\nu\cdot e_3=0$, so $n_\infty=\pm e_3$ is minimizing as $E_0(C_{R,L};\pm e_3)=0$. It suffices to show that this is the only minimum. Using the rotational symmetry of $C_{R,L}$ and that $\nu(s,\theta)=(\cos\theta,\sin\theta,0)$ we can see that
\begin{equation*}
    E_0(C_{R,L};n_\infty)=\int_{C_{R,L}}\ftf\Big(1-\sqrt{1-(\nu\cdot n_\infty)^2}\Big)\, d\H^2=\ftf RL\int_0^{2\pi}1-\sqrt{1-n_1^2\cos^2\theta}\, d\theta
    \, ,
\end{equation*}
for $n_\infty=(n_1,0,n_3)$. If $n_\infty\neq\pm e_3$, then $n_1\neq 0$ and $n_1^2>0$. Thus, on the interval $\theta\in[0,\pi/4]$, we have $n_1^2\cos^2\theta\geq \frac{n_1^2}{2}>0$. Therefore, $1-\sqrt{1-n_1^2\cos^2\theta}>0$ for $\theta\in[0,\pi/4]$. From this we can see that
\begin{equation*}
    E_0(\M_S;n_\infty) 
    \ > \ 
    E_0(\S^2_R;n_\infty) 
    \ = \
    E_0(\M_S;e_3)
    \ = \
    2\ftf\Big(2-\frac{\pi}{2}\Big)\pi R^2
    \, ,
\end{equation*}
so we conclude that $n_\infty=\pm e_3$ is the unique minimizer up to sign for $E_0(\M_S;\cdot)$. Using the convention that the complete elliptic integral of the second kind is given by
\begin{equation*}
    E(m)=\int_0^{\pi/2}\sqrt{1\answ{-}m\sin^2\theta}\, d\theta
    \, ,
\end{equation*}
we can write the energy $E_0(\M_S,n_\infty)$ as
\begin{equation*}
    E_0(\M_S,n_\infty)=2\ftf\Big(2-\frac{\pi}{2}\Big)\pi R^2+\ftf RL\left(2\pi-4\sqrt{1-n_1^2}E\left(\answ{\frac{1}{1-n_1^2}-1}\right)\right)
    \, .
\end{equation*}

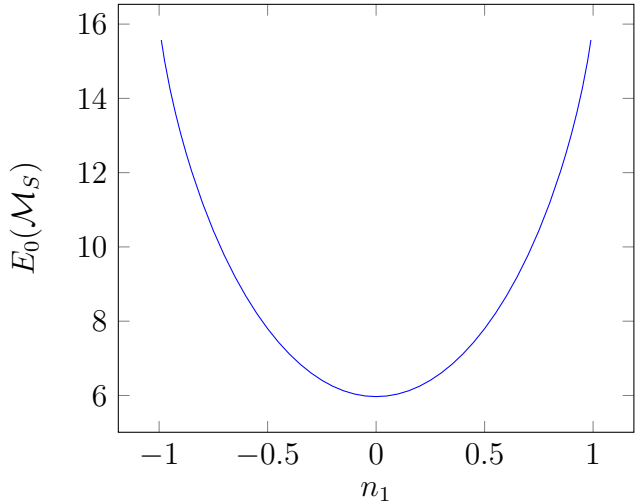
\begin{figure}[hbt!]
    \centering
    \begin{tikzpicture}
        \begin{axis}[xlabel=$n_1$, ylabel=$E_0(\M_S)$]
        \addplot[domain=-1:1,color=blue] coordinates {
        (-0.99,5.968925+9.60282)(-0.975,5.968925+9.0470157)(-0.95,5.968925+8.2881565)(-0.925,5.968925+7.6413149)(-0.9,5.968925+7.0668152)(-0.875,5.968925+6.5459916)(-0.85,5.968925+6.0679461)(-0.8,5.968925+5.2137358)(-0.75,5.968925+4.4678824)(-0.7,5.968925+3.8093796)(-0.65,5.968925+3.224605)(-0.6,5.968925+2.7040743)(-0.55,5.968925+2.2408598)(-0.5,5.968925+1.829728)(-0.45,5.968925+1.466629)(-0.4,5.968925+1.1483766)(-0.35,5.968925+0.872438)(-0.3,5.968925+0.63679118)(-0.25,5.968925+0.43982321)(-0.2,5.968925+0.2802611)(-0.15,5.968925+0.1571199)(-0.1,5.968925+0.0696658)(-0.05,5.968925+0.01739)
        (0,5.968925+0)(0.05,5.968925+0.01739)(0.1,5.968925+0.0696658)(0.15,5.968925+0.1571199)(0.2,5.968925+0.2802611)(0.25,5.968925+0.43982321)(0.3,5.968925+0.63679118)(0.35,5.968925+0.872438)(0.4,5.968925+1.1483766)(0.45,5.968925+1.466629)(0.5,5.968925+1.829728)(0.55,5.968925+2.2408598)(0.6,5.968925+2.7040743)(0.65,5.968925+3.224605)(0.7,5.968925+3.8093796)(0.75,5.968925+4.4678824)(0.8,5.968925+5.2137358)(0.85,5.968925+6.0679461)(0.875,5.968925+6.5459916)(0.9,5.968925+7.0668152)(0.925,5.968925+7.6413149)(0.95,5.968925+8.2881565)(0.975,5.968925+9.0470157)(0.99,5.968925+9.60282)
        };
        \end{axis}
    \end{tikzpicture}
    \caption{Graph of $E_0(\M_S;n_\infty)$ with constants $R=1$ and $L=2$ for $n_\infty=(n_1,0,n_3)$ where $n_1\in[-1,1]$ and $n_3=\sqrt{1-n_1^2}$ .}
    \label{fig:capsule-graph}
\end{figure}

The spherocylindrical particle represents a simple example of a non-symmetric colloidal particle and has been investigated in \cite{HGGAdP} and \cite{Hung} with normal anchoring condition.  
Elongated ellipsoidal shapes have been studied in \cite{TMMML}, for normal as well as tangential anchoring. 
The results in \cite[Fig. 4 and Fig. 6A]{TMMML} indicate, as our findings in Figure \ref{fig:capsule-graph}, that the preferred orientation is when the axis of symmetry of the elongated particle is parallel to the alignment at infinity. 
In the normal anchoring case, the preferred orientation is orthogonal to the director field \cite{HGGAdP}.

\bigskip

\noindent\underline{\textit{The Torus}}: We consider a torus with $R>r>0$ where $R$ and $r$ are as pictured in Figure \ref{fig:torus}.
\begin{figure}[hbt!]
    \centering
    \includegraphics[scale=0.42]{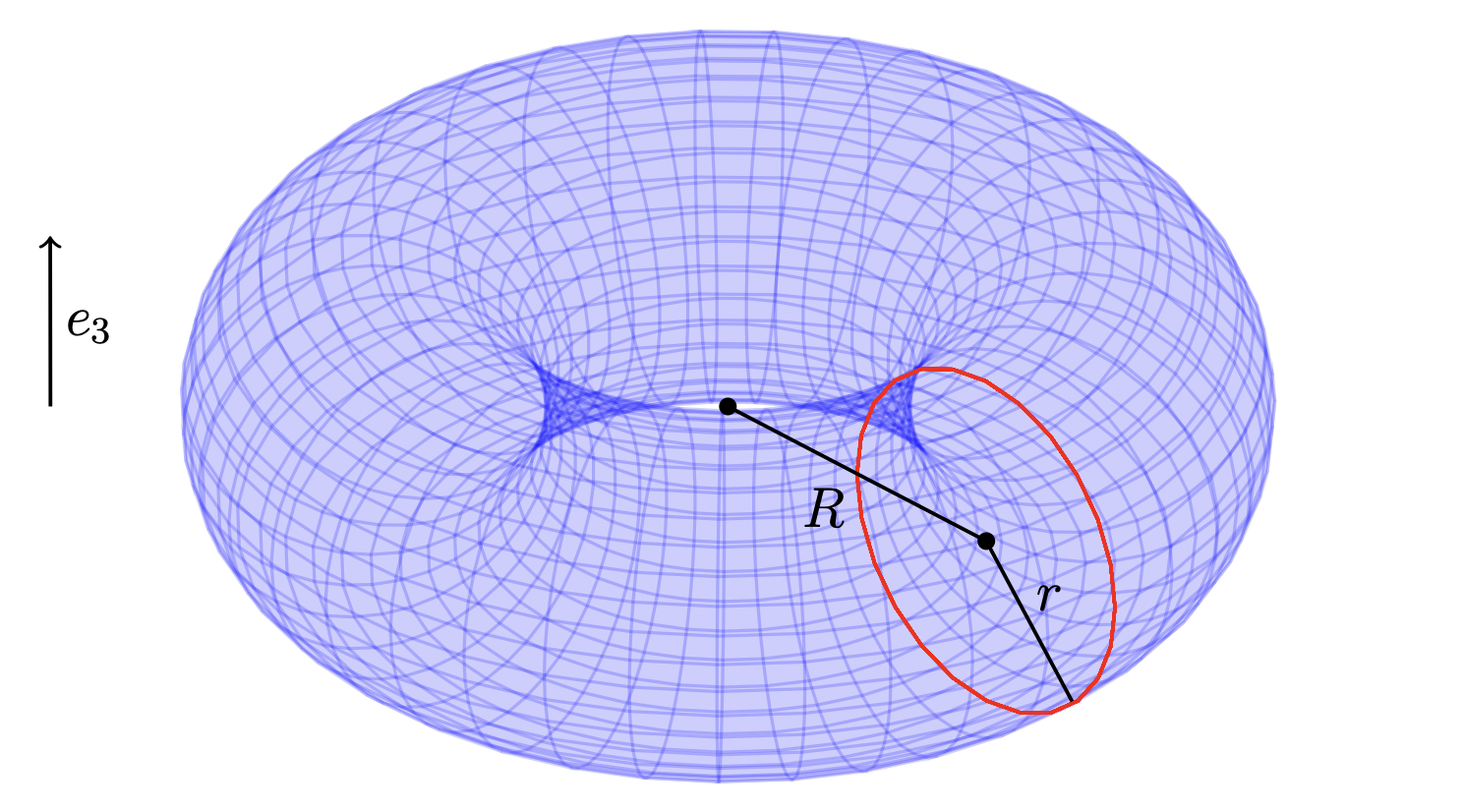}
    \caption{Torus of major radius $R>0$ and minor radius $r>0$. The globally energy minimizing alignment at infinity is \emph{orthogonal} to $e_3$.}
    \label{fig:torus}
\end{figure}

Applying the formula from Proposition \ref{prop:sym}, we obtain the energy
\begin{equation*}
    E_0(T^2;n_\infty)=2\pi r\ftf\int_0^1\int_0^{2\pi}(R-r\sin(2\pi s))\Big(1+\sqrt{1-(\nu(s,\theta)\cdot n_\infty)^2}\Big)\, d\theta ds
    \, ,
\end{equation*}
where $\nu(s,\theta)=(\cos\theta\sin(2\pi s),\sin\theta\sin(2\pi s),\cos(2\pi s))$. Using the rotation invariance of the energy as well as the identity $E_0(\M;n_\infty)=E_0(\M,-n_\infty)$, we can restrict $n_\infty=(n_1,0,n_3)$ where $n_1\geq 0$ and $-1\leq n_3\leq 1$, so $n_1=\sqrt{1-n_3^2}$. After a change of variables $\phi=2\pi s$, this gives the energy,
\begin{equation*}
    E_0(T^2;n_\infty)=r\ftf\int_0^{2\pi}\int_0^{2\pi}(R+r\sin\phi)\Big(1-\sqrt{1-(\sqrt{1-n_3^2}\cos\theta\sin\phi+n_3\cos\phi)^2}\Big)\, d\theta d\phi
    \, .
\end{equation*}
We can plot this energy as a function of $n_3$ on $[-1,1]$ to see that the minimum occurs precisely when $n_3=0$. This corresponds to a magnetic field which points in the plane orthogonal to $e_3$.

In \cite{Liu2013} the authors consider handlebodies with different genus $g\geq 1$ and tangential surface anchoring.
For the case of the torus ($g=1$) they conclude that the globally energy minimizing configuration is given when the normal vector of the plane of symmetry of the torus is orthogonal to the alignment at infinity \cite[Fig. 1 and 2]{Liu2013}. 
This is consistent with our findings, see also Figure \ref{fig:torus-graph}.

\begin{figure}[hbt!]
    \centering
    \begin{tikzpicture}
        \begin{axis}[xlabel=$n_3$, ylabel=$E_0(T^2)$]
        \addplot[domain=-1:1,color=blue] coordinates {
        (-1,63.504403)(-0.99,61.815481)(-0.98,60.511893)(-0.975,59.918862)(-0.95,57.297653)(-0.925,55.042305)(-0.9,53.026138)(-0.875,51.189197)(-0.85,49.496297)(-0.825,47.924279)(-0.8,46.456751)(-0.75,43.789151)(-0.7,41.424505)(-0.65,39.317792)(-0.6,37.437495)(-0.55,35.760484)(-0.5,34.269213)(-0.45,32.950054)(-0.4,31.792253)(-0.35,30.787237)(-0.3,29.928142)(-0.25,29.209488)(-0.2,28.626937)(-0.15,28.177129)(-0.1,27.857561)(-0.05,27.666499)(0,27.602923)(0.05,27.666499)(0.1,27.857561)(0.15,28.177129)(0.2,28.626937)(0.25,29.209488)(0.3,29.928142)(0.35,30.787237)(0.4,31.792253)(0.45,32.950054)(0.5,34.269213)(0.55,35.760484)(0.6,37.437495)(0.65,39.317792)(0.7,41.424505)(0.75,43.789151)(0.8,46.456751)(0.825,47.924279)(0.85,49.496297)(0.875,51.189197)(0.9,53.026138)(0.925,55.042305)(0.95,57.297653)(0.975,59.918862)(0.98,60.511893)(0.99,61.815481)(1,63.504403)
        };
        \end{axis}
    \end{tikzpicture}
    \caption{Graph of $E_0(T^2;n_\infty)$ for $n_\infty=(n_1,0,n_3)$ with $n_3\in[-1,1]$, $n_1=\sqrt{1-n_3^2}$ and $R=2$, $r=1$.}
    \label{fig:torus-graph}
\end{figure}
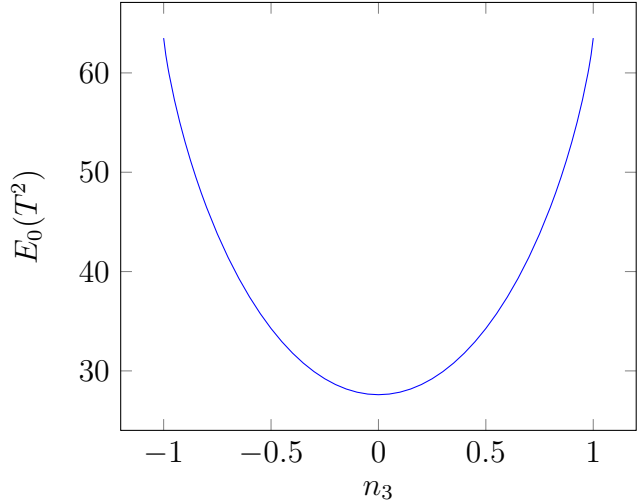

\bigskip
\newpage
\noindent\underline{\textit{The Cube}}: The last object that we analyze is a cube. 
Although the cube does not have an axis of rotational symmetry, its faces have a constant normal, so computing $E_0$ is straightforward. 
However, a cube $\M=\partial([-R/2,R/2]^3)$ is not $C^{\answ{2},1}$ so we will want to consider a smooth approximation of the cube $\M^\e$ and we will show in Proposition \ref{prop:approx_C11} below that the optimal orientation exhibits a continuous dependence on the approximation.
We first compute $E_0$ for $\M$ explicitly.
\begin{equation*}
    E_0(\M;n_\infty)=2R^2\ftf\Big(3-\sqrt{1-n_1^2}-\sqrt{1-n_2^2}-\sqrt{1-n_3^2}\Big)
    \, ,
\end{equation*}
but since $1-n_3^2=n_1^2+n_2^2$, we can write the energy in terms of only $n_1$ and $n_2$. We then optimize over the set
\begin{equation*}
    \{(n_1,n_2):n_1^2+n_2^2\leq 1\}
    \, ,
\end{equation*}
and from this we have eight global minima given by
\begin{equation*}
    n_\infty=\Big(\pm\frac{1}{\sqrt{3}},\pm\frac{1}{\sqrt{3}},\pm\frac{1}{\sqrt{3}}\Big)
\end{equation*}
taking every possible combination of signs. We confirm these results with the following heat map, see Figure \ref{fig:optimal_cube_n1_n2}.

In the case of normal anchoring, the cube and approximations thereof have been studied for example in \cite{BGL} and \cite{SPD}.
To the best of our knowledge, the case of tangential anchoring has not been considered numerically or experimentally. 
The optimal orientation found in \cite[Section 4]{BGL} for normal anchoring corresponds to the saddle points in the case of tangential anchoring, as seen in Figure \ref{fig:optimal_cube_n1_n2}.

\begin{figure}[H]
\begin{center}
\includegraphics[scale=0.9]{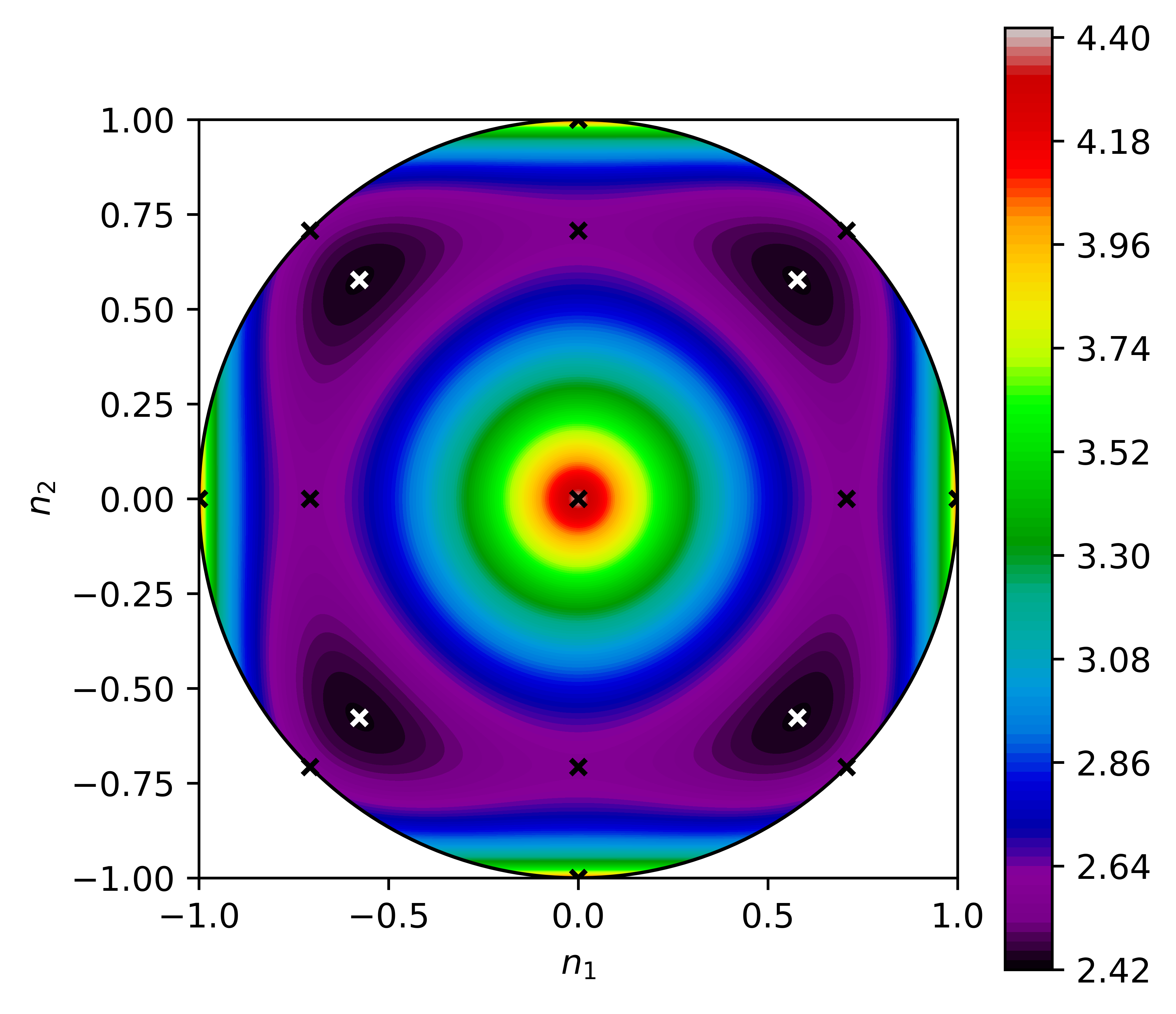} 
\caption{Plot of the minimal energy of the cube depending on $n_1$ and $n_2$. The optimal orientations $n_\infty=(\pm\frac{1}{\sqrt{3}},\pm\frac{1}{\sqrt{3}},\pm\frac{1}{\sqrt{3}})$ are indicated by white crosses. Other critical points of the energy are indicated with black crosses. They are: $n_\infty=\pm\ee_j$ for $j\in\{1,2,3\}$ (global maxima, four sides of the cube parallel to $n_\infty$, two side perpendicular) and $n_\infty=(\pm\frac{1}{\sqrt{2}},\pm\frac{1}{\sqrt{2}},0)$ or $n_\infty=(\pm\frac{1}{\sqrt{2}},0,\pm\frac{1}{\sqrt{2}})$ or $n_\infty=(0,\pm\frac{1}{\sqrt{2}},\pm\frac{1}{\sqrt{2}})$ (saddle points, two sides of the cube parallel to $n_\infty$, the other four in an $45^\circ$ angle).}
\label{fig:optimal_cube_n1_n2}
\end{center}
\end{figure}

Since the limiting energy $E_0$ has not been shown to hold for manifolds which are not $C^{\answ{2},1}$ we would like to be able to understand the optimal orientation for $C^{\answ{2},1}$ approximations $\M^\e$ and we summarize these results in the following proposition.
\begin{proposition}\label{prop:approx_C11}
    Let $\M$ be \answ{$C^{2,1}$ away from finitely many points and curves of finite length. Let} $(\M^\e)$ be a sequence of $C^{\answ{2},1}$ manifolds such that
    \begin{equation*}
        \H^2(\M^\e\Delta\M)<\e/\ftf
        \, \blue{,}
    \end{equation*}
    \blue{and such that there exists an open subset $U\subset\M\cap\M^\e$ such that $\H^2((\M\cap\M^\e)\setminus U)=0$ for all $\e>0$. }Define $A\subset\S^2$ to be the set of global minima of $E_0(\M;\cdot)$ and let
    \begin{equation*}
        B_\d=\{n\in\S^2:\dist_{\S^2}(n,A)<\d\}
        \, .
    \end{equation*}
    Then, for any $\d>0$, there exists $\e>0$ sufficiently small such that if $n^\e\in\S^2$ is a global minimum of $E_0(\M^\e;\cdot)$, it must be that $n^\e\in B_\d$.
\end{proposition}

\newpage

\begin{proof}
    We first remark that $E_0(\M;\cdot)$ is continuous \answ{and $\S^2$ is compact, so $A$ is non-empty and} $E_0(\M;n)>E_0(\M;m)$ for $m\in A$ and $n\in\S^2\setminus A$. \blue{Let $\nu$ and $\nu_\e$ be the unit outward normals to $\M$ and $\M^\e$ respectively, then $\nu=\nu_\e$ on $U$. As a result,} for any $n\in\S^2$, \blue{we have}
    \begin{align*}
        |E_0(\M^\e;n)&-E_0(\M;n)|\\
        &=\left|\int_{\M^\e\setminus\M}\ftf\Big(1-\sqrt{1-(\nu_\blue{\e}\cdot n)^2}\Big)\, d\H^2-\int_{\M\setminus\M^\e}\ftf\Big(1-\sqrt{1-(\nu\cdot n)^2}\Big)\, d\H^2\right|\\
        &\leq\ftf\Big(\H^2(\M^\e\setminus\M)+\H^2(\M\setminus\M^\e)\Big)=\ftf\H^2(\M^\e\Delta\M)<\e
        \, ,
    \end{align*}
    so $E_0(\M^\e;\cdot)\to E_0(\M;\cdot)$ uniformly on $\S^2$. Let $K\subset \S^2$ be any compact set, then if follows from uniform convergence that
    \begin{equation}\label{cpt:min}
        \Big|\min_{n\in K} E_0(\M^\e;n)-\min_{m\in K} E_0(\M;m)\Big|<\e
        \, .
    \end{equation}
    Let $\d>0$ be fixed, then $\S^2\setminus B_\d$ is a compact set and
    \begin{equation*}
        m_\d:=\min\{E_0(\M;n):n\in \S^2\setminus B_\d\}>\min\{E_0(\M;n):n\in\S^2\}
        \, .
    \end{equation*}
    Now setting $\mu_\d := m_\d-\min\{E_0(\M;n):n\in\S^2\}>0$ and choosing $0<\e<\mu_\d/2$, it follows \blue{from} \eqref{cpt:min} that
    \begin{equation*}
        \min\{E_0(\M^\e;n):n\in\S^2\setminus B_\d\}>\frac{\mu_\d}{2}+\min\{E_0(\M;n):n\in \S^2\}
        \, .
    \end{equation*}
    But we can also see by the same reasoning that
    \begin{equation*}
        \min\{E_0(\M^\e;n):n\in B_\d\}<\frac{\mu_\d}{2}+\min\{E_0(\M;n):n\in\S^2\}
        \, ,
    \end{equation*}
    thus if $n^\e$ is a global minimum of $E_0(\M^\e;\cdot)$, then $n^\e\in B_\d$.
\end{proof}

\paragraph{Disclosure statement.}
The authors report there are no competing interests to declare.

\paragraph{Data availability statement.} No data set associated with the paper.

\bibliographystyle{acm}

\begin{thebibliography}{10}

\bibitem{abc}
{\sc S.~Alama, L.~Bronsard, A.~Colinet.}
\newblock $\Gamma$-Convergence of the {G}inzburg-{L}andau Functional with tangential boundary conditions
\newblock {\em arXiv:2305.05815} (2023).

\bibitem{weakanchor}
{\sc S.~Alama, L.~Bronsard, B.~ Galv\~ao Sousa.}
\newblock Weak anchoring for a two-dimensional liquid crystal.
\newblock {\em Nonlinear Anal. 119\/} (2015), 74--97.

\bibitem{ABG20}
{\sc S.~Alama, L.~Bronsard, D.~Golovaty.}
\newblock Thin film liquid crystals with oblique anchoring and boojums.
\newblock In {\em Annales de l'Institut Henri Poincar{\'e} C, Analyse non-lin{\'e}aire}. Elsevier, 2020.

\bibitem{ABGL}
{\sc S.~Alama, L.~Bronsard, D.~Golovaty, X.~Lamy.}
\newblock Saturn ring defect around a spherical particle immersed in nematic liquid crystal.
\newblock {\em Calc. Var. Partial Differ. Equ.}. 60, 225 (2021).

\bibitem{colloid}
{\sc S.~Alama, L.~Bronsard, X.~Lamy.}
\newblock Minimizers of the {L}andau--de {G}ennes energy around a spherical  colloid particle.
\newblock {\em Arch. Ration. Mech. Anal. 222}, 1 (2016), 427--450.

\bibitem{colloidphys}
{\sc S.~Alama, L.~Bronsard, X.~Lamy.}
\newblock Analytical description of the Saturn-ring defect in nematic colloids.
\newblock {\em Phys. Rev. E 93\/} (Jan 2016), 012705.

\bibitem{abl}
{\sc S.~Alama, L.~Bronsard, X.~Lamy.}
\newblock Spherical particle in nematic liquid crystal under an external field:  The {S}aturn ring regime.
\newblock {\em J. Nonlinear Sci.}, 28(4):1443--1465, 2018.

\bibitem{ABv}
{\sc S.~Alama, L.~Bronsard, L.~van Brussel.}
\newblock On minimizers of the 2D {G}inzburg–{L}andau energy with tangential anchoring.
\newblock {\em Nonlinear Anal. 232}, (2023) Paper No 113367, 24 pp. 

\bibitem{ACS2021}
{\sc F.~Alouges, A.~Chambolle, D.~Stantejsky.}
\newblock The {S}aturn ring effect in nematic liquid crystals with external field: Effective energy and hysteresis.
\newblock {\em Arch. Ration. Mech. Anal.}, (2021).

\bibitem{ACS2024}
{\sc F.~Alouges, A.~Chambolle, D.~Stantejsky.}
\newblock Convergence to line and surface energies in nematic liquid crystal colloids with external magnetic field.
\newblock {\em Calc. Var. Partial Differ. Equ.}, (2024).

\bibitem{BaZar}
{\sc J.~Ball, A.~Zarnescu.} 
\newblock Orientability and Energy Minimization in Liquid Crystal Models.
\newblock {\em Arch Rational Mech Anal} 202, 493–-535 (2011)

\bibitem{baumanparkphillips12}
{\sc P.~Bauman, J.~Park, D.~Phillips.}
\newblock Analysis of nematic liquid crystals with disclination lines.
\newblock {\em Arch. Ration. Mech. Anal. 205}, 3 (2012), 795--826.

\bibitem{BGL}
{\sc D.~Beller, M.~Gharbi, I.~Liu.}
\newblock Shape-controlled orientation and assembly of colloids with sharp edges in nematic liquid crystals.
\newblock {\em Soft Matter}, 11(6), 1078–-1086, (2015).

\bibitem{BeChi}
{\sc F.~Bethuel, D.~Chiron.}
\newblock Some questions related to the lifting problem in Sobolev spaces.
\newblock {\em Perspectives in nonlinear partial differential equations}, {\em Contemporary Mathematics, 446, American Mathematical Society, Providence, RI}, (2007) .

\bibitem{bcs}
{\sc  L.~Bronsard, A.~Colinet, D.~Stantejsky.}
\newblock A priori $L^\infty-$bound for {G}inzburg-{L}andau energy minimizers with divergence penalization.
\newblock {\em arXiv:2403.09949} (2024).

\bibitem{bls}
{\sc L.~Bronsard, D.~Louizos, D.~Stantejsky.}
\newblock Spherical Particle in Nematic Liquid Crystal with a Magnetic Field and Planar Anchoring. \textit{J. Nonlinear Sci.}, 35 (2024).

\bibitem{canevari2d}
{\sc G.~Canevari.}
\newblock Biaxiality in the asymptotic analysis of a 2{D} {L}andau--de {G}ennes  model for liquid crystals.
\newblock {\em ESAIM Control Optim. Calc. Var. 21}, 1 (2015), 101--137.

\bibitem{canevari3d}
{\sc G.~Canevari.}
\newblock Line defects in the small elastic constant limit of a three-dimensional {L}andau-de {G}ennes model.
\newblock {\em Arch Rational Mech Anal} 223, 591–676 (2017). 

\bibitem{CO1}
{\sc G.~Canevari, G.~Orlandi.} 
\newblock Topological singular set of vector-valued maps, I: applications to manifold-constrained Sobolev and BV spaces.  
\newblock {\em Calc. Var.} 58, 72 (2019).

\bibitem{CO2}
{\sc G.~Canevari, G.~Orlandi.} 
\newblock Topological Singular Set of Vector-Valued Maps, II: 
$\answ{\Gamma}$-convergence for Ginzburg–Landau type functionals. 
\newblock {\em Arch Rational Mech Anal} 241, 1065-–1135  (2021). 

\bibitem{singperturb}
{\sc A.~Contreras, X.~Lamy, R.~Rodiac.}
\newblock On the convergence of minimizers of singular perturbation  functionals.
\newblock {\em Indiana Univ. Math. J. 67}, 4 (2018), 1665--1682.

\bibitem{biaxialescape}
{\sc A.~Contreras, X.~Lamy.}
\newblock Biaxial escape in nematics at low temperature.
\newblock {\em J. Funct. Anal. 272}, 10 (2017), 3987--3997.

\bibitem{DayZar}
{\sc S.~Day, A.~Zarnescu.}
\newblock Sphere-valued harmonic maps with surface
energy and the K13 problem.
\newblock {\em Adv. Calc. Var.},  2017

\bibitem{difrattaetal16}
{\sc G.~Di~Fratta, J.~M.~Robbins, V.~Slastikov, A.~Zarnescu.}
\newblock Half-{I}nteger {P}oint {D}efects in the {Q}-{T}ensor {T}heory of {N}ematic {L}iquid {C}rystals.
\newblock {\em J. Nonlinear Sci. 26}, 1 (2016), 121--140.

\bibitem{DiMiPi}
{\sc F.~Dipasquale, V.~Millot, A.~Pisante.}
\newblock Torus-like Solutions for the {L}andau-de {G}ennes Model. Part I: The Lyuksyutov Regime.
\newblock {\em Arch. Rational Mech. Anal} 239(6), 1 (2021), 559--678.

\bibitem{FedererGMT}
{\sc H. Federer.}
\newblock{Geometric Measure Theory.}
\newblock{\em Springer, Berlin} (1996).

\bibitem{Fleming}
{\sc W.~Fleming.} 
\newblock Flat chains over a coefficient group. 
\newblock {\em Trans. Amer. Math. Soc. } 121 (1966), 160-186.

\bibitem{fukudaetal04}
{\sc J.~Fukuda, H.~Stark, M.~Yoneya, H.~Yokoyama.}
\newblock Dynamics of a nematic liquid crystal around a spherical particle.
\newblock {\em J. Physics: Condens. Matter 16}, 19 (2004), S1957.

\bibitem{fukudayokoyama06}
{\sc J.~Fukuda, H.~Yokoyama.}
\newblock Stability of the director profile of a nematic liquid crystal around  a spherical particle under an external field.
\newblock {\em Eur. Phys. J. E 21}, 4 (2006), 341--347.

\bibitem{Gartland} 
{\sc E.~Gartland Jr.}
\newblock Scalings and limits of {L}andau-de {G}ennes models for liquid crystals: A comment on some recent analytical papers
\newblock {\em Mathematical Modelling and Analysis } 23(3), 414--432, (2018)

\bibitem{GWZZ}
{\sc Z.~Geng, W.~Wang, P.~Zhang, Z.~Zhang.} 
\newblock Stability of half-degree point defect profiles for 2-D nematic liquid crystal.
\newblock {\em Discrete and Continuous Dynamical Systems} 37(12), 6227--6242, (2017)

\bibitem{golovatymontero14}
{\sc D.~Golovaty, J.~A.~Montero.}
\newblock On minimizers of a {L}andau--de {G}ennes energy functional on planar domains.
\newblock {\em Arch. Ration. Mech. Anal. 213}, 2 (2014), 447--490.

\bibitem{hirsch}
{\sc M. W. Hirsch.}
\newblock{Differential Topology.}
\newblock{\em Springer, New York} (1976).

\bibitem{Hung}
{\sc F.~Hung.}
\newblock Quadrupolar particles in a nematic liquid crystal: Effects of particle size and shape.
\newblock{\em Phys. Rev. E} 79, (2009).

\bibitem{HGGAdP}
{\sc F.~Hung, O.~Guzm\'an, B.~Gettelfinger, N.~Abbott, J.~de Pablo.}
\newblock Anisotropic nanoparticles immersed in a nematic liquid crystal: Defect structures and potentials of mean force.
\newblock{\em Phys. Rev. E} 74, (2006).

\bibitem{INSZuniqhedgehog}
{\sc R.~Ignat, L.~Nguyen, V.~Slastikov, A.~Zarnescu.}
\newblock Uniqueness results for an {ODE} related to a generalized {G}inzburg-{L}andau model for liquid crystals.
\newblock {\em SIAM J. Math. Anal. 46}, 5 (2014), 3390--3425.

\bibitem{INSZstabhedgehog}
{\sc R.~Ignat, L.~Nguyen, V.~Slastikov, A.~Zarnescu.}
\newblock Stability of the melting hedgehog in the {L}andau--de {G}ennes theory of nematic liquid crystals.
\newblock {\em Arch. Ration. Mech. Anal. 215}, 2 (2015), 633--673.

\bibitem{INSZstab2d}
{\sc R.~Ignat, L.~Nguyen, V.~Slastikov, A.~Zarnescu.}
\newblock Stability of point defects of degree {$\pm\frac{1}{2}$} in a two-dimensional nematic liquid crystal model.
\newblock {\em Calc. Var. Partial Differential Equations 55}, 5 (2016), Paper No. 119, 33.

\bibitem{INSZinstab2d}
{\sc R.~Ignat, L.~Nguyen, V.~Slastikov, A.~Zarnescu.}
\newblock Instability of point defects in a two-dimensional nematic liquid crystal model.
\newblock {\em Ann. Inst. H. Poincar\'e Anal. Non Lin\'eaire 33}, 4 (2016), 1131--1152.

\bibitem{KraVi2001} 
{\sc S.~Kralj, E.~Virga.}
\newblock Universal fine structure of nematic hedgehogs,
\newblock {\em J. Phys. A: Math. Gen. 34,} 829, 2001.

\bibitem{biLip}
{\sc Kunziger, M., Steinbauer, R., and Stojkovi\'{c}, M.}
\newblock The exponential
map of a C1,1-metric,
\newblock {\em Differential Geom. Appl. 34,} 14–24, 2014.

\bibitem{lee}
{\sc J.~M.~Lee.}
\newblock{Introduction to Smooth Manifolds.}
\newblock{\em Springer, New York} (2012).

\bibitem{Liu2013} 
{\sc Q.~Liu, B.~Senyuk, M.~Tasinkevych, I.~Smalyukh.}
\newblock Nematic liquid crystal boojums with handles on colloidal handlebodies,
\newblock {\em Proceedings of the National Academy of Sciences} 110(23), 2013

\bibitem{Lyu1978} 
{\sc I.~F.~Lyuksyutov.}
\newblock Topological instability of singularities at small distances in nematics,
\newblock {\em Zh.Eksp.Teor.Fiz. 75} p.358-360 (1978), English translation in: {\em Sov.Phys.JETP 48}, p.178-179 (1978).

\bibitem{majumdarzarnescu10}
{\sc A.~Majumdar, A.~Zarnescu.}
\newblock Landau-de {G}ennes theory of nematic liquid crystals: The {O}seen--{F}rank limit and beyond.
\newblock {\em Arch. Ration. Mech. Anal. 196}, 1 (2010), 227--280.

\bibitem{MiVaS} 
{\sc P.~Mironescu, J.~Van~Schaftingen}
\newblock Lifting in compact covering spaces for fractional Sobolev mappings
\newblock {\em Anal. PDE} 14, 1851--1871, (2021)

\bibitem{NguZa} 
{\sc L.~Nguyen, A.~Zarnescu.}
\newblock Refined approximation for minimizers of a {L}andau-de {G}ennes energy functional
\newblock {\em Calc. Var. Partial Differ. Equ.} 47, 383--432, (2013)


\bibitem{Senyuk2021} 
{\sc B.~Senyuk, A.~Mozaffari, K.~Crust, R.~Zhang, J.~de~Pablo, I.~Smalyukh.}
\newblock Transformation between elastic dipoles, quadrupoles, octupoles, and hexadecapoles driven by surfactant self-assembly in nematic emulsion,
\newblock {\em Science Advances} 7, 2021

\bibitem{Shi2005} 
{\sc S.~Shiyanovskii, O.~Lavrentovich, T.~Schneider, T.~Ishikawa, I.~Smalyukh, C.~Woolverton, G.~Niehaus, K.~Doane.}
\newblock Lyotropic Chromonic Liquid Crystals for Biological Sensing Applications,
\newblock {\em Molecular Crystals and Liquid Crystals} 434, 2005

\bibitem{stark01}
{\sc H.~Stark.}
\newblock Physics of colloidal dispersions in nematic liquid crystals.
\newblock {\em Phys. Rep. 351}, 6 (2001), 387 -- 474.

\bibitem{St91}
{\sc P.~Sternberg.}
\newblock Vector-valued local minimizers of nonconvex variational problems.
\newblock {\em Rocky Mtn. J. Mathematics 21\/} (1991), 799--807.

\bibitem{SPD}
{\sc D.~Sudhakaran, R.~Pujala, S.~Dhara.} 
\newblock Orientation Dependent Interaction and Self‐Assembly of Cubic Magnetic Colloids in a Nematic Liquid Crystal. 
\newblock {\em Advanced Optical Materials} (2020).

\bibitem{TMMML}
{\sc M.~Tasinkevych, F.~Mondiot, O.~Mondain-Monval, O., J.C.~Loudet.} 
\newblock Dispersions of ellipsoidal particles in a nematic liquid crystal.
\newblock {\em  Soft Matter} 10(12), (2014).

\bibitem{Tas2012} 
{\sc M.~Tasinkevych, N.~Silvestre, M.~da~Gama.}
\newblock Liquid crystal boojum-colloids,
\newblock {\em New J. Phys.} 14, 2012

\bibitem{White}
{\sc B.~White.} 
\newblock Rectifiability of Flat Chains 
\newblock {\em Annals of Mathematics} 150, 165--184 (1999).

\end{thebibliography}

\end{document}